\documentclass[11pt]{article}
\usepackage{amsmath,amssymb,amsthm,eucal,manfnt,mathrsfs,tikz,colonequals}
\usepackage{color}
\usepackage[all]{xy}
\usepackage{authblk}

\title{Poincar\'e duality for Cuntz--Pimsner algebras of bimodules}

\author[1]{Adam Rennie}
\author[2]{David Robertson}
\author[1]{Aidan Sims}
\affil[1]{School of Mathematics and Applied Statistics\\
University of Wollongong\\
Northfields Ave 2522\\ Australia\\ \texttt{renniea@uow.edu.au,asims@uow.edu.au}}

\affil[2]{School of Mathematical and Physical Sciences\\ University of Newcastle\\
University Drive\\ Callaghan 2308\\ Australia\\ \texttt{dave84robertson@gmail.com}}
\advance\topmargin by -\headheight
\advance\topmargin by -\headsep
\evensidemargin=\oddsidemargin
\numberwithin{equation}{section} 

\theoremstyle{plain} 
\newtheorem{thm}{Theorem}[section]
\newtheorem{lemma}[thm]{Lemma}
\newtheorem{prop}[thm]{Proposition}
\newtheorem{corl}[thm]{Corollary}
\newtheorem{ass}{Assumption}

\theoremstyle{definition} 
\newtheorem{defn}[thm]{Definition}
\newtheorem{example}[thm]{Example}
\theoremstyle{remark} 
\newtheorem{rmk}[thm]{Remark}
\newtheorem{ntn}[thm]{Notation}

\DeclareMathOperator{\End}{End}   
\DeclareMathOperator{\ev}{ev}  
\DeclareMathOperator{\Id}{Id}     
\DeclareMathOperator{\linspan}{span} 

\newcommand{\op}{{\operatorname{op}}}       
\newcommand{\extcls}{[\operatorname{ext}]}
\newcommand{\extocls}{[\operatorname{ext}^\op]}
\newcommand{\extobarcls}{[\overline{\operatorname{ext}}^\op]}
\newcommand{\Aut}{\operatorname{Aut}}
\newcommand{\id}{\operatorname{id}}

\newcommand{\A}{\mathcal{A}}  
\newcommand{\B}{\mathcal{B}}  
\newcommand{\C}{\mathbb{C}}   
\newcommand{\D}{\mathcal{D}}  
\newcommand{\E}{\mathcal{E}}  
\renewcommand{\H}{\mathcal{H}}  
\newcommand{\K}{\mathcal{K}}
\newcommand{\N}{\mathbb{N}}   
\newcommand{\Oo}{\mathcal{O}}
\newcommand{\R}{\mathbb{R}}   
\newcommand{\V}{\mathcal{V}}   
\newcommand{\ox}{\otimes}     
\newcommand{\hox}{\hat\otimes} 
\renewcommand{\S}{\mathcal{S}} 
\newcommand{\Z}{\mathbb{Z}}   

\newcommand{\Fock}{\mathcal{F}}
\newcommand{\CPa}{\mathcal{O}}  
\newcommand{\Toep}{\mathcal{T}} 
\newcommand{\ol}[1]{\overline{#1}}
\newcommand{\Cliff}{\mathbb{C}{\rm liff}}  

\newcommand{\defeq}{\colonequals}

\begin{document}

\maketitle

\vspace{-2pc}

\begin{abstract}
We present a new approach to  Poincar\'e duality for Cuntz--Pimsner algebras. We provide  sufficient
conditions under which Poincar\'e self-duality for the coefficient algebra of a Hilbert
bimodule lifts to Poincar\'e self-duality for the associated Cuntz--Pimsner algebra.

With these conditions in hand, we can constructively produce fundamental classes in
$K$-theory for a wide range of examples. We can also produce $K$-homology fundamental
classes for the important examples of Cuntz--Krieger algebras (following
Kaminker--Putnam) and  crossed products of manifolds by isometries, and their
non-commutative analogues.
%
%
\end{abstract}

\tableofcontents

\parskip=5pt
\parindent=0pt

\addtocontents{toc}{\vspace{-1pc}}

\section{Introduction}
\label{sec:intro}

In this paper we explore a new approach to Poincar\'e duality for Cuntz--Pimsner
algebras. Our approach emphasises the interaction between the dynamics defined by a
bimodule and the topology of its coefficient algebra.

Our motivation was to understand existing proofs of Poincar\'e duality for $C^*$-algebras associated to dynamical
systems in a more geometric light: specifically, the proofs of Poincar\'e duality for
Cuntz--Krieger algebras \cite{KP}, $k$-graph algebras \cite{PZ} and Smale spaces
\cite{KPW}. Other examples of Poincar\'e duality, such as \cite{ConnesGrav} and
\cite{Heath} are somewhat different, but there is overlap in the algebras treated,
if not in the  techniques. Overviews of $C^*$-algebraic Poincar\'e
duality appear in \cite{BMRS,KS}.

Our approach has several elements. Given a suitable coefficient algebra $A$ satisfying
Poincar\'e self-duality (in the sense that $A$ is Poincar\'e dual to its opposite algebra $A^\op$), we
consider an $A$--$A$-correspondence $E$. Our aim is to lift the Poincar\'e self-duality
for $A$ to a duality for the Cuntz--Pimsner algebra $\CPa_E$.

In the first instance, based on Connes' work \cite{ConnesGrav}, we seek a Poincar\'e
duality between $\CPa_E$ and $\CPa_E^\op$. On the other hand, the results of Kaminker--Putnam suggest
that, when $E$ is a bi-Hilbertian bimodule, we might also (or instead) expect duality
between $\CPa_E$ and $\CPa_{E^\op}$. It turns out that when $E$ is an invertible bimodule
(that is, a self-Morita-equivalence bimodule), 
the algebras $\CPa_{E^\op}$ and $\CPa_E^\op$
coincide, so the question of which potential dual algebra to consider is moot. But no
such isomorphism exists in general, so we explore both possible dual algebras. To aid our
study of duality with $\CPa_E^\op$, we establish that $\CPa_E^\op$ is isomorphic to
$\CPa_{\ol{E}^\op}$, where $\ol{E}^\op$ is the dual module $\ol{E}$ regarded as a right
$A^\op$-module.

Our first main result provides checkable conditions, for both possible dual algebras, on
potential $K$-theory and $K$-homology fundamental classes that guarantee that they
implement a Poincar\'e duality for $\CPa_E$. The conditions involve the interaction of the dynamics
defined by $E$ with the fundamental classes witnessing the Poincar\'e self-duality of
$A$, and the Kasparov class of the defining short exact sequence for $\CPa_E$.

In fact we can obtain significant information from the existence of either ``half'' of a
Poincar\'e duality pair. We describe how either a $K$-theory or a $K$-homology
fundamental class, even in the absence of its counterpart, provides non-trivial
information: isomorphisms of $K$-theory groups for one algebra with $K$-homology groups
for the other.
This is important, because our next main result provides an explicit construction of a
$K$-theory fundamental class for both possible dual algebras under very mild hypotheses
on the bimodule $E$. In many cases we obtain explicit representatives of these classes
which can be compared directly with known examples.

Examples to which our methods for the $K$-theory fundamental class apply directly are: the
Wieler solenoids discussed in \cite{DGMW}; crossed products of 
compact spin$^c$ manifolds  by
isometries;  and more generally topological graphs over manifolds and crossed products by
injective endomorphisms. We also recover Kaminker and Putnam's $K$-theory fundamental class for Cuntz--Krieger
algebras and extend it to  more general graph algebras.

Establishing the existence of a $K$-homology fundamental class turns out to be more
challenging and we have not discovered a general method. We do show that Kaminker
and Putnam's $K$-homology fundamental 
class can be obtained via
our approach, and extend their construction to a
broader class of graph algebras. We also provide sufficient conditions for the
crossed product $C(M)\rtimes_\alpha\Z$ of a 
compact spin$^c$ manifold  by a
spin$^c$-structure-preserving isometry to satisfy Poincar\'e duality. This construction extends 
to $\theta$-deformations of manifolds.

Our formulation of Poincar\'e duality is not always
suitable for non-unital algebras. A compactly supported version of $K$-homology and the attendant
exact sequences etc are required, \cite{KasparovEqvar,RennieSmooth}.
This requires either $RKK$ \cite{KasparovEqvar}, 
for $C(X)$-algebras, or 
significant work to establish the required exact sequences for any proposed compactly supported $K$-homology.
For the most part we leave these issues to future work, focussing on the unital case.

The paper is organised as follows. The coefficient algebras, correspondences, bimodules
and Cuntz--Pimsner algebras we consider are discussed in Section~\ref{sec:bimods}. In
particular we tease out some of the relationships between the two possible algebras that
may play the role of Poincar\'e dual algebra to $\CPa_E$. The two contenders are the
opposite algebra $\CPa_E^\op$ and the Cuntz--Pimsner algebra $\CPa_{E^\op}$ of the
opposite module. Of course this choice influences what we mean by self-duality for
$\CPa_E$.

Our sufficient conditions for lifting Poincar\'e self-duality of the coefficient algebra
$A$  to the Cuntz--Pimsner algebra are described in Section~\ref{sec:funky-diagram}.
Again, we discuss criteria for both possible dual algebras. Section
\ref{subsec:K-theory-fun} covers the construction of the $K$-theory fundamental class for
$\CPa_E$. This uses mapping cone techniques, from \cite{AR, CPR1,Putnam}. We produce
explicit representatives of this class in our main examples. The construction of the
$K$-homology class for crossed products and for Cuntz--Krieger algebras is described in
Section~\ref{sec:K-fun}.

Our techniques apply to more general pairs of Poincar\'e dual algebras, as we describe 
in Appendix~A. Finally, in Appendix B we discuss the relationships between the
extension classes. These extension classes play a central role throughout the paper, and
the fine structure of our Poincar\'e duality classes (such as summability, real
structures etc) will depend on these relationships.

{\bf Acknowledgements} All authors were supported by the Australian Research Council. The
authors thank Kylie Fairhall for teaching them `the magic question', which was
instrumental in discovering the results in this paper. Numerous aspects of this work
overlap with projects with our collaborators Francesca Arici, Magnus Goffeng and Bram
Mesland, and we thank them for many conversations and lessons learned. We also wish to
thank Heath Emerson for discussions on sign conventions. This research was supported by
ARC Discovery grant DP120100507, and the MATRIX@Melbourne research program \emph{Refining
$C^*$-algebraic invariants for dynamics using KK-theory}, July 18--29 2016.

\section{Cuntz--Pimsner algebras and associated Kasparov classes}
\label{sec:bimods}

In all the following, we suppose that $A$ is a separable, 
$C^*$-algebra. Given a
right Hilbert $C^*$-$A$-module $E$ (written $E_A$ when we want to emphasise the
coefficient algebra), we denote the $C^*$-algebra of adjointable operators on $E$ by
$\End_A(E)$. For $e,f \in E$ we write $\Theta_{e,f} \in \End_A(E)$ for the rank-one
operator $\Theta_{e,f}(g) = e \cdot (f \mid g)$. We write $\End^0_A(E)$ for the closed
2-sided ideal
\[
    \End^{0}_A(E) := \overline{\linspan}\{\Theta_{e,f} : e,f \in E\} \subseteq \End_A(E)
\]
of generalised compact operators on $E$.

We denote by $\overline{E}$ a copy $\{\bar{e} : e \in E\}$ of $E$ as a set with
vector-space structure given by $\lambda\overline{e} + \overline{f} =
\overline{\bar{\lambda}e + f}$ for $\lambda\in\C$ and $\bar{e},\,\bar{f}\in
\overline{E}$. The vector space $\overline{E}$ is a left-Hilbert $A$-module with $a
\cdot\bar{e} = \overline{e \cdot a^*}$ and ${}_A(\bar{e} \mid \bar{f}) = (e \mid f)_A$.
We call $\overline{E}$ with this structure the \emph{conjugate module} of $E$. 

Given a $C^*$-algebra $A$, we write $\ell^2(A)$ for the standard $C^*$-module
\[
\ell^2(A) := \Big\{(a_n)^\infty_{n=1} \in \prod_\N A \mathbin{\Big|} \sum_n a^*_n a_n\text{ converges in $A$}\Big\}
\]
endowed with the diagonal right action of $A$ and with inner product $((a_n) \mid (b_n))
= \sum_n a^*_n b_n$. 

We make regular use of frames. A countable frame $\{e_j\}_{j\geq 1}\subset E_A$ for a
countably generated right Hilbert $C^*$-$A$-module $E$ is a sequence such that
\begin{equation}\textstyle
\sum_{j\geq 1}\Theta_{e_j,e_j}(e) = e\quad\text{ for all $e \in E$};
\label{eq:frame}
\end{equation}
that is, $\sum_j \Theta_{e_j, e_j}$ converges strictly to ${\rm Id}_E$ in $\End_A(E) =
\operatorname{Mult}(\End^0_A(E))$.
Frames  provide a stabilisation map in the sense of Kasparov as follows. Let
$\{e_j\}$ be a frame for $E_A$. A quick calculation using~\eqref{eq:frame} shows that
there is an isometry
\begin{equation}
\label{eq:isometry}
v : E_A\to \ell^2(A) \quad\text{ such that }\quad v(e)=((e_j\mid e)_A)_{j\geq 1}
\quad\text{ for all }e \in E.
\end{equation}
So $p := vv^* \in \End_A(\ell^2(A))$ is a projection and $E\cong p\ell^2(A)$. Writing
$\{\delta_i\}$ for the orthonormal basis of the separable Hilbert space $\ell^2$, and
writing $\Theta_{i,j}$ for the rank-one operator $h \mapsto \delta_i(\delta_j \mid h)$ on
$\ell^2$, we can express $p$ as the strict limit $p := \sum_{i,j} \Theta_{i,j} \otimes
(e_i\mid e_j)_A$.

\subsection{Toeplitz--Pimsner and Cuntz--Pimsner algebras}
\label{subsec:TE-OE}

\begin{defn}\label{defn:corr}
Let $A$ be a $C^*$-algebra. An \emph{$A$--$A$-correspondence}, or a \emph{correspondence
over $A$}, is a right $C^*$-$A$-module together with a homomorphism $\phi :A \to
\End_A(E)$, which we regard as defining a left action of $A$ on $E$. Given $a \in A$ and
$e \in E$, we frequently write $a \cdot e$ for $\phi(a)e$.
\end{defn}


{\em For all correspondences $E$ in this paper, we assume that $\phi$ is injective and
takes values in $\End_A^0(E)$.} 

The latter is automatic when $A$ is unital and $E$ is
finitely generated. These hypotheses are not necessary for constructing the
Toeplitz--Cuntz--Pimsner algebra (often just called the Toeplitz algebra) and the
Cuntz--Pimsner algebra of $E$ \cite{Katsura}, but we will need them for our later
results.

Given correspondences $E,\,F$ over $A$,
the formula $\big(e \otimes f \mid e' \otimes f')_A
= \big(f \mid (e \mid e')_A \cdot f')_A$ determines a positive-semidefinite sesquilinear
form on $E \odot F$. Taking the quotient by the subspace of vectors of length zero and
then completing
yields the balanced
tensor product $E \otimes_A F$, which is a correspondence over $A$
with left action $a\cdot(e \otimes f) = (a \cdot e) \otimes f$: see \cite[Proposition 4.5]{Lance}.

If $(e_i)$ and $(f_j)$ are frames for $E$ and $F$, then $(e_i \otimes f_j)_{i,j}$ is a
frame for $E \otimes_A F$: this is immediate for finitely generated modules, but needs a
little thought in general. If the left actions on $E$ and $F$ are implemented by
injective homomorphisms into the compacts, then so is the left action on $E \otimes_A F$.

We define $E^{\otimes 0} := {_A}A_A$, $E^{\otimes 1} := E$, and $E^{\otimes n+1} :=
E^{\otimes n} \otimes_A E$ for $n \ge 1$. The \emph{Fock module} of $E$ is the
$\ell^2$-direct sum $\Fock_E := \bigoplus_{n=0}^\infty E^{\otimes n}$ regarded as a
correspondence over $A$ with diagonal left action.

As in \cite{Pimsner}, the \emph{Toeplitz algebra} $\Toep_E$ is the $C^*$-subalgebra of
$\End_A(\Fock_E)$ generated by the creation operators $T_e$, $e\in E$ given by
$$
T_e(e_1\ox e_2\ox\cdots\ox e_k):=e\ox e_1\ox e_2\ox\cdots\ox e_k.
$$
The adjoint $T^*_e$ of $T_e$ is called the
annihilation operator associated to $e$ and
satisfies
$T^*_e(e_1 \ox \cdots \ox e_k) = (e\mid e_1)_A\cdot e_2 \ox \cdots \ox e_k$
for
$k \ge 1$, and $T^*_e|_{E^{\otimes0}} = 0$. We let
$T_a$ be the operator of left multiplication by $a\in A$ given on simple tensors by
$T_a(e_1\ox\cdots\ox e_k)=a\cdot e_1\ox\cdots\ox e_k$.

By \cite[Remark~1.2(4)]{Pimsner}, that each $\phi(a) \in \End_A^0(E)$ ensures that
$\End_A^0(\Fock_E) \subseteq \Toep_E$. The \emph{Cuntz--Pimsner algebra}
$\CPa_E$ is defined
to be the quotient $\Toep_E/\End_A^0(\Fock_E)$. Thus we have an exact sequence
\begin{equation}
0\to \End_A^0(\Fock_E) \longrightarrow \Toep_E \stackrel{q}{\longrightarrow} \CPa_E\to 0.
\label{eq:starting-sequence}
\end{equation}

Pimsner shows that $\Toep_E$ and $\CPa_E$ each enjoy a natural universal property, which
we now describe. A \emph{representation} of $E$ in a $C^*$-algebra $B$ is a pair $(\psi,
\pi)$ consisting of a linear map $\psi : E \to B$ and a homomorphism $\pi : A \to B$ such
that $\psi(a \cdot e) = \pi(a)\psi(e)$, $\psi(e \cdot a) = \psi(e)\pi(a)$ and $\pi((e\mid
f)_A)=\psi(e)^*\psi(f)$ for all $e,f \in E$ and $a \in A$. The maps $e \mapsto T_e$ and
$a \mapsto T_a$ constitute a representation of $E$ whose image generates $\Toep_E$. This
representation is universal, meaning that for any 
representation $(\psi, \pi)$ of $E$ in a $C^*$-algebra $B$, there is a homomorphism
$\psi \times \pi : \Toep_E \to B$ such that $\psi \times \pi(T_e) = \psi(e)$ and
$\psi\times \pi(T_a) = \pi(a)$ for all $e \in E$ and $a \in A$ (see
\cite[Theorem~3.4]{Pimsner}).

To describe the universal
property of $\CPa_E$, recall from \cite{Pimsner} that each representation $(\psi, \pi)$
of $E$ in a $C^*$-algebra $B$ determines  a homomorphism $\psi^{(1)} : \End_A^0(E) \to
B$ such that $\psi^{(1)}(\Theta_{e,f}) = \psi(e) \psi(f)^*$ for all $e,f \in E$. Under
our assumption that each $\phi(a) \in \End_A^0(E)$, the pair $(\psi, \pi)$ is called
\emph{covariant} if $\psi^{(1)}\circ \phi = \pi$. 
The Cuntz--Pimsner algebra $\CPa_E$ is
generated by the covariant representation of $E$ given by
$E\ni e\mapsto S_e:=q(T_e)$ and $A\ni a\mapsto S_a:=q(T_a)$ that is universal in the following sense.
For every covariant representation $(\psi, \pi)$ 
of $E$ in  a $C^*$-algebra $B$ there is a homomorphism
$\psi \times \pi : \CPa_E \to B$ such that $(\psi \times \pi)(S_e) = \psi(e)$ and $(\psi
\times \pi)(S_a) = \pi(a)$ for $e \in E$ and $a \in A$.

For each $z \in \mathbb{T}$ there is a unitary $U_z : \Fock_E \to \Fock_E$ satisfying
$U_z(\xi) = z^n\xi$ for $\xi \in E^{\otimes n}$. Writing $\operatorname{Ad} U_z \in
\Aut(\End_A(\Fock_E))$ for conjugation by $U_z$ it is routine to check that
$\operatorname{Ad} U_z$ restricts to an automorphism $\gamma_z$ of $\Toep_E$, and that
the map $z \mapsto \gamma_z$ is a strongly continuous action of $\mathbb{T}$ on
$\Toep_E$, called the \emph{gauge action}. This action satisfies $\gamma_z(T_e) =zT_e$
and $\gamma_z(T_a) = T_a$ for $e \in E$ and $a \in A$. Since $\End^0_A(\Fock_E)$ is
$\gamma$-invariant, $\gamma$ descends to an action, also denoted $\gamma$ and called the
gauge action, of $\mathbb{T}$ on $\CPa_E$.  Writing $S_e=q(T_e)$ and $S_a=q(T_a)$ we have
\begin{equation}
\label{eq:gauge}
\gamma_z(S_e)= zS_e\quad\text{ and }\quad \gamma_z(S_a)
= S_a\quad\text{ for all $e \in E$ and $a \in A$.}
\end{equation}

\subsection{Pimsner's six-term sequences}
\label{sec:CP-Kas-mods}

Some important $KK$-classes arise in the study of Cuntz--Pimsner algebras, and we
summarise them now. The first is the class of the Morita-equivalence  module $\Fock_E$ given by
$$
[\Fock_E]=[\End^0_A(\Fock_E),(\Fock_E)_A,0]\in KK(\End_A^0(\Fock_E),A).
$$
Since this $KK$-class is given by a Morita-equivalence bimodule, it is invertible in $KK$
with inverse given by the class of the conjugate module $[\overline{\Fock_E}] = [A,
\overline{(\Fock_E)}_{\End_A^0(\Fock_E)}, 0]$.

The next two $KK$-classes of particular importance to us arise from the
inclusions $\iota_{A,\Toep}:A\hookrightarrow \Toep_E$ and $\iota_{E, \Toep} :
\End_A^0(\Fock_E) \hookrightarrow \Toep_E$. These homomorphisms define classes
\[
[\iota_{A,\Toep}] = \big[A,(\Toep_E)_{\Toep_E},0\big]
\qquad\text{ and }\qquad
[\iota_{E,\Toep}] = \big[\End_A^0(\Fock_E),(\Toep_E)_{\Toep_E},0\big]
\]
in $KK(A,\Toep_E)$ and $KK(\End_A^0(\Fock_E),\Toep_E)$ respectively.

The fourth key $KK$-class was introduced by
Pimsner in \cite{Pimsner}. Let $P$ denote the
projection onto
$\Fock_E \ominus A := \bigoplus^\infty_{n=1} E^{\ox n} \subseteq\Fock_E$.
Let $\pi_0$ be the inclusion $\Toep_E \hookrightarrow \End_A(\Fock_E)$.
There is a representation $(\psi, \rho)$ of $E$
on $\End_A(\Fock_E)$ such that $\psi(e) = T_eP$
and $\rho(a) = T_aP$. So the universal property of
$\Toep_E$ gives a homomorphism
$\pi_1 : \Toep_E \to \End_A(\Fock_E)$ such that
$\pi_1(T_a) = \rho(a)$ and $\pi_1(T_e) =\psi(e)$ for all
$a \in A$ and $e \in E$. To describe $\pi_1$ explicitly, observe that
$\Fock_E \ominus A$ is  isomorphic to $\Fock_E \otimes_A E$. Under this
isomorphism, $\pi_1$ is identified with $\pi_0 \otimes 1_E$. That is
\begin{equation}
\label{eq:pi1}
\pi_1(b)(e_1 \ox \cdots \ox e_n)
    = \pi_0(b)(e_1 \ox \cdots \ox e_{n-1}) \ox e_n
        \quad\text{ for $b \in \Toep_E$ and $e_i \in e$.}
\end{equation}
In particular, the essential subspace of $\pi_1$
is $\Fock_E \ominus A \subseteq\Fock_E$.

Pimsner defines a class $[P]\in KK(\Toep_E,A)$ as the $KK$-class $\Big[\Big(\Toep_E,
{_{\pi_0 \oplus \pi_1}(\Fock_E\oplus \Fock_E)}, \Big(\begin{smallmatrix} 0 & \Id\\ \Id &
0\end{smallmatrix}\Big)\Big)\Big]$. We obtain an equivalent Kasparov module by
restricting to the essential submodule for $\pi_0 \oplus \pi_1$ (see for instance
\cite[Lemma~8.3.8]{HigsonRoe}). Explicitly, let $\iota_\Fock : \Fock_E \ominus A
\to\Fock_E$ be the inclusion map. Regard $\pi_0 \oplus \pi_1$ as an
adjointable left action of $\Toep_E$ on $\Fock_E \oplus \Fock_E \ominus A$. Then $[P]$ is
represented by the nondegenerate Kasparov module
\begin{equation}
\label{eq:P-module}
\left(\Toep_E, \begin{pmatrix} \Fock_E\\ \Fock_E\ominus A\end{pmatrix}_A,
\begin{pmatrix} 0 & \iota_\Fock\\ P & 0\end{pmatrix} \right).
\end{equation}
It turns out (see Theorem~4.4 of \cite{Pimsner}) that $[P]$ and $[\iota_{A,\Toep}]$ are
mutually inverse $KK$-equivalences:
$$
[\iota_{A,\Toep}]\otimes_{\Toep_E}[P]=\Id_{KK(A,A)},\qquad\text{ and }\qquad
[P]\otimes_{A}[\iota_{A,\Toep}]=\Id_{KK(\Toep_E,\Toep_E)}.
$$
In particular, $[P]\otimes_A\cdot\,:\,K^*(A)\to K^*(\Toep_E)$ is an isomorphism in
$K$-homology.

The fifth and final $KK$-class we will need later arises from the module $E_A$ itself.
Recalling that we assume that the left action of $A$ on $E$ is by compact endomorphisms,
we obtain a class
\[
    [E] := \big[A, E_A, 0\big] \in KK(A,A).
\]
Applying this construction with $E = {_AA_A}$, yields the Kasparov module $(A,A_A,0)$,
which is the identity element in the ring $KK(A,A)$. We will denote this variously by
$[A]$, ${\rm Id}_{KK(A,A)}$, ${\rm Id}_A$ or even just $[1]$.

Pimsner \cite[Theorem~4.9]{Pimsner} combines the $KK$-equivalences $[\Fock_E]$ and
$[\iota_{A,\Toep}]$ with the six-term exact sequence in $KK$ (in the first variable) for
the defining extension \(0 \to \End_A^0(\Fock_E) \longrightarrow \Toep_E
\stackrel{q}{\longrightarrow} \CPa_E \to 0\) to show that for any separable 
$C^*$-algebra $B$ we obtain the exact sequence
\begin{equation}
\parbox[c]{0.9\textwidth}{\centerline{\(
\xymatrix{KK^0(\CPa_E,B)\ar[r]^{(\iota_{A,\CPa})^*}
&KK^0(A,B)\ar[rr]^{(\Id_{KK(A,A)}-[E])\ox_A\cdot}&
& KK^0(A,B)\ar[d]^{\partial}\\
KK^1(A,B)\ar[u]^{\partial} & &KK^1(A,B)\ar[ll]^{(\Id_{KK(A,A)}-[E])\ox_A\cdot} & KK^1(\CPa_E,B)\ar[l]^{(\iota_{A,\CPa})^*}}
\)}}\label{eq:PV}
\end{equation}
in which the boundary maps $\partial$ are given by the taking the Kasparov product with
the class of the defining extension \(0 \to \End_A^0(\Fock_E) \longrightarrow \Toep_E
\stackrel{q}{\longrightarrow} \CPa_E \to 0\). More explicitly, the extension \(0 \to
\End_A^0(\Fock_E) \longrightarrow \Toep_E \stackrel{q}{\longrightarrow} \CPa_E \to 0\)
yields a class $\varepsilon \in KK^1(\CPa_E, \End_A^0(\Fock_E))$. Taking the Kasparov
product with the Morita-equivalence bimodule $\Fock_E$, yields a class
\begin{equation}
\label{eq:extdef}
\extcls := \varepsilon \otimes_{\End} [\Fock_E] \in KK^1(\CPa_E, A).
\end{equation}
The boundary maps $\partial :K^i(A,B) \to KK^{1-i}(\CPa_E, B)$ are then given by
$\partial(\Theta) = \extcls \otimes_A \Theta$. Explicit representatives of the class
$\extcls\in KK^1(\CPa_E,A)$ appear in \cite{RRSext,GMRkappa}. We use them to produce
concrete representatives of various classes discussed in this paper.

We conclude this section by using the exact sequence~\eqref{eq:PV} to see that for non-trivial dynamics the class
$[P]$ of~\eqref{eq:P-module} does not arise from a class
over the Cuntz-Pimsner algebra. Specifically, we obtain the following non-lifting result.

\begin{prop}
\label{prop:no-way} 
Let $E$ be a $C^*$-correspondence over $A$ with $A$ nuclear. Suppose that $[E] \neq
\Id_{KK(A,A)}$ in $KK(A,A)$. Then there is no class $x\in KK^0(\CPa_E,A)$ such that
$q^*x=[P]$. If $(A,E_A,0)=(A,A_A,0)$ then $\CPa_E\cong A\otimes C(S^1)$ and, writing $\ev :
C(S^1) \to \C$ for evaluation at~1, the class $x=[(\CPa_E,{}_{\ev}A_A,0)]$ satisfies
$q^*x = [P]$.
\end{prop}
\begin{proof}
We prove the contrapositive. Suppose that
$x \in KK^0(\CPa_E, A)$ satisfies $q^*x = [P]$.
Using the relations $(\iota_{A,\CPa})^*=(\iota_{A,\Toep})^*\circ q^*$ and
$[\iota_{A,\Toep}]=[P]^{-1}$ we see that
\[
(\iota_{A,\CPa})^*x = (\iota_{A,\Toep})^*q^*x = (\iota_{A,\Toep})^*[P] = \Id_{KK(A,A)}.
\]
So exactness of Pimsner's $K$-homology exact sequence
\[
\cdots\to KK(\CPa_E,A)\stackrel{(\iota_{A,\CPa})^*}{\longrightarrow}
KK(A,A)\stackrel{[A]-[E]}{\longrightarrow}\cdots
\]
implies that $([A]-[E]) \otimes_A \Id_{KK(A,A)} = 0$, forcing $[A]-[E] = 0$. Now suppose
that $(A,E_A,0)=(A,A_A,0)$. Then the class $x$ of the cycle $(\CPa_E,{}_{\ev}A,0)$
satisfies $(\iota_{A,\CPa})^*x=(A,A,0)=\Id_{KK(A,A)}$.
\end{proof}

\subsection{Bi-Hilbertian bimodules and their Cuntz--Pimsner algebras} \label{subsec:BHB}

Some of our results require 
bimodules that admit $A$--$A$-correspondence structures for
both the left and right
actions. Exploiting this bi-Hilbertian structure allows us to associate multiple
Cuntz--Pimsner algebras to each such module. Here we clarify the relationships between
these Cuntz--Pimsner algebras.

\begin{defn} \label{defn:bimod}
Let $A$ be a separable $C^*$-algebra. Following \cite{KajPinWat}, a
\emph{bi-Hilbertian $A$-bimodule} is a countably generated full right $C^*$-$A$-module
with inner product $(\cdot\mid \cdot)_A$ which is also a countably generated full left
$C^*$-$A$-module with inner product ${}_A(\cdot\mid \cdot)$ such that the left $A$-action
is adjointable for $(\cdot\mid \cdot)_A$, and the right $A$-action is adjointable for
${}_A(\cdot\mid \cdot)$.
\end{defn}

If $E$ is a bi-Hilbertian $A$-bimodule, then it is complete in the norms induced by both
inner products, and hence those norms are equivalent (see \cite[Lemma~2.2]{RRSext}).

When the coefficient algebra $A$ is unital, our main results apply to finitely generated
modules. For nonunital $A$, our results apply to modules $E$ arising as the restriction
to $A$ of a finitely generated module over some unitisation of $A$. As detailed in
\cite{RSims}, this condition is closely related to the 
finiteness of the right Watatani index of $E$, defined in
\cite{KajPinWat} as follows.

\begin{defn}[Kajiwara--Pinzari--Watatani]
\label{def:wat} A countably generated bi-Hilbertian $A$--$B$ bimodule has \emph{finite
right numerical index} if there is a constant $\lambda$ such that $\|\sum_j {_A(e_j \mid
e_j)}\| \le \lambda \|\sum_j \Theta_{e_j, e_j}\|$ for all finite subsets $\{e_j\}$ of
$E$. It has \emph{finite right Watatani index} if it has finite right numerical index and
there is a frame $(e_j)$ for $E_B$ such that
\[\textstyle
\sum_j{}_A(e_j\mid e_j)
\]
converges strictly in the multiplier algebra $\operatorname{Mult}(A)$.
\end{defn}

\begin{rmk}
By \cite[Corollaries 2.24~and~2.28]{KajPinWat}, the bi-Hilbertian bimodule ${}_AE_B$ has finite right
Watatani index if and only if the left action of $A$ is by compacts, and then the strict
limit
\[\textstyle
e^\beta := \sum_j{}_A(e_j\mid e_j)
\]
is independent of the choice of frame $(e_j)$, and is a positive central element of
$\operatorname{Mult}(A)$. We call this element the right Watatani index of $E$. When the
left action is injective, $e^\beta$ is also invertible (justifying our notation). For
unital algebras, we have $\operatorname{Mult}(A)=A$ and the strict topology is the norm
topology. In this case, bi-Hilbertian $A$-bimodules with finite right Watatani index are
finitely generated and projective as right modules, \cite[Corollary 2.25]{KajPinWat}.
Left Watatani index is defined analogously.
\end{rmk}

\begin{rmk}
Each left-Hilbert $A$-module $E$ determines a right-Hilbert $A^\op$-module $E^\op$, and
vice-versa. More precisely, let $E^\op$ be a copy of the Banach space $E$ (we write
$e^\op \in E^\op$ for the element corresponding to $e \in E$). There is a right action of
$A^\op$ on $E^\op$ given by $e^\op \cdot a^\op = (a \cdot e)^\op$, and there is an
$A^\op$-valued inner-product on $E^\op$ (see \cite[page~1625]{LRV}) given by
\[
(e^\op \mid f^\op)_{A^\op} = {}_A(e \mid f)^{*\op} = {}_A(f\mid e)^\op.
\]

So given a right Hilbert $A$-module $E_A$, the conjugate module $\overline{E}$ of
Section~\ref{sec:bimods} determines a right-Hilbert $A^\op$ module $\overline{E}^\op$.
\end{rmk}

\begin{lemma}
\label{lem:half-a-bar}
Given $A$--$A$-correspondences $E_1, \dots, E_n$, there is an
isometric conjugate-linear map from $E_1\ox_A\cdots\ox_AE_n$ to 
$\ol{E}^\op_1\ox_{A^\op}\cdots\ox_{A^\op}\ol{E}^\op_n$ such that
\begin{equation}\label{eq:tensor iso}
e_1\ox e_2\ox\cdots\ox e_n\mapsto \ol{e_1}^\op \ox \ol{e_2}^\op \ox \cdots \ox \ol{e_n}^\op
\end{equation}
for all $e_i \in E_i$. If the modules $E_i$ are bi-Hilbertian $A$-bimodules, then
\begin{equation}
\label{eq:other-tensor-iso}
e_1\ox e_2\ox\cdots\ox e_n\mapsto e_n^\op\ox \cdots\ox e_1^\op
\end{equation}
is a linear isomorphism from the left inner product $A$-module $({}_AE_1){}_A\ox\cdots{}_A\ox({}_AE_n)$ to the right inner product $A^\op$-module $E_{1}^\op\ox_{A^\op}\cdots\ox_{A^\op}E_n^\op$.
\end{lemma}
\begin{proof}
The result follows by an induction argument from the case $n = 2$. For $e_1, e_2 \in E_1$
and $f_1, f_2 \in E_2$, we have
\begin{align*}
\big(\ol{e_1}^\op\ox\ol{f_1}^\op\mid \ol{e_2}^\op\ox\ol{f_2}^\op\big)_{A^\op}
&=\big(\ol{f_1}^\op\mid (\ol{e_1}^\op\mid \ol{e_2}^\op)_{A^\op}\cdot\ol{f_2}^\op\big)_{A^\op}
=\big(\ol{f_1}^\op \mid \big((e_1\mid e_2)^*_A\big)^\op\cdot \ol{f_2}^\op\big)_{A^\op}\\
&=\big(\ol{f_1}^\op\mid \ol{(e_1\mid e_2)_A\cdot f_2}^\op\big)_{A^\op}
=\big((f_1 \mid (e_1\mid e_2)_A\cdot f_2)^*_A\big)^\op\\
&=\big((e_1\ox f_1\mid e_2\ox f_2)^*_A\big)^\op.
\end{align*}

So for any finite collection of elementary tensors $e_{1,1} \ox \cdots \ox e_{n,1},
\dots, e_{1,J} \ox \cdots \ox e_{n,J}$, we have $\big\|\sum_{j=1}^J e_{1,j} \ox \cdots \ox
e_{n,j}\big\|^2 = \big\|\sum_{j=1}^J \ol{e_{1,j}}^\op \ox \cdots \ox \ol{e_{n,j}}^\op\big\|^2$.
Hence the formula~\eqref{eq:tensor iso} determines a well-defined isometric linear map
from $E^{\ox n}$ to $(\overline{E}^\op)^{\ox n}$.

For the second statement we again consider the case $n=2$. We have
\begin{align*}
(e_2^\op\ox e_1^\op\mid f_2^\op\ox f_1^\op)_{A^\op}
&=(e_1^\op\mid (e_2^\op\mid f_2^\op)_{A^\op}f_1^\op)_{A^\op}
=(e_1^\op\mid {}_A(f_2\mid e_2)^\op f_1^\op)_{A^\op}\\
&=(e_1^\op\mid (f_1{}_A(f_2\mid e_2))^\op)_{A^\op}
={}_A((f_1{}_A(f_2\mid e_2)\mid e_1)^\op\\
&=({}_A(e_1\ox e_2\mid f_1\ox f_2))^{*\op}.
\end{align*}
The remainder of the argument is as before.
\end{proof}

In general, given a bi-Hilbertian $A$-bimodule, 
the Cuntz--Pimsner algebra of $E^\op$ can
be quite different from that of $E$ 
(see Example~\ref{eg:odd opposite}). By contrast, the
Cuntz--Pimsner algebras of $E$ 
and $\overline{E}^\op$ are anti-isomorphic.

\begin{lemma}
\label{lem:E-barE}
Let $E$ be a countably generated right $A$-module with an adjointable and non-degenerate
left action of $A$. Then the conjugate module $\ol{E}^\op$ is a correspondence over
$A^\op$ and there is an isomorphism $\CPa_E^\op \cong \CPa_{\ol{E}^\op}$ that carries
$S_e^\op$ to $S^{*}_{\overline{e}^\op}$ for all $e \in E$.
\end{lemma}
\begin{proof}
Define $\psi : E \to \CPa^\op_{\ol{E}^\op}$ by $\psi(e) = S^{*\op}_{\overline{e}^\op}$,
and $\pi : A \to \CPa^\op_{\ol{E}^\op}$ by $\pi(a) = \iota_{A^{\op},\CPa_{\ol{E}^\op}}(a^\op)^\op$. It is routine to
check that $(\psi, \pi)$ is a covariant representation of $E$, and so determines a
homomorphism $\CPa_E \to \CPa^\op_{\ol{E}^\op}$. The image of this representation clearly
contains all the generators of $\CPa^\op_{\ol{E}^\op}$, so it is surjective, and since it
intertwines the two gauge actions, it is injective by the gauge-invariant uniqueness
theorem \cite[Section 6]{Katsura}. 
Taking opposite algebras now gives the result.
\end{proof}

\subsection{Invertible $A$--$A$ bimodules}
\label{sub:SMEBS} An important special case of bi-Hilbertian bimodules is the class of
$C^*$-$A$--$A$-correspondences that are invertible in the category whose objects are
$C^*$-algebras and whose morphisms are isomorphism classes of $C^*$-correspondences
\cite[Lemma~2.4]{EKQR}. We will call these \emph{invertible $A$--$A$ bimodules}, or just
\emph{invertible bimodules}; they are commonly called imprimitivity $A$--$A$-bimodules or
self-Morita-equivalence bimodules. The morphisms in the $KK$-category defined by
invertible bimodules are also invertible.
\begin{defn}
\label{defn:SMEB} Let $A$ be a $C^*$-algebra. An \emph{invertible bimodule} over $A$ is a
bi-Hilbertian $A$-bimodule $E$ whose inner products are both full and satisfy the
imprimitivity condition
$$
{}_A(e\mid f)g=e(f\mid g)_A,\quad\mbox{ for all}\ e,\,f,\,g\in E.
$$
\end{defn}
These include the modules of sections of complex line bundles over locally compact
spaces, and modules arising from automorphisms of $C^*$-algebras. Hence, the class of
Cuntz--Pimsner algebras of invertible bimodules includes all crossed products by $\Z$
(see \cite[Section~2.1]{RRSext} for examples).

Pimsner identified $\CPa_E$ with $\CPa_{E\ox_A\CPa_E^\gamma}$, where $\CPa_E^\gamma$ is
the fixed point algebra for the gauge action~\eqref{eq:gauge}. In \cite{AR} it was shown
that $E\ox_A\CPa_E^\gamma$ is an invertible bimodule over $\CPa^\gamma_E$. Unfortunately,
the relationship between $A$ and $\CPa^\gamma_E$ is typically fairly complicated. For
example, in the standard realisation of the Cuntz algebra $\CPa_n$ as the Cuntz--Pimsner
algebra of $\C^n$, the coefficient algebra $A$ is $\C$, while $\CPa^\gamma_E$ is the UHF
algebra $M_{n^\infty}$. So the isomorphism $\CPa_E \cong \CPa_{E \ox_A \CPa_E^\gamma}$
does not help us relate $\CPa_E$ to $A$ except when $E$ is already an invertible
bimodule, which happens if and only if $\CPa_E^\gamma\cong A$, \cite{Katsura}.

If $E$ is an invertible bimodule, the conjugate module $\overline{E}$ satisfies $E
\otimes_A \overline{E} \cong A \cong \ol{E} \otimes_A E$. In particular, writing
$E^{-\ox n} := \overline{E}^{\otimes n}$ for $n \ge 1$, we have $E^{\otimes m} \otimes_A
E^{\otimes n} \cong E^{\otimes (m+n)}$ for all $m,n \in \Z$. We define the integer-graded
Fock space of an invertible bimodule over $A$ to be the module $\Fock_{E,\Z}=\bigoplus_{n
\in \Z} E^{\ox n}$.

\begin{lemma} \label{lem:rank-one}
Suppose the $A$-bimodule $E$ is an invertible bimodule.
Let $(\psi, \pi)$ and $(\psi^\op, \pi^\op)$
denote the universal representations of $E$ and $E^\op$ in their Cuntz--Pimsner algebras.
There is a left action $L$ of $\CPa_E$ on $\Fock_{E,\Z}$ by adjointable operators for the
right $A$-module inner-product such that $L(\pi(a))\xi =a\cdot \xi$ for all $\xi$ and
\[
L(\psi(e))(\xi) = \begin{cases}
    e \ox_A \xi &\text{ if $\xi \in \bigcup^\infty_{n=1} E^{\ox n}$}\\
    e\cdot a \in E &\text{ if $\xi = a \in A = E^{\ox 0}$} \\
    {}_A(e\mid e_1)\cdot \bar{e}_2 \ox_A \cdots \ox_A \bar{e}_n
        &\text{ if $\xi = \bar{e}_1 \ox_A \cdots \ox_A \bar{e}_n
        \in \bigcup^\infty_{n=1} \overline{E}^{\ox n}$.}
    \end{cases}
\]
There is a right action of $\CPa_{E^\op}$ on $\Fock_{E, \Z}$ by adjointable operators for
the left $A$-module inner-product such that $R(\pi^\op(a^\op))\xi = \xi\cdot a$ and
\[
R(\psi^\op(e^\op))(\xi)
    = \begin{cases}
    \xi \ox_A e &\text{ if $\xi \in \bigcup^\infty_{n=1} E^{\ox n}$} \\
    a \cdot e &\text{ if $\xi = a \in A = E^{\ox 0}$} \\
    \bar{e}_1\ox_A\cdots\ox_A \bar{e}_{n-1}\cdot(e_n\mid e)_A
        &\text{ if $\xi = \bar{e}_1 \ox_A \cdots \ox_A \bar{e}_n
        \in \bigcup^\infty_{n=1} \overline{E}^{\ox n}$.}
    \end{cases}
\]
Moreover, these actions commute.
\end{lemma}
\begin{proof}
The formulas in the lemma define left- and right-creation operators $L_e, R_{e^\op}$ for
each $e \in E$. Routine calculations show that $L_e$ is adjointable for $(\cdot\mid
\cdot)_A$ with adjoint given by
\begin{align*}
L_e^*(a)&=\bar{e}\cdot a,\quad L_e^*(e_1\ox\cdots\ox e_n)
=(e\mid e_1)_Ae_2\ox\cdots\ox e_n,\\
& L_e^*(\bar{e}_1\ox\cdots\ox \bar{e}_k)
=\bar{e}\ox \bar{e}_1\ox\bar{e}_2\ox\cdots\ox \bar{e}_k,\quad\text{ and }
\quad L_e^*(f)=(e\mid f)_A.
\end{align*}
Similarly each $R_{e^\op}$ is adjointable for ${}_A(\cdot\mid \cdot)$ with
\begin{align*}
R_{e^\op}^*(a)&=a\bar{e},\quad R_{e^\op}^*(e_1\ox\cdots\ox e_n)
=e_1\ox\cdots\ox e_{n-1}{}_A(e_n\mid e),\\
& R_{e^\op}^*(\bar{e}_1\ox\cdots\ox \bar{e}_k)
=\bar{e}_1\ox\cdots\ox \bar{e}_k\ox \bar{e},\quad\text{ and }
\quad R_{e^\op}^*(f)={}_A(f\mid e).
\end{align*}

Straightforward calculations using the imprimitivity condition and the formulas for $L_e$
and $L^*_f$ above show that $L_e L^*_f = L\big(\pi({}_A(e \mid f))\big)$ for all $e,f \in
E$. Since, for an invertible bimodule, the map $\Theta_{e,f} \mapsto {}_A(e \mid
f)$ is an isomorphism of $\End_A^0(E_A)$ onto $A$, it follows that $L$ is Cuntz--Pimsner
covariant. Similarly, $R$ is Cuntz--Pimsner covariant.

To see that $(L(x) \xi) R(y) = L(x) (\xi R(y))$, it suffices to consider $\xi$ an
elementary tensor in some $E^{\ox n}$ and (by symmetry) to consider $x$ of the form $S_e$
and $y$ either of the form $\psi^\op(f^\op)$ or of the form $\psi^\op(f^\op)^*$. This is
trivial for $|n| \ge 2$, and also when $y = \psi^\op(f^\op)$ and $n \not= -1$ and when $y
= \psi^\op(f^\op)^*$ and $n \not= 0$. When $y = \psi^\op(f^\op)$ and $n = -1$,
commutation is exactly the imprimitivity condition. The last case is $y =
\psi^\op(f^\op)^*$ and $n = 0$, so $\xi = a \in A$, with
\[
R(\psi^\op(f^\op)^*)(L(\psi(e)) a) = R(\psi^\op(f^\op)^*) (e\cdot a) = {}_A(e \cdot a \mid
f) = {}_A(e \mid f \cdot a^*),
\]
and
\[
L(\psi(e))(R(\psi^\op(f^\op))^* a)
    = L(\psi(e))(a\cdot \bar{f})
    = L(\psi(e))(\ol{f\cdot a^*})
    = {}_A(e \mid f\cdot a^*).\qedhere
\]
\end{proof}

\subsection{Relations between the two potential dual algebras} 
\label{subsec:dual-alg}
The Poincar\'e duality results of \cite{KP,PZ} suggest that one might hope to prove that
$\CPa_E$ and $\CPa_{E^\op}$ are Poincar\'e dual for suitable modules $E$. On the other
hand, Connes' picture of Poincar\'e self-duality, for the rotation algebras for example
\cite{ConnesGrav}, suggests that we should aim to decide whether $\CPa_E$ and
$\CPa_E^\op$ are Poincar\'e dual. While we always have an isomorphism $\CPa_E^\op\cong
\CPa_{\ol{E}^\op}$, our next example shows that $\CPa_{E^\op}$ can be quite different.

\begin{example}\label{eg:odd opposite}
Let $G$ be the directed graph with two vertices $G^0=\{v,w\}$,
four edges $G^1=\{g_1, \dots, g_4\}$, and
range and source maps given by $r(g_1) = s(g_1) = r(g_2) = s(g_2) = r(g_3) = v$ and
$s(g_3) = r(g_4) = s(g_4) = w$:
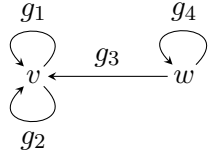
\[
\begin{tikzpicture}
    \node[circle, inner sep=1pt] (v) at (0,0) {$v$};
    \node[circle, inner sep=1pt] (w) at (2,0) {$w$};
    \draw[-stealth] (v) .. controls +(0.75,0.75) and +(-0.75, 0.75) .. (v) node[pos=0.5, above] {$g_1$};
    \draw[-stealth] (v) .. controls +(0.75,-0.75) and +(-0.75, -0.75) .. (v) node[pos=0.5, below] {$g_2$};
    \draw[-stealth] (w)--(v) node[pos=0.5, above] {$g_3$};
    \draw[-stealth] (w) .. controls +(0.75,0.75) and +(-0.75, 0.75) .. (w) node[pos=0.5, above] {$g_4$};
\end{tikzpicture}
\]
It is routine to check that the graph module (see Proposition~\ref{prp:graph classes})
 $E=E_G=C(G^1)$ over $C(G^0)$
has opposite module $E^\op = E_{G^\op}$ where $G^\op$ is the graph obtained by reversing
the edges in $G$. We use the conventions for graph algebras of \cite{Raeburn}; so $s_e^*
s_e = p_{s(e)}$. We have $\CPa^\op_E \cong C^*(G)^\op$ and $\CPa_{E^{\op}} \cong
C^*(G^\op)$. To see that these are not isomorphic, first observe that $C^*(G)$ can be
realised as the $C^*$-algebra $C^*(H_G)$ of a groupoid. In particular, the map $\theta :
C_c(H_G) \to C_c(H_G)$ given by $\theta(f)(\gamma) = f(\gamma^{-1})$ extends to an
isomorphism $C^*(G) \cong C^*(G)^\op$, \cite[Theorem 2.1]{BussSims}. So $\CPa^\op_E \cong C^*(G)$. So it suffices to
show that $C^*(G^\op) \not\cong C^*(G)$.

By the universal property of $C^*(G^\op)$, there is a $1$-dimensional
representation of $C^*(G^\op)$ given by $\pi(p_w) = \pi(S_{g_4}) = 1 \in \C$ and
$\pi(p_v) = \pi(S_{g_i}) = 0$ for $1 \le i \le 3$. On the other hand, if $\pi$ is any
nonzero representation of $C^*(G)$, then $\pi(p_v)$ is nonzero: since $S_{g_i} = p_v
S_{g_i}$ for $i \le 3$, and since $p_w = S_{g_3}^* S_{g_3}$ and $S_{g_4} = p_w S_{g_4}$,
we see that $p_v$ generates $C^*(G)$ as an ideal. So any nonzero representation $\pi$ of
$C^*(G)$ restricts to a nonzero representation of $p_v C^*(G) p_v \cong \CPa_2$, which
cannot be 1-dimensional.
\end{example}
For invertible bimodules, the ambiguity between potential dual algebras disappears.
\begin{prop}
\label{prop:op-is-op} Let $E$ be an invertible bimodule over $A$. Then $\CPa_{E^\op}
\cong \CPa_E^\op$.
\end{prop}
\begin{proof}
The left inner-product gives an isomorphism $\End_A^0(E)\cong A$ satisfying $\Theta_{e,f}
\mapsto {}_A(e\mid f)$. So by Lemma~\ref{lem:rank-one} and an application of the
gauge-invariant uniqueness theorem \cite[Section 6]{Katsura}, we have $\CPa_E \cong C^*(L_e :e\in E) \subset
\End_A(\Fock_{E,\Z})$. We claim that $\psi: e^\op \mapsto L_{e^\op}^\op, a \mapsto
\pi(a)^\op$ is a covariant Toeplitz representation of $E^\op$ in $\End_A(\Fock_{E,
\Z})^\op$. To see this, fix $e^\op,f^\op \in E^\op$. Then
\begin{align*}
 \psi( (e^\op\mid f^\op )_{A^\op}) &= \psi\big({}_A( f\mid e )^\op\big)
 = \pi\big({}_A( f\mid e)\big)^\op
 = (L_{f^\op} L_{e^\op}^*)^\op
= L_{e^\op}^{*\op} L_{f^\op}^\op
= \psi(e)^*\psi(f),
\end{align*}
and for covariance, we calculate that
\[
 \psi^{(1)}(\Theta_{e,f}) = L_e^\op L_f^{* \op} = (L_f^* L_e)^\op
 = \pi\big((f\mid e)_A\big)^\op = \psi\big((f\mid e)_A^\op\big)
 = \psi\big({}_{A^\op}(e\mid f)\big).
\]

The universal property of $\CPa_{E^\op}$ gives a homomorphism $\psi \times \pi :
\CPa_E^\op \to C^*(\{L_e^\op : e \in E\})$, and therefore induces a homomorphism
$\widetilde{\psi} : \CPa_{E^\op} \to \CPa_E^\op$. This $\widetilde{\psi}$ is surjective
because the elements $\widetilde{\psi}(S_{e^{\op}})$ generate $\CPa_E^\op$. Injectivity
follows from the gauge-invariant uniqueness theorem since $\widetilde{\psi} : A^\op \to
\CPa_E^\op$ is injective and the gauge action on $\CPa_E$ induces a gauge action on
$\CPa_E^\op$.
\end{proof}
\section{Basic criteria for Poincar\'e duality of Cuntz--Pimsner algebras}
\label{sec:funky-diagram}
In this section we derive  conditions under which Poincar\'e
self-duality of a $C^*$-algebra $A$ induces Poincar\'e self-duality for the
Cuntz--Pimsner algebra $\CPa_E$ of a $C^*$-correspondence $E$ over $A$. When $E$ is a
bi-Hilbertian bimodule, we also investigate when Poincar\'e self-duality for $A$ induces
Poincar\'e duality between $\CPa_E$ and $\CPa_{E^\op}$.

Following \cite{KasparovTech}, we say that $C^*$-algebras $A$ and $B$ are
\emph{Poincar\'e dual} if there exist classes $\mu\in KK^d(A\ox B,\C)$ (the Dirac class)
and $\beta\in KK^d(\C,A\ox B)$ (the Bott or dual-Dirac class) such that
\begin{align}
\beta\ox_A\mu=(-1)^d\Id_{KK(B,B)}\qquad\text{ and }\qquad \beta\ox_{B}\mu=\Id_{KK(A,A)}.
\label{eq:PD}
\end{align}
The classes $\mu$ and $\beta$ implement isomorphisms
\begin{align}
\cdot\otimes_{B}\mu &: K_*(B)\stackrel{\cong}{\longrightarrow} K^{*+d}(A),\qquad&
    \cdot\otimes_A\mu &: K_*(A)\stackrel{\cong}{\longrightarrow} K^{*+d}(B),\qquad\text{
    and}\label{eq:mu isos}\\
\beta \otimes_{B} \cdot &: K^*(B)\stackrel{\cong}{\longrightarrow} K_{*+d}(A),\qquad&
    \beta \otimes_A \cdot &: K^*(A)\stackrel{\cong}{\longrightarrow} K_{*+d}(B).\label{eq:beta isos}
\end{align}
We call a class $\mu \in KK(A \otimes B, \C)$ implementing isomorphisms as
in~\eqref{eq:mu isos} a \emph{$K$-homology fundamental class}, even if there is no corresponding class $\beta$. Similarly, we call a class $\beta$
implementing isomorphisms as in~\eqref{eq:beta isos} a \emph{$K$-theory fundamental
class}.
If $A$ and $A^\op$ are Poincar\'e dual, then we say that $A$ is \emph{Poincar\'e
self-dual}.

Strictly speaking, the formulation here is appropriate only for
unital algebras. For non-unital algebras Poincar\'e duality should be formulated using an appropriate 
analogue of
compactly supported $K$-homology in Equations \eqref{eq:mu isos} and \eqref{eq:beta isos}, 
though $\beta$ and $\mu$ need not be and usually are not compactly supported. For 
commutative algebras and $C(X)$-algebras,
a suitable compactly supported theory is provided by $RKK$, defined by Kasparov in \cite{KasparovEqvar}. A version 
for some non-commutative algebras is presented in \cite{RennieSmooth}.
We will restrict our discussion to the formulation of
Poincar\'e duality above, but include 
some additional non-unital
results  for future use.

Suppose that $A$ is Poincar\'e self-dual, and consider a bi-Hilbertian $A$-bimodule $E$.
Recall that 
$[P]$ denotes the Kasparov class 
of the Kasparov
module described in \eqref{eq:P-module}. The $KK$-equivalences between $A$ and $\Toep_E$ and between
$A$ and $\End_A^0(\Fock_E)$ described in Section~\ref{sec:CP-Kas-mods}, and the
corresponding equivalences between $A^\op$ and $\Toep_{E}^\op$ and between $A^\op$ and
$\End_{A}^0(\Fock_{E})^\op$ lift $(\mu,\beta)$ to Poincar\'e self-dualities $(\mu_\Toep,
\beta_\Toep)$ for $\Toep_E$ and $(\mu_E, \beta_E)$ for $\End_A^0(\Fock_E)$ as follows:
\begin{align*}
\mu_\Toep&=([P]\ox[P^\op])\ox_{A\ox A^\op}\mu_A,\quad
\beta_\Toep=\beta_A\ox_{A\ox A^\op}(\iota_{A,\Toep}\ox\iota_{A^\op,\Toep^\op})\\
\mu_E&=([\Fock_E]\ox[\Fock_{E}^\op])\ox_{A\ox A^\op}\mu_A,\quad
\beta_E=\beta_A\ox_{A\ox A^\op}([\overline{\Fock_E}]\ox[\overline{\Fock_{E}^\op}]).
\end{align*}

We cannot expect simple formulae of the same sort to describe classes implementing
Poincar\'e self-duality for $\CPa_E$. One reason for this is
Proposition~\ref{prop:no-way}. Another is that we expect a shift in parity: the algebra
$\CPa_E$ is in important respects like a crossed product of $A$ by $\Z$, so the passage
from $A$ to $\CPa_E$ should add a noncommutative dimension, leading us to expect that the
fundamental class for $\CPa_E$ has parity $d+1$ if $\mu$ has parity $d \in \Z/2\Z$.

\subsection{Lifting Poincar\'e duality from the coefficient algebra}
\label{subsec:lift}

In this subsection, assuming Poincar\'e self-duality for the coefficient algebra $A$, we
produce sufficient conditions for the existence of fundamental classes implementing a Poincar\'e duality
$$
\delta\in KK^{d+1}(\C,\CPa_E\ox \CPa_{E^\op}) \quad\mbox{and}\quad \Delta\in
KK^{d+1}(\CPa_E\ox \CPa_{E^\op},\C),
$$
and also for the existence of fundamental classes implementing a Poincar\'e duality
$$
\ol{\delta}\in KK^{d+1}(\C,\CPa_E\ox \CPa_{E}^\op)
\quad\mbox{and}\quad
\ol{\Delta}\in KK^{d+1}(\CPa_E\ox \CPa_{E}^\op,\C).
$$
Our sufficient conditions involve the dynamics encoded by $E$, the modules $E^\op$ and
$\ol{E}^\op$, and the existence of suitably $E$-invariant Poincar\'e self-duality classes
for $A$. First, starting from the extension
\begin{equation}
0\to \End^0_A(\Fock_E)\stackrel{\iota_{E,\Toep}}{\longrightarrow}
\Toep_E\stackrel{q}{\longrightarrow} \CPa_E\to 0
\label{eq:ext}
\end{equation}
and the analogous extensions for $E^\op$ and $\overline{E}^\op$, we describe sufficient
conditions under which classes $\overline{\Delta}\in KK^{d+1}(\CPa_E\ox \CPa_{E}^\op,\C)$
and $\overline{\delta}\in KK^{d+1}(\C,\CPa_E\ox \CPa_{E}^\op)$ yield isomorphisms
\begin{equation}\label{eq:isos}
\begin{split}
\cdot\ox_{\CPa_E}\overline{\Delta}\,:\,K_*(\CPa_{E})\to K^{*+d+1}(\CPa_{E}^\op),
    &\qquad \cdot\ox_{\CPa_{E^\op}}\overline{\Delta}\,:\,K_*(\CPa_{E}^\op)\to K^{*+d+1}(\CPa_{E})\\
\overline{\delta}\ox_{\CPa_E}\cdot\,:\,K^*(\CPa_E)\to K_{*+d+1}(\CPa_{E}^\op),
    &\qquad \overline{\delta}\ox_{\CPa_{E^\op}}\cdot\,:\,K^*(\CPa_{E}^\op)\to K_{*+d+1}(\CPa_{E}).
\end{split}
\end{equation}
We then give sufficient conditions for classes $\Delta\in KK^{d+1}(\CPa_E\ox \CPa_{E}^\op,\C)$
and $\delta\in KK^{d+1}(\C,\CPa_E\ox \CPa_{E}^\op)$ to yield isomorphisms
\begin{equation}
\label{eq:isos-op}
\begin{split}
\cdot\ox_{\CPa_E}{\Delta}\,:\,K_*(\CPa_{E})\to K^{*+d+1}(\CPa_{E^\op}),
    &\qquad \cdot\ox_{\CPa_{E^\op}}{\Delta}\,:\,K_*(\CPa_{E^\op})\to K^{*+d+1}(\CPa_{E})\\
{\delta}\ox_{\CPa_E}\cdot\,:\,K^*(\CPa_E)\to K_{*+d+1}(\CPa_{E^\op}),
    &\qquad {\delta}\ox_{\CPa_{E^\op}}\cdot\,:\,K^*(\CPa_{E^\op})\to K_{*+d+1}(\CPa_{E}).
\end{split}
\end{equation}

Of course, the maps~\eqref{eq:isos} must be isomorphisms if $\overline{\delta}$ and
$\overline{\Delta}$ implement a Poincar\'e self-duality. The converse is false, but we
show that the existence of classes $(\overline{\delta}, \overline{\Delta})$ satisfying
our sufficient conditions for~\eqref{eq:isos} implies that $\CPa_E$ and
$\CPa_{\ol{E}^\op}$ \emph{are} Poincar\'e dual, via a suitably modified pair of
$KK$-classes. A related result, that 
isomorphisms $KK(A\ox B,C)\to KK(B,A^\op\ox C)$
for all $B,\,C$ ensures $A$ and $A^\op$ 
are Poincar\'e dual
appears in \cite{EEK}.

Similarly, while not every pair $(\delta, \Delta)$ satisfying our sufficient condition
guaranteeing~\eqref{eq:isos-op} implements a Poincar\'e duality between $\CPa_E$ and
$\CPa_{E^\op}$, the existence of such a pair does guarantee Poincar\'e duality of
$\CPa_E$ and $\CPa_{E^\op}$, again via a modified pair of $KK$-classes.

We begin by  combining the $K$-theory exact sequence arising from the  exact sequence~\eqref{eq:ext} for $\ol{E}^\op$
with the $K$-homology exact sequence arising from
\eqref{eq:ext} using the Poincar\'e self-duality of $A$.

\begin{thm}\label{thm:(OE)op sufficient}
Let $E$ be a correspondence over $A$ with compact and non-degenerate left action of $A$.
Suppose that $\mu\in KK^d(A\ox A^\op,\C)$ and $\beta\in KK^d(\C,A\ox A^\op)$ implement a
Poincar\'e self-duality for $A$. Let $\iota_{A, \CPa} : A \to \CPa_E$ be the canonical
inclusion.
\begin{enumerate}
\item Suppose that $[E]\ox_A\mu=[\ol{E}^\op]\ox_{A^\op}\mu \in KK(A \ox A^{\op},
    \C)$. Suppose that $\overline{\Delta}\in KK^1(\CPa_E\ox \CPa_{E}^{\op},\C)$
    satisfies
    \begin{equation}
    \iota_{A,\CPa}\ox_{\CPa_E}\overline{\Delta} =-\extobarcls\otimes_{ A^{\op}}\mu
        \quad\text{ and }\quad
    \iota_{A^\op,\CPa^\op}\ox_{\CPa_E^\op}\overline{\Delta} = \extcls\ox_{A}\mu.
    \label{eq:comm1}
    \end{equation}
    Then the maps defined by $\overline{\Delta}$ in~\eqref{eq:isos} are isomorphisms,
    so $\ol{\Delta}$ is a $K$-homology fundamental class.
\item Suppose that $\beta\ox_A[E]=\beta\ox_{A^\op}[\ol{E}^\op] \in KK(\C, A \ox
    A^{\op})$. Suppose that $\overline{\delta}\in KK^{d+1}(\C,\CPa_E\ox
    \CPa_{E}^\op)$ satisfies
    \begin{equation}
    \beta\ox_A \iota_{A,\CPa}=-\overline{\delta}\ox_{\CPa_{E}^\op}\extobarcls
    \quad\text{ and }\quad
    \beta\ox_{A^\op} \iota_{A^\op,\CPa^\op}=-\overline{\delta}\ox_{\CPa_{E}}\extcls.
    \label{eq:comm2}
    \end{equation}
    Then the maps defined by $\overline{\delta}$ in~\eqref{eq:isos} are isomorphisms,
    so $\ol{\delta}$ is a $K$-theory fundamental class.
\item Suppose that there exist classes $\overline{\Delta}$ and $\overline{\delta}$ satisfying
    the conditions in (1)~and~(2). Then $\CPa_E$ is Poincar\'e self-dual.
\end{enumerate}
\end{thm}

To prove the theorem, we first need two lemmas.

\begin{lemma}
\label{lem:diagram-commutes} Let $E$ be a correspondence over $A$ with compact and
non-degenerate left action of $A$. Suppose that $\mu\in KK^d(A\ox A^\op,\C)$ and
$\beta\in KK^d(\C,A\ox A^\op)$ implement a Poincar\'e self-duality for $A$. Resume the
notation of Section~\ref{sec:CP-Kas-mods}, and write $\partial$ for all boundary maps in
Pimsner's six-term $K$-theory and $K$-homology sequences. If
$[E]\ox_A\mu=[\ol{E}^\op]\ox_{A^\op}\mu$ then all subdiagrams consisting of solid arrows
in the diagram
\begin{equation}
\parbox[c]{0.9\textwidth}{\centerline{
\xymatrix{ &  & K_0(A^\op)\ar[ddrr]^{\ \quad\cdot\ox_{A^\op}({\rm Id}_{A^\op}-[\ol{E}^\op])} \ar[d]^{\ox_{A^\op}\mu}  &  & \\
 &  & K^d(A) \ar[dr]^{\ \ ({\rm Id}_{A}-[E])\ox_A\cdot} &  & \\
K_1(\CPa_{E}^\op)\ar@{-->}[r]^\theta\ar[uurr]^{\partial^\op}& K^{d}(\CPa_E)\ar[ur]^{\iota_{A,\CPa}^*} & & K^d(A)\ar[d]^\partial &
K_0(A^\op)\ar[l]^{\ox_{A^\op}\mu} \ar[d]^{\iota_{A^\op,\CPa^\op*}} \\
K_1(A^\op)\ar[u]^{\iota_{A^\op,\CPa^\op*}} \ar[r]^{\ox_{A^\op}\mu}& K^{d+1}(A)\ar[u]^\partial &  & K^{d+1}(\CPa_E)\ar[dl]^{\iota_{A,\CPa}^*} & K_0(\CPa_{E}^\op) \ar@{-->}[l]^\theta\ar[ddll]^{\partial^\op}\\
 &  & K^{d+1}(A) \ar[ul]^{({\rm Id}_{A}-[E])\ox_A\cdot\ \ } &  & \\
 &  & K_1(A^\op)\ar[u]^{\ox_{A^\op}\mu}\ar[uull]^{\cdot\ox_{A^\op}({\rm Id}_{A^\op}-[\ol{E}^\op])\quad\ } &  & }}}
\label{eq:diagram}
\end{equation}
commute. Any homomorphism $\theta : K_*(\CPa_{E}^\op)\to K^{*+1+d}(\CPa_E)$ making the
whole diagram commute is an isomorphism.
\end{lemma}
\begin{proof}
The completely defined squares in the upper right and lower left commute by the
hypothesis on $\mu$. The remaining solid rectangles trivially commute because exactness
of the two hexagonal sequences shows that the long sides of these rectangles are zero
maps. The final assertion follows from the five lemma.
\end{proof}

\begin{lemma} \label{lem:oh-yeah}
Suppose that $\overline{\Delta}$ and $\overline{\delta}$ are as in
Theorem~\ref{thm:(OE)op sufficient}(1)~and~(2) respectively. Write $\sigma_{12} :
\CPa^\op_E \otimes \CPa_E \to \CPa_E \otimes \CPa^\op_E$ for the flip
isomorphism, and define $v
\in KK(\CPa_E, \CPa_E)$ by
\begin{equation}\label{eq:vdef}
v=(\overline{\delta}\hox_\C\Id_{KK(\CPa_E,\CPa_E)})\hox_{\CPa_E\ox \CPa_E^\op\ox \CPa_E}
(\Id_{KK(\CPa_E,\CPa_E)}\hox_{\C}\sigma_{12}^*\overline{\Delta}).
\end{equation}
Then
$$ \iota_{A,\CPa_E}\ox_{\CPa_E}v=\iota_{A,\CPa_E}\quad\text{ and }\quad v\ox_{\CPa_E}\extcls=\extcls.
$$
\end{lemma}
\begin{proof}
We calculate:
\begin{align*}
\iota&_{A,\CPa_E}\hox_{\CPa_E}v
    =(\overline{\delta}\hox_\C\,\iota_{A,\CPa_E})\hox_{\CPa_E\ox \CPa_E^\op\ox \CPa_E}
        (\Id_{KK(\CPa_E,\CPa_E)}\hox_{\C}\,\sigma_{12}^*\overline{\Delta})\\
    &=(\overline{\delta}\hox_{\CPa_E\ox \CPa_E^\op}(\Id_{KK(\CPa_E,\CPa_E)}\hox_\C
        \Id_{KK(\CPa_E^\op,\CPa_E^\op)}\hox_\C\,\iota_{A,\CPa_E}))\hox_{\CPa_E\ox \CPa_E^\op\ox \CPa_E}
        (\Id_{KK(\CPa_E,\CPa_E)}\hox_{\C}\,\sigma_{12}^*\overline{\Delta})\\
    &=(\overline{\delta}\hox_{\CPa_E\ox \CPa_E^\op}(\Id_{KK(\CPa_E,\CPa_E)}\hox_\C
        \iota_{A,\CPa_E}\hox_\C\Id_{KK(\CPa_E^\op,\CPa_E^\op)}))\hox_{\CPa_E\ox \CPa_E\ox \CPa_E^\op}
        (\Id_{KK(\CPa_E,\CPa_E)}\hox_{\C}\,\overline{\Delta})\\
    &=-\overline{\delta}\hox_{\CPa_E\ox \CPa_E^\op}\Id_{KK(\CPa_E,\CPa_E)}\hox_\C\extobarcls\hox_{A^\op}\mu\\
    &=\beta\ox_A\iota_{A,\CPa_E}\ox_{A^\op}\mu\\
    &=\iota_{A,\CPa_E}.
\end{align*}
The second equality is proved similarly, remembering the anti-symmetry of the external
product to see  that
\begin{align*}
v\ox_{\CPa_E}\extcls&=\ol{\delta}\ox_{\CPa_{E}^\op}\ol{\Delta}\ox_{\CPa_E}\extcls
=(\ol{\delta}\ox_{\CPa_E}\extcls)\ox_{\CPa_E^\op}\ol{\Delta}(-1)^{d+1}\\
&=-\beta\ox_{A^\op}\iota_{A^\op,O^\op}\ox_{\CPa_E^\op}\ol{\Delta}(-1)^{d+1}
=-\beta\ox_{A^\op}(\extcls\ox_A\mu)(-1)^{d+1}\\
&=\extcls\ox_A(\beta\ox_{A^\op}\mu)
=\extcls.\qedhere
\end{align*}
\end{proof}

\begin{proof}[Proof of Theorem~\ref{thm:(OE)op sufficient}]
(1) By \cite[Section 7]{KasparovTech}, the boundary maps $\partial$ and $\partial^\op$
are implemented by Kasparov products with $\extcls$ and $\extobarcls$ respectively. So
the conditions in \eqref{eq:comm1} are equivalent to commutation of the
diagram~\eqref{eq:diagram}.

(2) There is a diagram dual to~\eqref{eq:diagram} in which $\cdot\ox_A\mu$ and
$\cdot\ox_{A^\op}\mu$ are replaced by maps $\beta\ox_A\cdot:\,K^*(A)\to K_{*+d}(A^\op)$
and $\beta\ox_{A^\op}\cdot:\,K^*(A^\op)\to K_{*+d}(A)$ from the inner exact hexagon to
the outer one (and the arrows $\theta$ are reversed). The proof of~(2) is the same as
that of~(1) but applied to this dual diagram and the conditions in \eqref{eq:comm2}.

(3) Let $v \in KK(\CPa_E, \CPa_E)$ denote the product $v := \overline{\delta}
\ox_{\CPa_E^\op} \overline{\Delta}$ (see~\eqref{eq:vdef}). Define $\overline{\delta}' :=
\overline{\delta}\ox_{\CPa_E}v^{-1}$. We will show that $\CPa_E$ is Poincar\'e self-dual
with duality implemented by the pair $\overline{\delta}'$, $\overline{\Delta}$. We
calculate:
\begin{align*}
\overline{\delta}'\ox_{\CPa_E^\op}\!\overline{\Delta}\!
&:=\!(\overline{\delta}\hox_{\CPa_E\ox \CPa_E^\op}(v^{-1}\hox_\C\Id_{KK(\CPa_E^\op,\CPa_E^\op)}))
\hox_\C\Id_{KK(\CPa_E,\CPa_E)}\hox_{\CPa_E\ox \CPa_E^\op\ox \CPa_E}(\Id_{KK(\CPa_E,\CPa_E)}\hox\,
\sigma_{12}^*\overline{\Delta})\\
&=(\overline{\delta}\hox_\C\Id_{KK(\CPa_E,\CPa_E)})\hox_{\CPa_E\ox \CPa_E^\op\ox \CPa_E}\\
&\qquad\qquad\qquad(v^{-1}\hox_\C\Id_{KK(\CPa_E^\op,\CPa_E^\op)}\hox_\C\Id_{KK(\CPa_E,\CPa_E)})
\hox_{\CPa_E\ox \CPa_E^\op\ox \CPa_E}(\Id_{KK(\CPa_E,\CPa_E)}\hox\,\sigma_{12}^*\overline{\Delta})\\
&=(\overline{\delta}\hox_\C\Id_{KK(\CPa_E,\CPa_E)})\hox_{\CPa_E\ox \CPa_E^\op\ox \CPa_E}
(v^{-1}\hox_\C\,\sigma_{12}^*\overline{\Delta})\\
&=(\overline{\delta}\hox_\C\Id_{KK(\CPa_E,\CPa_E)})\hox_{\CPa_E\ox \CPa_E^\op\ox \CPa_E}
(\Id_{KK(\CPa_E,\CPa_E)}\hox_\C\,\sigma_{12}^*\overline{\Delta})\hox_{\CPa_E}v^{-1}\\
&=v\hox_{\CPa_E}v^{-1}=\Id_{KK(\CPa_E,\CPa_E)}.
\end{align*}
Combining $\overline{\delta}'\ox_{\CPa_E^\op}\overline{\Delta}=\Id_{KK(\CPa_E,\CPa_E)}$
with the anti-symmetry of the external product, we see that
\begin{align*}
\overline{\delta}'\ox_{\CPa_E}\!\overline{\Delta}
    &= \ol{\delta}'\ox_{\CPa_E}(\ol{\delta}'\ox_{\CPa_E^\op}\ol{\Delta})\ox_{\CPa_E}\ol{\Delta}\\
    &=((-1)^{d+1}\ol{\delta}'\ox_{\CPa_E^\op}(\ol{\delta}'\ox_{\CPa_E}\ol{\Delta}))\ox_{\CPa_E}\ol{\Delta}
    =(-1)^{d+1}(\ol{\delta}'\ox_{\CPa_E}\ol{\Delta})\ox_{\CPa_E^\op}(\ol{\delta}'\ox_{\CPa_E}\ol{\Delta}).
\end{align*}
Multiplying through by $(-1)^{d+1}$ gives $\big((-1)^{d+1}\overline{\delta}' \ox_{\CPa_E}\overline{\Delta}\big)^2 
=
(-1)^{d+1}\overline{\delta}' \ox_{\CPa_E}\overline{\Delta}$. So
$(-1)^{d+1}\overline{\delta}' \ox_{\CPa_E}\overline{\Delta}$ is an idempotent  in
the group of units of the ring $KK(\CPa_E^\op, \CPa_E^\op)$, and therefore equal to
$\Id_{KK(\CPa_E^\op, \CPa_E^\op)}$. Hence $\overline{\delta}'
\ox_{\CPa_E}\overline{\Delta}  = (-1)^{d+1}\Id_{KK(\CPa_E^\op, \CPa_E^\op)}$.
%
%
\end{proof}

If $E$ is a bi-Hilbertian $A$-bimodule with left action by compacts,
we can repeat the discussion above using the module $E^\op$, replacing
$\ol{E}^\op$ by 
$E^\op$, $\CPa_E^\op$ by $\CPa_{E^\op}$, and $\extobarcls$ by $\extocls$.
We obtain the following result.

\begin{thm}
\label{thm:(OEop) sufficient}
Let $E$ be a bi-Hilbertian bimodule over $A$ with compact and non-degenerate left action
of $A$. Suppose that $\mu\in KK^d(A\ox A^\op,\C)$ and $\beta\in KK^d(\C,A\ox A^\op)$
implement a Poincar\'e self-duality for $A$. Let $\iota_{A, \CPa} : A \to \CPa_E$ be the
canonical inclusion.
\begin{enumerate}
\item Suppose that $[E]\ox_A\mu=[E^\op]\ox_{A^\op}\mu$. Suppose that $\Delta\in KK^1(\CPa_E\ox
    \CPa_{E^{\op}},\C)$ satisfies
    \[
    \iota_{A,\CPa_E}\ox_{\CPa_E}\Delta = \extocls\otimes_{ A^{\op}}\mu \quad\text{ and }\quad
    \iota_{A^\op,\CPa_{E^\op}}\ox_{\CPa_{E^\op}}\Delta = \extcls\ox_{A}\mu.
    \]
    Then the maps defined by $\Delta$ in~\eqref{eq:isos-op} are isomorphisms, so $\Delta$ is a $K$-homology fundamental class.
\item Suppose that $\beta\ox_A[E]=\beta\ox_{A^\op}[E^\op]$. Suppose that $\delta\in
    KK^{d+1}(\C,\CPa_E\ox \CPa_{E^\op})$ satisfies
    \[
    \beta\ox_A \iota_{A,\CPa_E}=\delta\ox_{\CPa_{E^\op}}\extocls
    \quad\text{ and }\quad
    -\beta\ox_{A^\op} \iota_{A^\op,\CPa_{E^\op}}=\delta\ox_{\CPa_{E}}\extcls.
    \]
    Then the maps defined by $\delta$ in~\eqref{eq:isos-op} are isomorphisms, so $\delta$ is a $K$-theory fundamental class.
\item Suppose that there exist classes $\delta$ and $\Delta$ satisfying the conditions in parts
    (1)~and~(2). Then $\CPa_E$ and $\CPa_{E^\op}$ are Poincar\'e dual.
\end{enumerate}
\end{thm}
%
%
%
%

\subsection{Examples of Poincar\'e duality classes for the coefficient algebra}
\label{subsec:examples}

The restrictions on the classes $\beta$ and $\mu$ in the preceding section are
fairly stringent, so we discuss two key examples where they are satisfied: Cuntz--Krieger
algebras and crossed products by $\Z$ in geometric settings. We will carry these examples
through the remainder of the paper. We describe the Poincar\'e duality classes, and the
conditions for commutation of the diagram.

\subsubsection{Finite dimensional coefficient algebras and Cuntz--Krieger algebras}
\label{subsub:findim-alg}

Our first example is $A=\C$. Clearly $\C = \C^\op = \C \ox \C^\op$, but we avoid these
identifications in the first instance to clarify how the components of our discussion
relate to the general setting. The $K$-homology fundamental class is
$\mu=[\C\ox\C^\op,\C,0]$, where the left action is $(w_1 \ox w^\op_2)\cdot z = w_1w_2z$.
The $K$-theory fundamental class is $\beta=[\C,\C_{\C \ox \C^\op},0]$ where the right
action is $z\cdot (w_1 \otimes w^\op_2) = zw_1w_2$ and the inner-product is $(w \mid z) =
(\bar{w} \ox z^\op)$. These classes trivially constitute a Poincar\'e self-duality for
$\C$.

Tensoring with the explicit Morita equivalence cycle $(M_n(\C),\C^n_\C,0)$ and its
inverse $(\C,\C^n_{M_n(\C)},0)$ yields a Poincar\'e self-duality for $M_n(\C)$. We also
obtain Poincar\'e duality classes for the compact operators, on $\ell^2(\N)$ say, but
this takes us out of the finite-index setting needed later: 
see \cite{RSims}.

To extend the preceding paragraph from $A = \C$ to $A = \C^r$, we record the following
easy lemma.
\begin{lemma} \label{lem:oplus PD}
Let $A$ be a $C^*$-algebra, and suppose that
\[
\mu \in KK(A\ox A^\op,\C),\quad\text{ and }\quad \beta \in KK(\C,A\ox A^\op)
\]
implement a Poincar\'e self-duality for $A$. Then for each $r\geq 1$, the algebra $A^r :=
\oplus_{i=1}^r A$ is Poincar\'e self-dual with respect to the classes
\[\textstyle
\tilde\mu =\bigoplus_{i=1}^r\mu\in KK(A^r\ox (A^{\op })^r,\C),\quad\text{ and }\quad
    \tilde\beta = \bigoplus_{i=1}^r\beta\in KK(\C,A^r \ox (A^{\op })^r).
\]
\end{lemma}
\begin{proof}
We compute
\[\textstyle
\tilde\beta\ox_{A^r}\tilde\mu
    = \bigoplus_{i,j=1}^r\beta\ox_A\mu
    = \bigoplus_{i=1}^r\Id_{KK(A^\op,A^\op)}
    = \Id_{KK((A^\op)^r,(A^\op)^r)},
\]
and similarly for $\tilde{\beta} \ox_{(A^{\op })^r} \tilde{\mu}$.
\end{proof}

\begin{rmk}\label{rmk:Cr PD}
Lemma~\ref{lem:oplus PD} shows that for any finite set $X$, the algebra $C(X) := \C^X$ is
Poincar\'e self-dual with respect to the classes
\begin{equation}\label{eq:PD-for-KP}
\begin{split}
\mu&=[(C(X)\ox C(X),C(X)_\C,0)]\in KK(C(X)\ox C(X),\C),\quad\text{ and}\\
\beta&=[(\C, C(X)_{C(X)\ox C(X)}, 0)]\in KK(\C, C(X)\ox C(X)).
\end{split}
\end{equation}
\end{rmk}

The inner product is given on basis vectors
of $C(X)$ by 
$$
(\delta_x\mid \delta_y)_{C(X) \ox
C(X)} = \left\{\begin{array}{ll} \delta_x \ox \delta_x & x=y\\ 0 & \mbox{otherwise}\end{array}\right.,
$$ 
so the action of $C(X)$ is diagonal. In the language of projections,
\begin{equation}
\beta = \sum_{x \in X} [\delta_x] \ox [\delta_x^\op].
\label{eq:beta-proj}
\end{equation}

\begin{rmk}
As discussed above for the 1-summand case, the Poincar\'e self-dualities in
Remark~\ref{rmk:Cr PD} determine Poincar\'e self-dualities for algebras of the form
$\bigoplus_{x \in X} \K(\ell^2(I_x))$ for any collection of finite index
sets $I_x$ just by tensoring with the $KK$-classes of the invertible bimodules
$$
{_{\bigoplus_x \K(\ell^2(I_x))}}\Big(\bigoplus_x \ell^2(I_x)\Big)_{C(X)}.
$$
\end{rmk}

The following result shows that the hypotheses on $\beta$ and $\mu$ in
Theorem~\ref{thm:(OEop) sufficient} hold for finite graph algebras where $A=C(G^0)$,
whether we use $E^\op$ or $\ol{E}^\op$. We build on this to produce Poincar\'e duality
classes for Cuntz--Krieger algebras and appropriate graph algebras in Sections
\ref{ex:cuntz-krieger}~and~\ref{subsub:GMI}.

\begin{prop}
\label{prp:graph classes}
Let $G = (G^0, G^1, r, s)$ be a finite directed graph
with no sources. Let
$\mu,\beta$ be as in Remark~\ref{rmk:Cr PD}
for $X = G^0$. Let $E$ be the edge module $E
= C(G^1)$ with
\[
(a\cdot e\cdot b)(g)
=a(r(g))e(g)b(s(g))\qquad\text{ for all $a,b\in A$, all $e\in E$ and all $g\in G^1$}
\]
and
\[
(e \mid f)_{A}(v) = \sum_{s(g) = v} \overline{e(g)}f(g)\qquad\mbox{and}\qquad
{}_{A}(e \mid f)(v) = \sum_{r(g) = v} e(g)\overline{f(g)}.
\]
Then $\beta \otimes_A [E] = \beta \otimes_{A^\op}[\overline{E}^\op]$ and $[E] \otimes_A
\mu = [\ol{E}^\op] \otimes_{A^\op} \mu$.
If 
$G$ has no sinks then we also have
$\beta \otimes_A [E]=\beta \otimes_{A^\op} [E^\op]$
and
$[E] \otimes_A\mu=[E^\op]\otimes_{A^\op}\mu$.
\end{prop}
\begin{rmk}
The hypothesis of of `no sources' is necessary for the
injectivity of the left action of $A$ on $E$. 
It also implies the injectivity of the
left action of $A^\op$ on $\ol{E}^\op$, 
and the right action of $A^\op$ on $E^\op$.
The hypothesis of `no sinks' gives injectivity of the
left action of $A^\op$ on $E^\op$.
\end{rmk}
\begin{proof}
We first compute $\beta\ox_{A}[E]:=\beta\ox_{A\ox A^\op}([E]\ox[A^\op])$. Using the
diagonality of  the action and the inner product for the class $\beta$, one checks
that $C(G^0)\ox_{C(G^0)\ox C(G^0)}\big(C(G^1)\ox C(G^0)\big)_{C(G^0)\ox C(G^0)}$ is isomorphic to
$C(G^1)$ as a linear space. The product module has action and inner product
\begin{equation}
(h\cdot(f_1\ox f_2))(g) = h(g)f_1(s(g))f_2(r(g)), \quad
(h_1\mid h_2)_A(v,w) = \sum_{s(g) = v,\,r(g) = w} \ol{h_1(g)}h_2(g)
\label{eq:s-and-r}
\end{equation}
for all $h, h_1, h_2\in C(G^1)$, all $f_1,f_2\in C(G^0)$ and all $g\in G^1$. Hence 
\begin{equation}\label{eq:firstprodbeta}
\beta\ox_A [E]
    = [(\C,C(G^1)_{{}_sC(G^0)\ox {}_rC(G^0)},0)],
\end{equation}
where we have labelled the actions by range and source to indicate that the right action
is defined as in Equation \eqref{eq:s-and-r}.
A similar line of reasoning shows that
\begin{equation}
\label{eq:secondprodbeta}
\beta\ox_{A^\op}[E^\op]
:= \beta\ox_{A\ox A^\op}([A]\ox[E^\op]) = [(\C,C(G^1)_{{}_sC(G^0)\ox {}_rC(G^0)},0)],
\end{equation}
where the product module has action and inner product
$$
(h\cdot(f_1\ox f_2))(g)=h(g)f_1(s(g))f_2(r(g)), \quad\text{and}\quad
(h_1\mid h_2)_A(v,w)=\sum_{r(g)=w,\,s(g)=v}\ol{h_1(g)}h_2(g)
$$
for all $h, h_1, h_2\in C(G^1)$, all $f_1, f_2\in C(G^0)$, and all $g\in G^1$.
Hence $\beta\ox_AE=\beta\ox_{A^\op}E^\op$.

The analogous identification $V:C(G^0)\ox_{C(G^0)\ox C(G^0)}C(G^0)\ox
\ol{C(G^1)}^\op_{C(G^0)\ox C(G^0)}\to C(G^1)$ is given by  $V(a\ox b\ox\ol{h})(g)
=a(s(g))b(s(g))\ol{h(g)}$. This linear identification of modules shows that the class
$\beta \ox_{A^\op} [\ol{E}^\op] := \beta\ox_{A\ox A^\op}([A]\ox[\ol{E}^\op])$ is
represented by
\begin{equation}
\label{eq:thirdprodbeta}
\beta \ox_{A^\op} [\ol{E}^\op] = [(\C,C(G^1)_{{}_sC(G^0)\ox {}_rC(G^0)},0)],
\end{equation}
where
$$
(h\cdot(f_1\ox f_2))(g)=h(g)f_1(s(g))f_2(r(g))\quad\mbox{and}\quad
(h_1\mid
h_2)_A(v,w)=\sum_{r(g)=w,\,s(g)=v}\ol{h_1(g)}h_2(g).
$$
So $\beta\ox_{A^\op}[\ol{E}^\op]=\beta\ox_A[E]$.
For the $K$-homology statements, we first check,
using the same ideas as above, that
\begin{equation}\label{eq:E ox mu}
[E] \ox_A \mu
    = ([E]\ox[A^\op])\ox_{A\ox A^\op}\mu
    =[(A\ox A^\op,{}_{r\ox s}\ell^2(G^1),0)],
\end{equation}
where
\begin{equation}\label{eq:ip1}
\langle\xi_1,\xi_2\rangle = \sum_{v\in G^0}\sum_{s(g)=v}\ol{\xi_1(g)}\xi_2(g) = \sum_{g \in G^1} \ol{\xi_1(g)}\xi_2(g)
\end{equation}
is the standard $\ell^2$ inner-product, and the actions are given by
\begin{equation}\label{eq:dirac actions}
((a\ox b)\cdot \xi)(g) = a(r(g))b(s(g))\xi(g).
\end{equation}

On the other hand,
\[
[E^\op] \ox_{A^\op} \mu
    = ([A]\ox[E^\op])\ox_{A\ox A^\op}\mu
    = (A\ox A^\op,{}_{r\ox s}\ell^2(G^1),0),
\]
where the actions are given by exactly the same formula as for $[E] \ox_A \mu$, but the
inner product is given by
\begin{equation*}
\langle\xi_1,\xi_2\rangle = \sum_{v\in G^0}\sum_{r(g)=v}\ol{\xi_1(g)}\xi_2(g) = \sum_{g \in G^1} \ol{\xi_1(g)}\xi_2(g),
\end{equation*}
which is~\eqref{eq:ip1}, giving $[E]\ox_A\mu=[E^\op]\ox_{A^\op}\mu$.

For $[\ol{E}^\op]\ox_{A^\op}\mu$ there is a unitary $U:C(G^0)\ox \ol{C(G^1)}^\op
\ox_{A\ox A^\op}\ell^2(G^0)\to \ell^2(G^1)$ such that
$U(f\ox\ol{h}\ox\xi)(g)=f(r(g))\ol{h(g)}\xi(r(g))$ for $g \in G^1$. The resulting unitary
equivalence of Kasparov modules and Equation~\ref{eq:E ox mu} yield
\begin{equation*}
[A]\ox[\ol{E}^\op]\ox_{A\ox A^\op}\mu=[(A\ox A^\op,{}_{r\ox s}\ell^2(G^1),0)] = [E] \ox_A \mu.
\qedhere
\end{equation*}
\end{proof}

\subsubsection{Crossed products}
\label{subsub:cp-CP-commute}

Here we investigate the conditions of Theorem~\ref{thm:(OE)op sufficient} and
Theorem~\ref{thm:(OEop) sufficient} for modules constructed from automorphisms of
$C^*$-algebras. That is, we fix a $C^*$-algebra $A$ satisfying Poincar\'e self-duality
with respect to the classes $\mu$ and $\beta$, and we consider an automorphism $\alpha\in
\Aut(A)$. We put $E := {_\alpha A_A}$, which has right action given by multiplication,
inner-product  $(a \mid b)_A = a^*b$, and left action given by $a \cdot b = \alpha(a)b$.
A  left inner product making $E$ bi-Hilbertian is ${}_A(a\mid b)=\alpha^{-1}(ab^*)$.
These definitions make $E$ an invertible bimodule over $A$, so it yields an invertible
Kasparov class
$$
[E]=[(A,{}_{\alpha}A_A,0)].
$$
The opposite module $E^\op$ has a similar description: we have $E^\op = {}_{\id}
A^\op_{\alpha^\op}$, where $\alpha^\op$ is the automorphism of $A^\op$ implemented by
$\alpha$; that is $\alpha^\op(a^\op) = \alpha(a)^\op$. So the left action of $A^\op$ on
$E^\op$ is left multiplication (in $A^\op$), the right action of $a^\op \in A^\op$ is by
right-multiplication by $\alpha^\op(a)$, and the inner-product is $(e^\op \mid
f^\op)_{A^\op} = (\alpha^{\op})^{-1}(e^{*\op}f^\op) = \alpha^{-1}(fe^*)^\op$. We study
the commutation of the diagram~\eqref{eq:diagram} for $E^\op$.

\begin{lemma}
\label{lem:CP-conditions}
Let $A$ be a $C^*$-algebra, and take $\alpha \in \Aut(A)$. Suppose that $A$ satisfies
Poincar\'e self-duality with respect to the classes $\beta$ and $\mu$. Let $E$ and
$E^\op$ be as described above.
\begin{enumerate}
\item\label{it:autobeta} Suppose that $\beta$ has representative $\beta = [\C,X_{A\ox
    A^\op},T]$ for which there is an invertible $\C$-linear map $V:\,X\to X$ such
    that $V^{-1}(VT-TV)$ is a compact adjointable endomorphism, and such that for all
    $x,\,y\in X$, $a\in A$ and $b^\op\in A^\op$ we have
    \[
    V(x\cdot(a\ox b^\op))=V(x)(\alpha\ox\alpha^{\op})(a\ox b^\op),\quad\text{and}\quad
    (V(x)\mid V(y))_{A\ox A^\op}
    = (\alpha\ox\alpha^{\op})(x\mid y)_{A\ox A^\op}.
        \]
    Then $\beta \otimes_A [E] = \beta \otimes_{A^\op} [E^\op]$.
\item\label{it:automu}
Suppose that $\H$ is a Hilbert space and $\A \subseteq A$ is an $\alpha$-invariant dense
    $^*$-subalgebra.
    Suppose that $\phi : A\to \B(\H)$ and $\psi :
    \A^\op \to \B(\H)$ are $^*$-homomorphisms that determine a spectral triple
    $(\A\ox\A^\op,{}_{\phi\ox\psi}\H,\D)$ representing $\mu
    =[(\A\ox\A^\op,{}_{\phi\ox\psi}\H,\D)]$. Suppose that $W_\alpha \in
    \mathcal{B}(\H)$ is unitary and satisfies $W_\alpha^* \phi(a) W_\alpha =
\phi(\alpha(a))$ and $W_\alpha^* \psi(a^\op) W_\alpha = \psi(\alpha^\op(a^\op))$ for all
$a \in A$. If $[\D, W_\alpha]$ is bounded, then $[E] \otimes_A \mu = [E^\op]
\otimes_{A^\op} \mu$.
\end{enumerate}
\end{lemma}
\begin{proof}
(\ref{it:autobeta}) The map $a^\op\mapsto \alpha^{-1}(a)^\op$ from ${}_{\rm
Id}A^{\op}_{\alpha^\op}$ to ${}_{\alpha^{\op-1}}A^\op_{A^\op}$ is an isomorphism of
invertible bimodules, and so determines a unitary isomorphism of Kasparov modules
$$
(A^\op,{}_{\rm Id}A^{\op}_{\alpha,A^\op},0)\to (A^\op,{}_{(\alpha^{\op})^{-1}}A^\op_{A^\op},0),\qquad
a^\op\mapsto \alpha^{-1}(a)^\op.
$$
Hence
$$
[E]\ox[A^\op]=(A\ox A^\op, {}_{\alpha\ox {\rm Id}}A\ox A^\op,0),\quad\text{ and }\quad
[A]\ox[E^\op]=(A\ox A^\op, {}_{{\rm Id}\ox \alpha^{\op-1}}A\ox A^\op,0),
$$
where the right actions are by multiplication and the inner products are the standard
ones $(a \mid b) = a^*b$. We claim that there is a linear map $\widetilde{V}:\,X\ox_{A\ox
A^\op}(E\ox A^\op)\to X\ox_{A\ox A^\op}(A\ox E^\op)$ such that
\[
\widetilde{V}(x\ox e\ox b^\op)=V(x)\ox e\ox b^\op\quad\text{ for all $x \in X$, $e \in E$ and $b \in A$.}
\]
To see this, fix $y_1, y_2 \in X$, $e_1, e_2 \in E$ and $a_1, a_2 \in A$ and calculate:
\begin{align*}
\big(V(y_1)\ox e_1\ox a_1^\op \mid {}& V(y_2)\ox e_2\ox a_2^\op\big)_{A\ox A^\op} \\
    &=(e_1\ox a_1^\op\mid (V(y_1)\mid V(y_2))_{A\ox A^\op} \cdot e_2\ox a_2^\op)_{A\ox A^\op}\nonumber \\
    &=(e_1\ox a_1^\op\mid (1\ox (\alpha^\op)^{-1})(V(y_1)\mid V(y_2))_{A\ox A^\op} e_2\ox a_2^\op)_{A\ox A^\op}\nonumber\\
    &=(e_1\ox a_1^\op\mid (\alpha\ox 1)(y_1\mid y_2)_{A\ox A^\op} e_2\ox a_2^\op)_{A\ox A^\op}\nonumber\\
    &=(y_1\ox e_1\ox a_1^\op\mid y_2\ox e_2\ox a_2^\op)_{A\ox A^\op}.\label{eq:isom}
\end{align*}
Consequently, given $y_i \in X$, $e_i \in E$ and $a_i \in A$, we have
\begin{align*}
\Big\|\sum_i V(y_i) \ox e_i \ox a_i\Big\|^2
    &= \sum_{i,j} \big(V(y_i) \ox e_i \ox a_i \mid  V(y_j) \ox e_j \ox a_j\big)_{A\ox A^\op} \\
    &= \sum_{i,j} \big(y_i \ox e_i \ox a_i \mid  y_j \ox e_j \ox a_j\big)_{A\ox A^\op} \\
    &= \Big\|\sum_i y_i \ox e_i \ox a_i\Big\|^2.
\end{align*}
Thus there is an isometric linear operator on $\linspan\{x\ox e\ox b^\op : x \in X, e \in
E, b \in A\}$ carrying each $x\ox e\ox b^\op$ to $V(x)\ox e\ox b^\op$, and this extends
to an isometric linear operator $\widetilde{V}$ on $X\ox_{A\ox A^\op}(E\ox A^\op)$. 

Since $V^{-1}(VT-TV)$ is a compact adjointable endomorphism, it is now straightforward to
check that $\widetilde{V}^{-1}(\widetilde{V}(T\ox 1)-(T\ox 1)\widetilde{V})$ is also.
Hence $\widetilde{V}(T\ox 1)\widetilde{V}^{-1}$ is homotopic to $T\ox 1$ via the straight
line path. Thus $(\C,X\ox E\ox A^\op, T\ox 1)$ is unitarily equivalent modulo compact
perturbation to $(\C,X\ox A\ox E^\op, T\ox 1)$, completing the proof of the first
statement.

(\ref{it:automu}) Let $\E\subset E$ and $\E^\op\subset E^\op$ be the submodules $\A$ and
$\A^\op$. Then,
\[
[E] \ox_A \mu
    =  [(\E\ox\A^\op)\ox_{\A\ox\A^\op}(\A\ox\A^\op,{}_{\phi\ox\psi}\H,\D)] = [\A\ox\A^\op,{}_{\phi\circ\alpha\ox\psi}\H,\D],
\]
and
$$
[E^\op] \ox_{A^\op} \mu
    = [(\A\ox\E^\op)\ox_{\A\ox\A^\op}(\A\ox\A^\op,{}_{\phi\ox\psi}\H,\D)]
    = [\A\ox\A^\op,{}_{\phi\ox\psi\circ\alpha^{\op-1}}\H,\D].
$$
Using that $\alpha$ is implemented by $W_\alpha$ and that $[W_\alpha, \D]$ is bounded, we
see that
\[
[(\A\ox\A^\op,{}_{\phi\circ\alpha\ox\psi}\H,\D)]
    = [(\A\ox\A^\op,{}_{\phi\ox\psi}\H,W_\alpha^*\D W_\alpha)]
    = [(\A\ox\A^\op,{}_{\phi\ox\psi}\H,\D+W_\alpha^*[\D, W_\alpha])].
\]
Hence
\[
[E] \ox_A \mu
    = [(\A\ox\A^\op,{}_{\phi\ox\psi}\H,\D)].
\]
A similar computation using that $\alpha^\op$ is also implemented by $W_\alpha$ shows
that
\[
[E^\op] \ox_{A^\op} \mu
    = \big[\A\ox\A^\op,{}_{\phi\ox\psi}\H,\D+W_\alpha[\D, W^*_\alpha]\big]
    = [(\A\ox\A^\op,{}_{\phi\ox\psi}\H,\D)]
    = [E] \ox_A \mu.\qedhere
\]
\end{proof}

\begin{rmk}
If the operator $T$ in the representative $(\C,X,T)$ of $\beta$ in
Lemma~\ref{lem:CP-conditions}(\ref{it:autobeta}) is unbounded and $V^{-1}TV-T$ bounded,
we can replace compact perturbation by bounded perturbation, as we did in the
$K$-homology case. And vice versa.
\end{rmk}

We will show next that the criterion appearing in
Lemma~\ref{lem:CP-conditions}\eqref{it:autobeta} holds for modules of the form ${_\alpha
C(M)}$ where $M$ is a compact Riemannian spin$^c$-manifold and $\alpha$ is 
an automorphism induced by a
spin$^c$-structure-preserving isometry on $M$.

\subsubsection{Spin$^c$ manifolds}
\label{subsub:mflds}

The classical examples of $C^*$-algebras satisfying Poincar\'e self-duality are algebras
of the form $C_0(M)$ where $(M,g)$ is a complete Riemannian spin$^c$ manifold of dimension $d$. Given such a
manifold $(M,g)$ with a fixed spin$^c$ structure, there is a spectral triple
\begin{equation*}
(C^\infty_0(M)\ox C^\infty_0(M),L^2(S,g),\D),
\label{eq:spin-Dirac}
\end{equation*}
where $S$ is the spinor bundle of the 
spin$^c$ structure and $\D$ the Dirac operator.
This spectral triple represents the Dirac class
$\mu \in KK^d(C_0(M)\ox C_0(M),\C)$ 
in a Poincar\'e self-duality for $C_0(M)$. When $M$ is non-compact,
the product with the Dirac class gives an isomorphism
$K_*(C_0(M))\stackrel{\cong}{\to} K^*_c(C_0(M))$, where
$K_c^*$ is compactly supported $K$-homology,
\cite{KasparovEqvar,RennieSmooth}.

Likewise the dual Bott class, described below, is well-defined for
complete spin$^c$ manifolds, giving a class in $KK^d(\C,C_0(M)\ox C_0(M))$. By \cite[Theorem
4.9]{KasparovEqvar}, together with the Morita equivalence \cite{Plymen} between $C_0(M)$
and $\Cliff_0(M)$, the Bott and Dirac classes provide a Poincar\'e duality pair for $C_0(M)$, provided one uses
compactly supported $K$-homology when $M$ is non-compact,
\cite[Corollary 31]{RennieSmooth}. 
For the non-spin$^c$ case, see Appendix \ref{sec:Magnus-made-me-do-it}.

For the $K$-theory fundamental class, we recall the key elements of Kasparov's Bott class from
\cite{KasparovEqvar}. Let $U\subset M\times M$ be a neighbourhood of the diagonal such
that for each $(x,y)\in U$ there is a unique geodesic from $x$ to $y$. Let
$\overrightarrow{xy}$ denote the tangent vector to this geodesic at $x$.

Let $p_2:U\to M$ be the projection on the second factor. Set $X=\Gamma_0(p_2^*S)$, a
(non-full) right $C^*$-module over 
$C_0(M)\otimes C_0(M)$.  There is a choice of numerical function $\rho(x,y)$ of
the distance $d(x,y)$ such that the self-adjoint operator $T\in \End_{C_0(M)\ox C_0(M)}(X)$ defined by
$$
(T\sigma)(x,y)=\rho(x,y)\gamma(\overrightarrow{xy})\sigma(x,y),\qquad \sigma\in L^2(S)
$$
has the property that $T^2-1$ is a compact endomorphism of $X$. 
Then the Bott class
$\beta\in KK^d(\C,C_0(M)\ox C_0(M))$ is represented
by the Kasparov module $(\C,X, T)$.

We consider a module $E=\Gamma_0(M,Z)$ 
of continuous sections vanishing at infinity 
of a locally trivial vector bundle $Z$
over $M$ vanishing at infinity. 
To give $E$ the structure of a bi-Hilbertian 
$C_0(M)$--$C_0(M)$-bimodule we fix a diffeomorphism
$\phi$ of $M$ defining an automorphism 
$\alpha\in \Aut(C_0(M))$ via
$\alpha(f)(x)=f(\phi^{-1}(x))$. 
We define ${_{C_0(M)}}(\cdot \mid \cdot)$ by
$$
{_{C_0(M)}}(e \mid f)(x) =
{_\C}(e(\phi(x)) \mid f(\phi(x)))=\alpha^{-1}((f \mid e)_{C_0(M)})(x)
$$
and we define left and right actions 
by $(a \cdot e\cdot b)(x) = a(\phi^{-1}(x))
e(x)b(x)=\alpha(a)(x)e(x)b(x)$. These definitions yield
$[E]\in KK(C_0(M),C_0(M))$.

\begin{prop}
\label{prop:spin-cee-hyp} Let $(M,g)$ be a 
complete Riemannian spin$^c$ manifold, $\mu$ and $\beta$ the fundamental classes described above, $Z\to
M$ a vector bundle  and $\phi:M\to M$ 
a diffeomorphism with dual
automorphism $\alpha:C_0(M)\to C_0(M)$. If $\phi$ is spin$^c$-structure preserving, then
$[E]\ox_{C_0(M)}\mu=[E^\op]\ox_{C_0(M)}\mu$. 
If $\phi$ is also an isometry then
$\beta\ox_{C_0(M)}[E]=\beta\ox_{C_0(M)}[E^\op]$.
\end{prop}
\begin{proof}
If $\phi$ is spin$^c$-structure preserving, there exists $V:\Gamma_0(S)\to
\Gamma_0(S)$ such that for $f\in C_0(M)$ acting by multiplication we have
$VfV^{-1}\sigma(x)=f(\phi^{-1}(x))\sigma(x)$. For $v\in T^*M$, $\gamma(v\cdot
d\phi^{-1})=V\cdot\gamma(v)\cdot V^{-1}$, where $\gamma$ denotes 
the Clifford action of forms on spinors.
(The scalar ambiguity in this characterisation of $V$ is resolved precisely
by the choice of spin$^c$ structure.)

We can the compute the commutator
$$
[\D,VfV^{-1}]=V[\D,f]V^{-1}+[[\D,V]V^{-1},VfV^{-1}]
=\gamma(df\cdot d\phi^{-1}))+[[\D,V]V^{-1},f\circ\phi^{-1}].
$$
Since $V$ is a smooth map, $[\D,V]V^{-1}$ is at most a first order differential operator,
so $[\D,VfV^{-1}]$ is bounded, and the conditions of Lemma~\ref{lem:CP-conditions}(2)
are satisfied and so $[E]\ox_{C_0(M)}\mu=[E^\op]\ox_{C_0(M)}\mu$.

We now consider the $K$-theory fundamental class.
Given a spin$^c$-structure-preserving diffeomeorphism $\phi$ of $M$, we obtain a lift
$V:X\to X$ of $\phi$ satisfying $(VfV^{-1}\sigma)(x,y)=f\circ\phi^{-1}(x)\sigma(x,y)$ and
$V(\sigma(f\otimes g))(x,y)=(V\sigma)(x,y)f\circ\phi^{-1}(x)g\circ\phi^{-1}(y)$.

If $\phi$ is an isometry, so that the distance and hence also
$\rho$ is invariant, then we can compute the commutator
of $T$ and $V$ as follows: for each section $\sigma$, we have
\begin{align*}
(VT\sigma)(x,y)&=\rho(\phi^{-1}(x),\phi^{-1}(y))\gamma(\overrightarrow{xy}\cdot d\phi^{-1})(V\sigma)(x,y)\\
    &=\rho(x,y)\gamma(\overrightarrow{\phi^{-1}(x)\phi^{-1}(y)})V\sigma(x,y)
     =(TV\sigma)(x,y).
\end{align*}
So  the conditions of Lemma~\ref{lem:CP-conditions}(1) are satisfied
and so $\beta\ox_{C_0(M)}[E]=\beta\ox_{C_0(M)}[E^\op]$.
\end{proof}

\section{The $K$-theory fundamental class}
\label{subsec:K-theory-fun}

In this section we start with a $C^*$-algebra $A$ satisfying Poincar\'e self-duality with
fundamental classes $\beta,\,\mu$ of parity $d$.

We identify hypotheses on a bi-Hilbertian $A$-bimodule $E$ that allow us to apply
Theorem~\ref{thm:(OEop) sufficient} to construct fundamental classes $\delta \in
KK^{d+1}(\C, \CPa_E \ox \CPa_{\ol{E}^\op})$ and $\overline{\delta} \in KK^{d+1}(\C,
\CPa_E \ox \CPa_E^\op)$. In what follows, if $E_b$ is a full Hilbert module
over a unitisation $A_b$ of $A$, then the \emph{restriction} of
$E_b$ to $A$ is defined as $E := E_b \ox_{A_b} A$. This
condition arises as the most useful notion of 
(sections of) non-commutative vector bundles for non-unital algebras, \cite{RSims}.

To obtain the class $\delta$, we assume that:
\begin{enumerate}
\item $E$ is a bi-Hilbertian $A$-bimodule which is the restriction of a module
    $E_b$ over a unitisation $A_b$ which is finitely
    generated as a left and  right module over $A_b$; and
\item $\beta\ox_A[E] = \beta\ox_{A^\op}[E^\op]$.
\end{enumerate}

To obtain the class $\overline{\delta}$, we assume that:
\begin{enumerate}
\item $E$ is an $A$-$A$ correspondence, and is the restriction of a module
    $E_b$ over a unitisation $A_b$ that is finitely
    generated as a  right module over $A_b$ with injective left action of $A_b$; and
\item $\beta\ox_A[E] = \beta\ox_{A^\op}[\ol{E}^\op]$.
\end{enumerate}


\subsection{The $K$-theory fundamental class for $\CPa_{E^\op}$}
\label{subsec:ee-op}

To construct the $K$-theory fundamental class, 
we first need to recall the mapping cone
exact sequence in our setting, and some constructions 
from \cite{AR}.

The inclusion 
$\iota_{A, \CPa_E}:A\hookrightarrow \CPa_E$ 
gives rise to a mapping cone
algebra
$$
M(A,\CPa_E)=\{f \in C_0([0,\infty), \CPa_E) : f(0)\in \iota_{A, \CPa_E} (A)\}.
$$
Write $\S\CPa_E$ for the suspension $C_0((0,\infty))\ox\CPa_E$, and
$j:\,\S\CPa_E\to M(A,O_E)$ for the inclusion.
The evaluation map $\ev:M(A,\CPa_E)\to A$, 
given by $\ev(f)=f(0)$, induces a short exact sequence
\begin{equation}
0\to \S \CPa_E\stackrel{j}{\to}M(A,\CPa_E)\stackrel{\ev}{\to} A\to 0,
\label{eq:map-cone}
\end{equation}
where $j$ is the inclusion of the suspension $\S \CPa_E$ into the mapping cone. Thus we
obtain a $K$-theory exact sequence (see \cite{AR} and \cite[Section~3]{CPR1})
\begin{equation*}
\cdots\to K_1(A)\stackrel{-\iota_{A,\CPa_E*}}{\longrightarrow} K_1(\CPa_E)
\stackrel{j_*}{\longrightarrow} K_0(M(A,\CPa_E))
\stackrel{\ev_*}{\longrightarrow} K_0(A)
\stackrel{-\iota_{A,\CPa_E*}}{\longrightarrow} K_0(\CPa_E)\to\cdots.
\label{eq:m-cone}
\end{equation*}
Elements of $K_0(M(A,\CPa_E))$ can be described as homotopy classes of partial isometries
$v$ over $\widetilde{\CPa}_E$ whose range and source projections $v^*v$ and $vv^*$ are
projections over $\widetilde{A}$, \cite{Putnam}. In this language,
$\ev_*([v])=[v^*v]-[vv^*]$. 


If $E=E_b\ox_{A_b}A$ is the restriction of a module over the minimal
unitisation $A_b$ of $A$, then \cite[Section~6.2]{AR} describes an explicit
Kasparov module representing a class $[W] \in KK(A, M(A, \CPa_E))$. To describe this
representative and its key properties, we fix a frame $(x_j)_{j=1}^k\subset E_b$ (or just $E$ when $A$ is unital), and define $w \in
M_k(\CPa_{E_b})$ by
\begin{equation}
w=
\begin{pmatrix}S_{x_1}^* & 0 & \cdots & 0 & \\ S_{x_2}^* &  0 & \cdots & 0\\
\vdots &  & \ddots & \vdots\\
S_{x_k}^* &0& \cdots & 0\end{pmatrix}.
\label{eq:dubbya-from-S}
\end{equation}
We have $w^*w={\rm Id}_{\mathcal{O}_{\tilde{E}}} \oplus 0_{k-1}
=\iota_{A^\sim,\mathcal{O}_{\tilde{E}}}({\rm Id}_{A^\sim}) \oplus 0_{k-1}$ and 
$$
ww^*=\begin{pmatrix} (x_1|x_1)_{A_b} &(x_1|x_2)_{A_b} & \cdots &(x_1|x_k)_{A_b}\\
(x_2|x_1)_{A_b}  & (x_2|x_2)_{A_b} & \cdots & (x_2|x_k)_{A_b}\\
\vdots & &\ddots  &\vdots\\
(x_k|x_1)_{A_b} & (x_k|x_2)_{A_b} & \cdots & (x_k|x_k)_{A_b}
\end{pmatrix}= (x_i|x_j)_{i,j \geq 1} = : q\in M_{k}(A_b).
$$

%
Then $E_b\cong qA_b^k$, with isomorphism given by $e\mapsto ((x_1|e)_{A_b},\dots,(x_k|e)_{A_b}))^T$.
We can explicitly realise $[w]$ as a difference of classes of projections 
over the minimal unitisation $M(A_b,\CPa_{E_b})^\sim$ of the mapping cone $M(A_b,\CPa_{E_b})$.\footnote{
Since $e_w(\infty)=1_k$, we obtain a class 
in the $KK$ group for $M(A,\CPa_{E})$. See \cite[Corollary 1, Section 7]{KasparovTech}} 
Using \cite{Putnam}, we have an identification of classes $[w]=[e_w]-[1_k]$, where
\begin{equation}
e_w(t)=
\begin{pmatrix} 1_k-\frac{1}{1+t^2}q & \frac{-it}{1+t^2}w\\ 
\frac{it}{1+t^2}w^* & \frac{1}{1+t^2}{\rm Id}_{O_{E_b}}\end{pmatrix}
=
\begin{pmatrix} \frac{1}{1+t^2}(1_k-q)+\frac{t^2}{1+t^2}1_k & \frac{-it}{1+t^2}w\\ 
\frac{it}{1+t^2}w^* & \frac{1}{1+t^2}{\rm Id}_{O_{E_b}}\end{pmatrix}.
\label{eq:ee-dubbya}
\end{equation}

%
%

Then
\[
\varphi(a) \defeq \big((x_i \mid \phi(a)x_j)_{A_b}\big)_{i,j}
\]
defines a left action of $A_b$ on $q(A_b)^k$.
Since $w(0_{k-1}\oplus\phi(a))w^*=\varphi(a)$ and $w^*\varphi(a)w=0_{k-1}\oplus\phi(a)$,
it is straightforward to check that for all $t\in [0,\infty)$
$$
e_w(t)\begin{pmatrix} ((x_i \mid \phi(a)x_j)_{A_b})_{i,j} & 0\\ 0 & \phi(a)\end{pmatrix}
=\begin{pmatrix} ((x_i \mid \phi(a)x_j)_{A_b})_{i,j} & 0\\ 0 & \phi(a)\end{pmatrix}e_w(t)
$$
as operators on 
$\CPa_{\widetilde{E}}^{2k}$ (or $(A_b)^{2k}$ for $t=0$). 
The last ingredient is the unitisation of $M(A,\CPa_E)$
consisting of functions $f:[0,\infty)\to \CPa_E+A_b$ which have a limit at infinity lying in $A_b$,
so
\begin{equation}
M(A,\CPa_E)_b:=\{f:[0,\infty)\to \CPa_E+A_b\,:\,f\ \mbox{continuous},\ f(0)\in A_b,\ \lim_{t\to\infty}f(t)\in A_b\}.
\label{eq:unitise}
\end{equation}
Then,
using the injective and nondegenerate left action of $A_b$ and 
\cite[Lemma 6.1, Lemma 6.2]{AR},
gives a Kasparov class
$$
[W]=\left[\left(A,
\begin{pmatrix} e_w(M(A,\CPa_{E_{A}})_b)^{2k}\\ (M(A,\CPa_{E_{A}})_b)^k\end{pmatrix}, 0\right)\right]
\in KK(A,M(A,\CPa_{E_{A}})).
$$

An important ingredient in the following arguments is a class $\widehat{\extcls}\in
KK(M(A,\CPa_E),A)$ which is  $KK$-inverse to the class $W$, when $A$ belongs to the
bootstrap class. To describe $\widehat{\extcls}$, start from the mapping cone exact
sequence \eqref{eq:map-cone} to obtain the  exact sequence
$$
\cdots\stackrel{{\rm ev}^*}{\to}KK^0(M(A,\CPa_E),A)\stackrel{j^*}{\to}KK^0(\S \CPa_E,A)\stackrel{\partial}{\to}KK^1(A,A)\to\cdots.
$$
In this exact sequence, the boundary map $\partial$ is given (up to sign and Bott
periodicity) by the inclusion $\iota_{A,\CPa}:A\hookrightarrow \CPa_E$,
\cite[Lemma~3.1]{CPR1}. Restricting the extension class $\extcls$ to $A\subset\CPa_E$
gives the zero class in $KK(A,A)$, because the class of the extension $\extcls$
implements the boundary map in the Pimsner exact sequence in $K$-theory. Thus the
boundary map $\partial$ in the mapping cone exact sequence applied to $\extcls$ gives
zero. This implies the existence of a class $\widehat{\extcls}\in KK^0(M(A,\CPa_E),A)$
such that $j^*\widehat{\extcls}=j\ox_M\widehat{\extcls}=\extcls$.

We now recall the key relation between $[W]$ and the $KK$-class $\widehat{\extcls}$
described in \cite{AR}.

\begin{lemma}(\cite[Lemma~6.1]{AR})
\label{lem:pairing} Let $\extcls\in KK^1(\CPa_E,A)=KK(\S \CPa_E,A)$ be the class of the
defining extension for $\CPa_E$ and let $[W]\in KK(A,M(A,\CPa_E))$ be as above. Let
$\widehat{\extcls} \in KK(M(A,\CPa_E),A)$ be a class such that $j^*\widehat{\extcls} =
\extcls$ as above. Then
$$
[W]\ox_{M}\widehat{\extcls}=-\Id_{KK(A,A)}.
$$
\end{lemma}

Let $\mathbb{M}:=M(A\ox A^\op,\CPa_E\ox \CPa_{E^\op})$ be the mapping cone algebra for the
inclusion $A\ox A^\op\hookrightarrow \CPa_E\ox \CPa_{E^\op}$. Using the canonical
identification $\S(\CPa_E\ox\CPa_E^\op) \cong \S\CPa_E\ox\CPa_{E^\op}$, we have an exact
sequence
\begin{equation}
0\to \S\CPa_E\ox\CPa_{E^\op}\stackrel{\mathbf{j}}{\to}\mathbb{M}
\stackrel{\mathbf{ev}}{\to}A\ox A^\op\to 0.
\label{eq:map-map-cone}
\end{equation}
By reasoning similar to that used to define $W$, we obtain a Kasparov module
$$
\mathbb{W}=\Big(A\ox A^\op,
\begin{pmatrix} (e_{w\ox{\rm Id}_{A^\op}})\mathbb{M}_b^{2k}\\
\mathbb{M}_b^k\end{pmatrix}_{\mathbb{M}},0\Big),
$$
where the mapping $\mathbb{M}$ has been unitised by considering
functions $f:[0,\infty)\to \CPa_E\ox\CPa_{E^\op}+A_b\ox A_b^\op$ as in Equation \eqref{eq:unitise}.

If $E$ is the restriction of a bimodule over $A_b$ that is finitely generated
on both sides, then $E^\op$ is similarly a restriction of a bimodule over
$A_b^\op$. So we may apply the discussion above to the module $E^\op$ over
$A^\op$ to obtain a partial isometry $w^\op$ and a class $[W^\op]$, and
\cite[Lemma~6.1]{AR} gives $[W] \ox_{M^\op} \widehat{\extcls}^\op = -\Id_{KK(A^\op,
A^\op)}$. So  we obtain a class
\[
\mathbb{W}^\op = \Big(A \ox A^\op,
\begin{pmatrix} (e_{{\rm Id}_{A}\ox w^\op})\mathbb{M}_b^{2k}\\
\mathbb{M}_b^k\end{pmatrix}_{\mathbb{M}},0\Big).
\]

\begin{defn}
\label{defn:K-fun} Suppose that $E$ is the restriction of a finitely generated
$A_b$-bimodule, and let $\beta$ be a $K$-theory fundamental class in $KK^d(\C,
A \ox A^\op)$. We define
$$
\hat\delta_{E, \beta} \defeq \beta\ox_{A\ox A^\op}\mathbb{W} -\beta\ox_{A\ox A^\op}\mathbb{W}^\op \in KK^{d}(\C,\mathbb{M}).
$$
We will generally suppress the subscripts $E, \beta$ and just denote this class by
$\hat\delta$.
\end{defn}

\begin{lemma}
\label{lem:eval=0} Suppose that $E$ is the restriction of a finitely generated
$A_b$-bimodule. Given a $K$-theory fundamental class $\beta\in KK^d(\C,A\ox
A^\op)$ satisfying $\beta\ox_A[E]=\beta\ox_{A^\op}[E^\op]$, the class $\hat\delta$
satisfies
$$
\hat\delta\ox_{\mathbb{M}}\ev=0.
$$
There exists a class ${\delta}\in K_d(\S \CPa_E\ox \CPa_{E^\op})$ such that
$\hat\delta={\delta}\ox_{\S \CPa_E\ox \CPa_{E^\op}}\mathbf{j}$ where $\mathbf{j}\in KK(\S
\CPa_E\ox \CPa_{E^\op},\mathbb{M})$ is the class of the inclusion.
\end{lemma}
\begin{proof}
The second statement will follow from the first by exactness of the $K$-theory exact sequence.
So we compute
\begin{align*}
\hat\delta\ox_{\mathbb{M}}\ev
&=\beta\ox_{A\ox A^\op}(([A]-[E])\ox[A^\op])-\beta\ox_{A\ox A^\op}([A]\ox([A^\op]-[E^\op]))\\
&=\beta\ox_{A^\op}[E^\op]-\beta\ox_A[E]
=0
\end{align*}
by assumption on the class $\beta$.
\end{proof}

The preceding lemma provides us with the tools we need to check that the product of the
class $\delta\in KK^d(\C,\S \CPa_E\ox \CPa_{E^\op}) = KK^{d+1}(\C,\CPa_E\ox
\CPa_{E^\op})$ with the extension class satisfies the condition appearing in
Theorem~\ref{thm:(OE)op sufficient}(2).

To do this we consider the mapping cone exact sequence~\eqref{eq:map-map-cone} and apply
the same reasoning that we did for the `one-variable' mapping cone sequence
\eqref{eq:map-cone}. This shows that restricting the class $\extcls \ox
\Id_{KK(\CPa_{E^\op}, \CPa_{E^\op})}$ to $A\ox A^\op$ gives the zero class. Hence there
is  a lift  of $\extcls\ox{\rm Id}_{KK(\CPa_{E^\op}, \CPa_{E^\op})}$ in
$KK(\mathbb{M},A\ox \CPa_{E^\op})$. We claim that we can choose a lift
$\widehat{\widehat{\extcls}}$ such that
\begin{equation}\label{eq:intertwines js}
\mathbf{j}\ox_{\mathbb{M}}\widehat{\widehat{\extcls}} = j\ox_M\widehat{\extcls}\ox{\rm Id}_{KK(\CPa_{E^\op},\CPa_{E^\op})}.
\end{equation}
To see this, pick any representative $(M(A,\CPa_E),Y_A,S)$ of the class
$\widehat{\extcls}$ such that the action of $M(A,\CPa_E)$ on $Y_A$ is non-degenerate. We
compute the right hand side of~\eqref{eq:intertwines js} to find
\begin{align*}
j\ox_M\widehat{\extcls}\ox{\rm Id}_{KK(\CPa_{E^\op},\CPa_{E^\op})}
&=(\S \CPa_E,M_M,0)\ox_M(M,Y_A,S)\ox(\CPa_{E^\op},(\CPa_{E^\op})_{\CPa_{E^\op}},0)\\
&=(\S \CPa_E\ox \CPa_{E^\op},Y_A\ox (\CPa_{E^\op})_{\CPa_{E^\op}},S\ox 1)\\
&=(\S \CPa_E\ox \CPa_{E^\op}, \mathbb{M}_{\mathbb{M}},0)\ox_{\mathbb{M}}
(\mathbb{M},Y_A\ox (\CPa_{E^\op})_{\CPa_{E^\op}},S\ox 1).
\end{align*}
Let $\widehat{\widehat{\extcls}}$ be the class of the Kasparov module $(\mathbb{M},Y_A\ox
(\CPa_{E^\op})_{\CPa_{E^\op}},S\ox 1)$. Then
$$
j\ox_M\widehat{\extcls}\ox{\rm Id}_{KK(\CPa_{E^\op},\CPa_{E^\op})}
=\mathbf{j}\ox_{\mathbb{M}}\widehat{\widehat{\extcls}},
$$
and
$$
(\S \CPa_E\ox \CPa_{E^\op}, \mathbb{M}_{\mathbb{M}},0)\ox_{\mathbb{M}}
(\mathbb{M},Y_A\ox (\CPa_{E^\op})_{\CPa_{E^\op}},S\ox 1)
=(\S\CPa_E\ox\CPa_{E^\op}, Y_A\ox(\CPa_{E^\op}),S\ox 1),
$$
which represents $\extcls\ox{\rm Id}_{KK(\CPa_{E^\op},\CPa_{E^\op})}$. So any such
$\widehat{\widehat{\extcls}}$ provides the desired lift.
\begin{rmk}
The class $\delta$ of Lemma \ref{lem:eval=0} need
not be unique. With some additional
regularity properties on $E$ and $\extcls$, 
we can explicitly construct a concrete representative of such a lift: see \cite{AR,CPR1}.
\end{rmk}
\begin{thm}
\label{thm:yeah} Suppose that $E$ is the restriction of a finitely generated
$A_b$-bimodule for some unitisation $A_b$ of $A$. 
Suppose that $\beta$ is a $K$-theory fundamental class in
$KK^d(\C,A\ox A^\op)$, and suppose that $\beta \ox_A [E] = \beta \ox_{A^\op} [E^\op]$.
Let $\delta \in K_d(\S \CPa_E\ox \CPa_{E^\op})$ be the class obtained from
Lemma~\ref{lem:eval=0}. Then
$$
\delta\ox_{\CPa_E}\extcls=-\beta\ox_{A^\op}\iota_{A^\op,\CPa_{E^\op}}\quad\mbox{and}\quad
\delta\ox_{\CPa_{E^\op}}\extocls=\beta\ox_{A}\iota_{A,\CPa_E}.
$$
In particular, $\delta$ defines isomorphisms $K^*(\CPa_E)\to K_{*+d+1}(\CPa_{E^\op})$ and
$K^*(\CPa_{E^\op})\to K_{*+d+1}(\CPa_E)$.
\end{thm}
\begin{proof}
We have
\begin{align*}
\delta\ox_{\S \CPa_E\ox \CPa_{E^\op}}\extcls
&\defeq\delta\ox_{\S \CPa_E\ox \CPa_{E^\op}}(\extcls\ox {\rm Id}_{\CPa_{E^\op}})
=\delta\ox_{\S \CPa_E\ox \CPa_{E^\op}}\big((j\ox_M\widehat\extcls)\ox {\rm Id}_{\CPa_{E^\op}}\big)\\
&=\delta\ox_{\S \CPa_E\ox \CPa_{E^\op}}\mathbf{j}\ox_{\mathbb{M}}\widehat{\widehat{\extcls}}
=\hat\delta\ox_{\mathbb{M}}\widehat{\widehat{\extcls}}.
\end{align*}

Let $X$ be any Stinespring dilation module for the Fock module of $E$, and let $P : X \to
\Fock_E$ denote the projection onto the Fock space. Let $w \in M_k(\CPa_E)$ be as in
\eqref{eq:dubbya-from-S} above. Define
\[
\widetilde{w} \defeq (P\ox{\rm Id}_{\CPa_{E^\op}}\ox1_k)(w\ox{\rm Id}_{\CPa_{E^\op}}) (P\ox{\rm Id}_{\CPa_{E^\op}}\ox 1_k)
    \colon w^*wX^k\ox \CPa_{E^\op} \to ww^*X^k\ox \CPa_{E^\op}.
\]
Regard $\operatorname{Index}(\tilde{w}) = [\ker(\tilde{w})]-[\ker(\tilde{w}^*)]$ as an element of
$KK(A\ox A^\op,A\ox\CPa_{E^\op})$ as in Lemma~\ref{lem:pairing}. Then
\cite[Theorem~2.11]{CNNR} gives
\[
\mathbb{W}\ox_{\mathbb{M}}\widehat{\widehat{\extcls}}
    = -\operatorname{Index}(\widetilde{w}),
\]
So, just as in Lemma~\ref{lem:pairing}, we have
\[
\mathbb{W}\ox_{\mathbb{M}}\widehat{\widehat{\extcls}}
 =-(A\ox A^\op,(A\ox \CPa_{E^\op})_{A\ox \CPa_{E^\op}},0)
 =-{\rm Id}_{KK(A,A)}\ox\iota_{A^\op,\CPa_{E^\op}}.
\]
The product $\mathbb{W}^\op\ox_\mathbb{M}\widehat{\widehat{\extcls}}$ is zero, because the
restriction of $\widehat{\widehat{\extcls}}$ to $A\ox \CPa_{E^\op}$ is zero. Thus
$$
\delta\ox_{\CPa_E}\extcls=- \beta\ox_{A^\op}\iota_{A^\op,\CPa_{E^\op}}.
$$
An analogous argument gives
$
\delta\ox_{\CPa_{E^\op}}\extocls= \beta\ox_A\iota_{A,\CPa_E}.
$
\end{proof}

\subsection{The $K$-theory fundamental class for $\CPa_{E}^\op$.}
\label{subsec:bar-ee-op}

Suppose that $E$ is the restriction to $A$ 
of a finitely generated right $A_b$-module
$E_b$ with injective and nondegenerate left action by $A_b$,
and $\beta\ox_A[E]=\beta\ox_{A^\op}[\ol{E}^\op]$. 
Then $\ol{E}^\op$ is
the restriction of a finitely generated right $A^\op$ module, and so we can produce a
class $\ol{\delta}$ analogous to $\delta$. The difference between this construction and
the one in the preceding section is illustrated by the following three lemmas. The first
is standard: we include it for completeness since we need an explicit description of the
isomorphism.

\begin{lemma}
\label{lem:K-op} 
For any $C^*$-algebra $B$ there is an isomorphism $K_0(B)\cong
K_0(B^\op)$. When $B$ is unital and $p^\op\in M_n(B^\op)$ is a projection, the
isomorphism sends $[p^\op]$ to $[p^T]$, where we now regard the entries of $p$ as
elements of $B$.
\end{lemma}
\begin{proof}
It suffices to prove the lemma for unital algebras. Given a finitely generated and
projective right $B^\op$ module $p^\op(B^\op)^n$ (thought of as columns) we obtain a
finitely generated and projective left $B^\op$ module $(B^\op)^n(p^\op)^T$, thought of as
rows.

A finitely generated and projective left $B^\op$ module is the same thing as a
finitely generated and projective right $B$ module, namely $p^TB^n$.

As this construction is plainly symmetric in $B$ and $B^\op$, we obtain the stated
isomorphism.
\end{proof}

\begin{lemma}
\label{lem:minus-dubbya} Use Lemma \ref{lem:E-barE} to identify $\CPa_E^\op$ and $\CPa_{\ol{E}^\op}$. The isomorphism
$$
K_0(M(A^\op,\CPa_E^\op))=K_0(M(A,\CPa_E)^\op)\cong K_0(M(A,\CPa_E))
$$
of Lemma~\ref{lem:K-op} carries the class of
$$
w_{\ol{E}^\op}=
\begin{pmatrix}S_{\ol{x_1}^\op}^* & 0 & \cdots & 0 & \\ S_{\ol{x_2}^\op}^* &  0 & \cdots & 0\\
\vdots &  & \ddots & \vdots\\
S_{\ol{x_k}^\op}^* &0& \cdots & 0\end{pmatrix}
$$
to the class of
$$
w^*_E=
\begin{pmatrix}S_{x_1} & S_{x_2} & \cdots & S_{x_k} \\ 0 &  0 & \cdots & 0\\
\vdots &  & \ddots & \vdots\\
0 &0& \cdots & 0\end{pmatrix}.
$$
\end{lemma}
\begin{proof}
The isomorphism $\CPa_{\ol{E}^\op}\to \CPa_E^\op$
of Lemma \ref{lem:E-barE} is given on generators by
$S_{\ol{x}}\mapsto S_x^{*\op}$. Thus
$$
w_{\ol{E}^\op}=
\begin{pmatrix}S_{\ol{x_1}^\op}^* & 0 & \cdots & 0 & \\ S_{\ol{x_2}^\op}^* &  0 & \cdots & 0\\
\vdots &  & \ddots & \vdots\\
S_{\ol{x_k}^\op}^* &0& \cdots & 0\end{pmatrix}
\mapsto
\begin{pmatrix}S_{x_1}^{\op} & 0 & \cdots & 0 & \\ S_{x_2}^\op &  0 & \cdots & 0\\
\vdots &  & \ddots & \vdots\\
S_{x_k}^\op &0& \cdots & 0\end{pmatrix}
\mapsto
\begin{pmatrix}S_{x_1} & S_{x_2} & \cdots & S_{x_k} \\ 0 &  0 & \cdots & 0\\
\vdots &  & \ddots & \vdots\\
0 &0& \cdots & 0\end{pmatrix}
=w_E^*,
$$
where the second isomorphism is the isomorphism of Lemma~\ref{lem:K-op}.
\end{proof}
We can now define analogues of the classes $W^\op,\,\mathbb{W}^\op$ and so forth using
$\ol{E}^\op$ in place of $E^\op$. We denote the resulting classes and their
representatives by $\ol{W}^\op,\,\ol{\mathbb{W}}^\op$ and so on.

An argument similar to Lemma~\ref{lem:eval=0}, but using Lemma~\ref{lem:minus-dubbya},
proves the following lemma.

\begin{lemma}
\label{lem:eval=0-op} Suppose that $E$ is the restriction of a finitely generated right
$A_b$-module $E_b$ to $A$. Let $\mathbf{j}\in KK(\S \CPa_E\ox
\CPa_{E}^\op,\mathbb{M})$ be the class of the inclusion. 
Given a $K$-theory fundamental
class $\beta\in KK^d(\C,A\ox A^\op)$ satisfying
$\beta\ox_A[E]=\beta\ox_{A^\op}[\ol{E}^\op]$, the class $\hat{\ol{\delta}}$ defined by
$$
\hat{\ol{\delta}}=\beta_{A\ox A^\op}\mathbb{W}+\beta\ox_{A\ox A^\op}\ol{\mathbb{W}}^\op
$$
satisfies
$$
\hat{\ol{\delta}}\ox_{\mathbb{M}}\ev=0.
$$
Hence there exists a class $\ol{\delta}\in K_d(\S \CPa_E\ox \CPa_{E}^\op)$ such that
$\hat{\ol{\delta}}=\ol{\delta}\ox_{\S \CPa_E\ox \CPa_{E}^\op}\mathbf{j}$.
\end{lemma}

So, just as before, we obtain a $K$-theory fundamental class.

\begin{thm}
\label{thm:yeah-op} Suppose that $E$ is the restriction of a finitely generated right
$A_b$-module $E_b$ to $A$. 
Given a $K$-theory fundamental class $\beta\in
KK^d(\C,A\ox A^\op)$ satisfying $\beta\ox_A[E]=\beta\ox_{A^\op}[\ol{E}^\op]$,
$$
\ol{\delta}\ox_{\CPa_E}\extcls=-\beta\ox_{A^\op}\iota_{A^\op,\CPa_E^\op}\quad\mbox{and}\quad
\ol{\delta}\ox_{\CPa_E^\op}\extobarcls=-\beta\ox_{A}\iota_{A,\CPa_E}.
$$
Consequently $\ol{\delta}$ defines isomorphisms $K^*(\CPa_E)\to K_{*+d+1}(\CPa_E^\op)$
and $K^*(\CPa_E^\op)\to K_{*+d+1}(\CPa_E)$.
\end{thm}

\subsection{The $K$-theory classes for an invertible bimodule}
\label{subsec:ee-smeb}

Comparing Lemma~\ref{lem:eval=0} with Lemma~\ref{lem:eval=0-op}, we see a discrepancy of
sign between the definitions of $\hat{\delta}$ and $\hat{\ol{\delta}}$. 
The following proposition reconciles
this difference in the context of our Poincar\'e duality 
pairings in the situation of
invertible bimodules $E$.

In this section, we write $\mathscr{O}_E$ for the dense $*$-subalgebra of $\CPa_E$ generated
by $A$ and the elements $\{S_e : e \in E\}$, so
\begin{equation}\label{eq:scriptO}\textstyle
\mathscr{O}_E = \linspan\big\{S_\eta S^*_\zeta : \eta,\zeta \in \bigcup_{n \ge 0} E^{\ox n}\big\}.
\end{equation}

\begin{prop}[\cite{RRSext}]
\label{prop:minus-enn} Suppose that $E$ is an invertible bimodule. Let $N$ be the
densely-defined number operator on $(\Fock_{E,\Z})_A$ such that $N \rho := n \rho$ and $N
\overline{\rho} = -n \overline{\rho}$ for $n \ge 0$ and $\rho \in E^{\ox n}$. Then
$\extcls\in KK^1(\CPa_E,A)$ has representative
$$
(\mathscr{O}_E,(\Fock_{E,\Z})_A,N).
$$
Similarly, the class $\extocls\ox_{\End^\op}[\Fock_{E^\op}]\in KK^1(\CPa_{E^\op},A^\op)$ is
represented by
$$
(\mathscr{O}_{E^\op},(F_{{E},\Z})_{A^\op},N)=(\mathscr{O}_E^\op, (\Fock_{E,\Z})_{A^\op},N).
$$
The class $\extobarcls\in KK^1(\CPa_E^\op,A^\op)$ has representative
$$
(\mathscr{O}_{\ol{E}^\op},(\Fock_{\ol{E}^\op,\Z})_{A^\op},N_{\ol{E}})
=(\mathscr{O}_E^\op, (\Fock_{E,\Z})_{A^\op},-N)
$$
where we identify algebras by
$\CPa_{\ol{E}^\op}\ni S_{\bar{e}}\mapsto S_e^{*\op}\in \CPa_E^\op$,
and we do nothing else except notice that $(\Fock_{\ol{E},\Z})_{A^\op}=(\Fock_{E,\Z})_{A^\op}$
and that $N_{\ol{E}}$ acts on $(\Fock_{E,\Z})_{A^\op}$ as $-N$.
\end{prop}

When $E$ is an invertible bimodule, we can identify $\CPa_{E^\op}\cong \CPa_E^\op$. So in
this situation we can compare the different extension classes, and so obtain a
relationship between $\delta$ and $\ol{\delta}$:

\begin{corl}
\label{cor:all-my-ops} Let $E$ be an invertible bimodule
which is the restriction to $A$ of an (invertible) $A_b$-bimodule 
$E_b$ for some unitisation $A_b$ of $A$. 
Identifying $\CPa_{E^\op}$ with
$\CPa_E^\op$ via the isomorphism of Proposition~\ref{prop:op-is-op}, we have
$$
\extocls=-\extobarcls\in KK^1(\CPa_E^\op,A^\op)
$$
and
$$
(\delta-\ol{\delta})\ox_{\CPa_E^\op}\extocls
=(\delta-\ol{\delta})\ox_{\CPa_E^\op}\extobarcls
=(\delta-\ol{\delta})\ox_{\CPa_E}\extcls=0.
$$
\end{corl}
\begin{proof}
Proposition~\ref{prop:minus-enn} gives the first statement. We know that
\begin{align*}
\mathbb{W}^\op\ox_{\CPa_E^\op}\extocls&=-{\rm Id}_{KK(A^\op,A^\op)}
\quad\mbox{and}\quad
\ol{\mathbb{W}}^\op\ox_{\CPa_E^\op}\extobarcls=-{\rm Id}_{KK(A^\op,A^\op)}.
\end{align*}
Using Proposition~\ref{prop:minus-enn} again yields
\begin{align*}
\mathbb{W}^\op\ox_{\CPa_E^\op}\extobarcls&={\rm Id}_{KK(A^\op,A^\op)}
\quad\mbox{and}\quad
\ol{\mathbb{W}}^\op\ox_{\CPa_E^\op}\extocls={\rm Id}_{KK(A^\op,A^\op)}.
\end{align*}
Since
$$
\delta-\ol{\delta}=-\beta\ox_{A\ox A^\op}(\mathbb{W}^\op+\ol{\mathbb{W}}^\op)
$$
the result follows from associativity of the Kasparov product.
\end{proof}

\subsection{Examples}
\label{subsub:egs} In the following examples, we construct explicit representatives
of the $K$-theory fundamental class. 

\subsubsection{The circle}
\label{ex:cee-to-circle} The simplest example is when $A=E=\C$ as algebra and bimodule.
Here clearly $\beta=[1_\C\ox 1_\C]$ is $E$-invariant. So we recover the fact that
$C(S^1)=\CPa_E$ satisfies Poincar\'e duality in the $K$-theory sense. Since $E$ is an
invertible bimodule, we can realise $\CPa_E$ as shift operators on
$F_{E,\Z}=\ell^2(\Z)\cong L^2(S^1)$. Since $E$ is singly generated by $1_\C$, the partial
isometry $w$ is the unitary given by the bilateral shift (i.e. multiplication by $z$).
Following the recipe for constructing $\delta$ we obtain
$$
\delta_{C(S^1)}=[z\ox 1_{C(S^1)}]-[1_{C(S^1)}\ox z]
=[z]\ox[\iota_{\C,C(S^1)}]-[\iota_{\C,C(S^1)}]\ox[z]\in KK^1(\C,C(S^1)\ox C(S^1)).
$$
\subsubsection{The rotation algebras}
\label{ex:circle-to-Atheta} 
Let $A=C(S^1)$ and let $E$ be the bimodule implementing the
automorphism of rotation by an angle $\theta$, defined as in
Subsection~\ref{subsub:cp-CP-commute}. That is, define 
$\alpha:\,C(S^1)\to C(S^1)$ by $\alpha(a)(e^{i\phi})=a(e^{i(\phi+\theta)})$,
and then define
$$
E=A,\qquad a\cdot e\cdot b\defeq\alpha(a)eb, \quad\mbox{and}\quad  a,\,b\in A,\ e\in E.
$$
Then $\CPa_E$ is  isomorphic to the rotation algebra $A_\theta$. Since
$\beta\defeq[z\ox 1_{C(S^1)}]-[1_{C(S^1)}\ox z]\in KK^1(C(S^1)\ox C(S^1),\C)$ satisfies
$\beta\ox_{C(S^1)}[E]=\beta\ox_{C(S^1)}[E^\op]=\beta$ (the class of the unitary $z$ is
invariant under rotations), we obtain a $K$-theory fundamental class 
$\delta$ for $A_\theta$ as follows.

Since $E$ is an invertible bimodule, $\CPa_E$ can be realised as the algebra generated by
the shift operators on $F_{E,\Z}=\ell^2(\Z)\ox C(S^1)$. Specifically, writing $\{e_n : n
\in \Z\}$ for the standard orthonormal basis of $\ell^2(\Z)$,
$$
a\cdot(e_n\ox b)=e_n\ox \alpha^n(a)b\quad\mbox{for all } a,\,b\in C(S^1).
$$
As $E$ is  generated by $1_{C(S^1)}$,  the partial isometry $w$ is the unitary
given by the bilateral shift (which we continue to denote by $w$). This $w$ does not
commute with the left action of $z\in C(S^1)$: we have $wz=e^{-i\theta}zw$.

To determine $\delta$,
we follow the recipe of
Section~\ref{sec:K-fun} (remembering the antisymmetry of the external product) and obtain
\begin{align*}
\delta_{A_\theta}&=([z\ox 1_{C(S^1)^\op}]-[1_{C(S^1)}\ox z^\op])\ox_{C(S^1)\ox C(S^1)^\op}[\mathbb{W}]\\
&\qquad\qquad-([z\ox 1_{C(S^1)^\op}]-[1_{C(S^1)}\ox z^\op])\ox_{C(S^1)\ox C(S^1)^\op}[\mathbb{W}^\op]\\
&=[z]\ox_{C(S^1)}[W]\ox_\C[\iota_{\C,A_\theta^\op}]+[w]\ox_\C[z^\op]
-[z]\ox_\C[w^\op]-[\iota_{\C,A_\theta}]\ox_\C[z^\op]\ox_{C(S^1)^\op}[W^\op].
\end{align*}
Up to sign, this expression agrees with the Bott element identified by
Connes in \cite{ConnesGrav}, and  identifies the class of the Powers-Rieffel projector
with $[z]\ox_{C(S^1)}[W]\in K_0(A_\theta)$.

\subsubsection{Automorphisms}
\label{ex:automorphisms}

More generally, given an algebra $A$ with Bott class $\beta$ that is invariant under an
automorphism $\alpha \in \Aut(A)$ in the sense that $\beta\ox_A {_\alpha A} =
\beta\ox_{A^\op} ({_\alpha A})^\op$, we obtain a Bott class $\delta$ for
$A\rtimes_\alpha\Z$. This applies in particular to isometric actions of $\Z$ on compact
spin$^c$ manifolds, and more generally when we have a Bott class satisfying the
conditions of Lemma~\ref{lem:CP-conditions}(1).

Let $U\in A\rtimes_\alpha\Z$ be the unitary implementing  $\alpha$. The
projection $e_w$ of Equation~\eqref{eq:ee-dubbya} is given by
$$
e_w(t)=\frac{1}{1+t^2}\begin{pmatrix} t^2 &-itU\\ itU^* & 1\end{pmatrix}.
$$
So we obtain an explicit representative of $\delta$ from an explicit representative of
$\beta$.

\subsubsection{Cuntz--Krieger algebras and graph algebras}\label{ex:cuntz-krieger}
Consider a finite directed graph $G$.
Suppose that there is at most one
(directed) edge between any pair of vertices if $G$. This is a constraint on the graphs
we consider, but not the algebras: replacing a graph with its dual graph does not change
the graph $C^*$-algebra, and the dual graph has at most one edge between any two
vertices, \cite{Raeburn}.

Let $E$ be the graph bimodule of $G$ as in Section~\ref{subsub:findim-alg}. We proved
there that the diagram~\ref{eq:diagram} for $E^\op$ commutes. Here we compute $\delta$ by
taking the products of $\beta=(\C,C_0(G^0),0)$ with 
the classes $\mathbb{W}$ and $\mathbb{W}^\op$
described in Definition~\ref{defn:K-fun} and Lemma~\ref{lem:eval=0}.


Enumerate $G^1 = g_1, \dots, g_{|G^1|}$. For $i \le |G^1|$,  write $e_i \in
C(G^1)$ for the point-mass function.

Let $e_{w\ox 1}$ be the projection over the mapping cone for $C(G^0)\ox
C(G^0)\hookrightarrow C^*(G)\ox C^*(G)$ determined by Equation~\eqref{eq:ee-dubbya} from
the partial isometry $w$ over $C^*(G)$ in Equation~\eqref{eq:dubbya-from-S}. Using the
Cuntz-Krieger relations it is not hard to show that $(p_v\ox p_v^\op) e_{w\ox 1}$ is the
projection $e_V$ associated as in Equation~\eqref{eq:ee-dubbya} to the partial isometry
$$
V=\sum_{s(e_i)=v}E_{i1}S_{g_i}^*\ox p_v^\op,
$$
where the $E_{i1}$ are matrix units. With this  observation and  $\beta=\sum_{v\in
G^0}p_v\ox p_v^\op$ one checks that
$$
\beta\ox_{A\ox A^\op}\mathbb{W}=\sum_{v\in G^0}\sum_{s(g_i)=v}[E_{i1}S_{e_i}^*\ox p_v^\op]
\in KK(\C,M(A\ox A^\op,\CPa_E\ox \CPa_{E^\op})).
$$
Similarly\footnote{This identification of left and right frames does not hold in general,
and relies heavily on the orthonormality of the frames in this example}, writing
$f_j^\op$ for the point mass function $\delta_j \in C(G^1)$ regarded as an element of
$E^\op$, we compute
$$
\beta\ox_{A\ox A^\op}\mathbb{W}^\op
=\sum_{v\in G^0}\sum_{r(g_i)=v}[p_v\ox E_{i1}S_{f_i^\op}^*]
\in KK(\C,M(A\ox A^\op,\CPa_E\ox \CPa_{E^\op})).
$$
We have
$$
\Big(\sum_{s(g_j)=w}E_{1j}S_{e_j}\ox p_w^\op\Big)\Big(\sum_{s(g_i)=v}E_{i1}S_{e_i}^*\ox p_v^\op\Big)
=p_v\ox p_v^\op
$$
and
$$
\Big(\sum_{r(g_j)=w}p_w\ox E_{1j}S_{f_j^\op}\Big)\Big(\sum_{r(g_i)=v}p_v\ox E_{i1}S_{f_i^\op}^*\Big)
=p_v\ox p_v^\op.
$$
So we can use \cite[Lemma~3.3]{CPR1} to compute that
\begin{align*}
\beta\ox_{A\ox A^\op}\mathbb{W}-\beta\ox_{A\ox A^\op}\mathbb{W}^\op
&=\sum_{v,\,w\in G^0}
\Big(\sum_{s(g_i)=w}[E_{i1}S_{e_i}^*\ox p_w^\op]-\sum_{r(g_j)=v}[p_v\ox E_{j1}S_{f_j^\op}^*]\Big)\nonumber\\
&=\sum_{v,\,w\in G^0}
\sum_{s(g_i)=w}\sum_{r(g_j)=v}[(E_{i1}S_{e_i}^*\ox p_w^\op)(p_v\ox E_{1j}S_{f_j^\op})]\nonumber\\
&=\sum_{v,\,w\in G^0}
\sum_{s(g_i)=w}\sum_{r(g_j)=v}[E_{ij}S_{e_i}^*\ox S_{f_j^\op}]\delta_{v,r(e_i)}\delta_{w,s(f_j)}\nonumber\\
&=\sum_{g_j\in G^1}[E_{jj}S_{e_j}^*\ox S_{f_j^\op}]
\quad\quad\text{ as each $|vG^1w| \le 1$}\nonumber\\
&=\sum_{g_j\in G^1}[S_{e_j}^*\ox S_{f_j^\op}].
\label{eq:Kam-Put}
\end{align*}
Identifying $E^\op$ with the edge module of the opposite graph, we deduce that
the class constructed in Lemma~\ref{lem:eval=0} recovers the $K$-theory Poincar\'e
duality class of Kaminker and Putnam \cite{KP}, but
for any  finite graph with at most one edge between any two vertices.

In the non-unital case, Kajiwara, Pinzari 
and Watatani \cite[Section~6.1]{KajPinWat}
showed that a left inner product can be defined 
on $C_0(G^1)$ making $C_0(G^1)$ bi-Hilbertian precisely when 
the in- and out-valences of the graph are uniformly bounded. 
The requirement that $C_0(G^1)$ is the restriction of
a finitely generated module over $C_b(G^1)$ is proved in \cite{AR}.
The construction of the classes $W$ and so $\delta$ extend to this
generality, but the construction of the $K$-homology 
fundamental class discussed in the next section does not immediately
extend.

%

\section{Examples of $K$-homology fundamental classes}
\label{sec:K-fun}

We have been unable to identify a general procedure for lifting $K$-homology fundamental
classes from $A$ to $\CPa_E$. For the special cases of crossed products by $\Z$ and
graph algebras, we \emph{can} produce the required $K$-homology class; but the
procedure in each case is ad hoc.

\subsection{Crossed products by $\Z$}
\label{subsec:cpi}

For this subsection, we suppose that  $\alpha:A\to A$ is an automorphism, and that $\mu$
is a $K$-homology fundamental class for $A$ such that $\mu$ and $\alpha$ satisfy the
conditions of Lemma~\ref{lem:CP-conditions}(2). Thus $\mu$ is represented by a spectral
triple $(\A\ox\A^\op,{}_\pi\H,\D)$ with $\ol{\A\H}=\ol{\A^\op\H}=\H$, both $\alpha$ and
$\alpha^\op$ preserve the subalgebras $\A$ and $\A^\op$, and $\alpha$ is implemented on
$\H$ by a unitary $W_\alpha$ such that $[\D,W_\alpha]$ is bounded. The main constructions of
this section do not require $A$ to be unital.

\begin{thm}
\label{thm:cp-fun} Take $A$, $\alpha$, and $\mu=[(\A\ox\A^\op,{}_\pi\H,\D)]\in KK^d(A\ox
A^\op,\C)$ as above. Let $E := {_\alpha A_A}$. Write $S_1 \in \CPa_E$ for the generator
corresponding to $1_A$ regarded as an element of $E$. Let $\mathscr{O}_E \subseteq
\CPa_E$ be the subalgebra described at~\eqref{eq:scriptO}, and similarly for
$\mathscr{O}_{E^\op}$. Let $U\in \B(\oplus_{n\in\Z}\H)$ be the shift,
$(U\xi)_n=\xi_{n+1}$. There is a representation $\tilde{\pi}$ of
$\mathscr{O}_E\ox\mathscr{O}_E^\op$ on $\bigoplus_{n\in\Z}\H$ such that for all $a,b \in
\A$ and $\xi \in \bigoplus_{n \in \Z} \H$, we have
\begin{equation}
\begin{split}
(\tilde{\pi}(a\ox b^\op&)\xi)_n = \pi(\alpha^n(a)\ox b^\op)\xi, \qquad
(\tilde{\pi}(S_1\otimes 1)\xi)_n = (U\xi)_n=\xi_{n+1},\text{ and}\\
&(\tilde{\pi}(1\otimes S_{1^\op})\xi)_n=(U^{-1}W_\alpha\xi)_n=(W_\alpha U^{-1}\xi)_n=W_\alpha\xi_{n-1}.
\end{split}
\label{eq:two-ways}
\end{equation}
Let $N : \bigoplus_{n \in \Z} \H \to \bigoplus_{n \in \Z} \H$ be the densely defined
number operator $(N \xi)_n = n \xi_n$. If $d$ is even, write $\H_+$ and $\H_-$ for the
even and odd subspaces of $\H$ so that $\H=\begin{pmatrix} \H_+\\ \H_-\end{pmatrix}$ and
$\D=\begin{pmatrix} 0 & \D_-\\ \D_+ & 0\end{pmatrix}$. Define
\[
\Delta_0 :=
    \begin{cases}
    \bigg(\mathscr{O}_E\ox\mathscr{O}_E^\op, \bigoplus_{n\in\Z}(\H\ox\C^2), \bigg(\begin{matrix} 0 & N-i\D\\ N+i\D & 0\end{matrix}\bigg)\bigg)
        &\text{ if $d$ is odd}\\
    \bigg(\mathscr{O}_E\ox\mathscr{O}_E^\op, \bigoplus_{n\in\Z}\bigg(\begin{matrix} \H_+\\ \H_-\end{matrix}\bigg),\bigg(\begin{matrix} N & \D_-\\ \D_+ & -N \end{matrix}\bigg)\bigg),
        &\text{ if $d$ is even.}
    \end{cases}
\]
If both $[\D,\alpha^n(a)]$ and $[\D,\alpha^n(a)^\op]$ are uniformly norm-bounded in $n$,
then $\Delta_0$ is an unbounded Kasparov module. 
If in addition the operators $W_\alpha^n[\D,W_\alpha^{-n}](\D\pm i)^{-1}$
are uniformly bounded in $n$, the class $\Delta \in
KK^{d+1}(\CPa_E \otimes \CPa_E^\op, \C)$ that $\Delta_0$ defines is a $K$-homology
fundamental class.
\end{thm}
\begin{proof}
The universal properties of $\CPa_E$ and $\CPa_E^\op$, together with that of the tensor
product, show that there is a representation of $\CPa_E \ox \CPa_E^\op$ whose restriction
to $\mathscr{O}_E\ox\mathscr{O}_E^\op$ satisfies the desired formulas.

We will first construct the product of the extension class
and the $K$-homology fundamental class for $A$. We only present the argument for $d$ odd, as the case for $d$ even is similar.
The extension class $\extcls$ is represented by
$(\CPa_E,\oplus_{n\in\Z}A,N)$
\cite[Theorem~3.1]{RRSext}, as described in Proposition
\ref{prop:minus-enn}.

The internal product of $\oplus_{n\in\Z}A$ with $\H$ is just $\oplus_{n\in\Z}\H$ because
the action is non-degenerate. Since both $\extcls$ and $\mu$ are odd, we need to double
both of them to even classes using the Clifford algebra $\Cliff_1$. We omit the details
but refer to  \cite[Appendix]{BCR} for the mechanics and determination of signs.

Abusing notation slightly, we write $N$ for $N \ox 1$ on $\bigoplus_{n \in \Z} A \ox_A
\H$ and we write $\D$ for the operator $\bigoplus_{n \in \Z} \D$ on the same space. We
will show that
\begin{equation}\label{eq:O spec trip}
\left(\mathscr{O}_E\ox\A^\op, \bigoplus_{n\in\Z}(\H\ox\C^2),\begin{pmatrix} 0 & N-i\D\\ N+i\D & 0\end{pmatrix}\right)
\end{equation}
is a spectral triple representing the Kasparov product $\extcls\ox_A\mu$.
To see this, first note that the operator
\[
N\#\D := \begin{pmatrix} 0 & N-i\D\\ N+i\D & 0\end{pmatrix}
\]
is self-adjoint by \cite[Proposition~3.12, Theorem~3.18 and Lemma~4.2]{MR}, and has
locally compact resolvent by \cite[Theorem~6.7]{KaadLesch}. 

By assumption we have uniform boundedness of 
$[\D,\alpha^n(a)]$ and $[\D,\alpha^n(a)^\op]$. For commutators with the other generators of $\mathscr{O}_E\ox\mathscr{O}_E^\op$
we just need to recall that $[\D,W_\alpha]$ is assumed to be bounded,
and observe that
$$
[\D,U]=0,\quad [N,U]=U,\quad\mbox{and}\quad [N,W_\alpha]=0.
$$
Hence $N\#\D$ has
bounded commutators not only with elements of $\mathscr{O}_E\ox\A^\op$,
but also with all of $\mathscr{O}_E\ox\mathscr{O}_E^\op$. (Similar but more general conclusions were
reached in \cite[Lemma 3.4, Proposition 3.5]{BKM}.)

We deduce that $[N\#\D, \mathscr{O}_E\ox\A^\op] \subseteq \B(\oplus_{n \in \Z} \H)$, and
so~\eqref{eq:O spec trip} is indeed a spectral triple. Theorem~4.4 of \cite{MR} shows
that this triple represents $\extcls \ox_A \mu$.
Using the fact that we have bounded commutators with all of $\mathscr{O}_E\ox\mathscr{O}_E^\op$, we obtain a spectral triple $\Delta_0$ for
$\mathscr{O}_E\ox\mathscr{O}_E^\op$, whose class we denote by $\Delta$. Plainly
$\iota_{A^\op,\CPa_E^\op}\ox_{\CPa_E^\op}\Delta$ coincides with $\extcls\ox_A\mu$.


For the final statement set
$V=\oplus_{n\in\Z}W_\alpha^{-n}$,
and assume that $V[\D,V^*](\D\pm i)^{-1}$ is bounded. 
We need to show that $\iota_{A,\CPa_E}\ox_{\CPa_E}\Delta = \extocls \otimes_{A^\op} \mu$.
Since $V$ is unitary,
 $\Delta$ is
represented by the spectral triple
$$
\left(\mathscr{O}_E\ox\mathscr{O}_{E^\op},{}_{V\tilde{\pi}(\cdot)V^*}(\oplus_n\H\ox\C^2),
\begin{pmatrix} 0 & N-iV\D V^*\\N+iV\D V^* & 0\end{pmatrix}\right)
$$
obtained from $\Delta_0$ by unitary equivalence and by identifying $\Oo_E^\op$ with
$\Oo_{E^\op}$ using Proposition~\ref{prop:op-is-op}.

We have $[V,N]=0$, and $V[\D,V^*](\D\pm i)^{-1}$
is bounded by assumption. A simple calculation shows
that the straight-line path
$\D_t:=\D+tV[\D,V^*]$ between $\D$ and $V\D V^*$
is graph norm differentiable \cite[Definition 6]{CarPotSuk}, and so
$\D_t(1+\D_t^2)^{-1/2}$ is an operator homotopy by \cite[Theorem 20]{CarPotSuk}. We deduce that
$$
\Delta = \left[\left(\mathscr{O}_E\ox\mathscr{O}_{E^\op},{}_{V\tilde{\pi}(\cdot)V^*}(\oplus_n\H\ox\C^2),
\begin{pmatrix} 0 & N-i\D \\N+i\D & 0\end{pmatrix}\right) \right].
$$
For $a,b \in \A$ and $\xi \in \bigoplus_{n \in \Z} \H$, the representation
$V\tilde{\pi}(\cdot)V^*$ satisfies
\begin{gather*}
(V\tilde{\pi}(a\ox b^\op)V^*\xi)_n=\pi(a\ox \alpha^{-n}(b)^\op)\xi,\qquad
    (\tilde{\pi}(1\otimes S_{1^\op})\xi)_n=(U^{-1}\xi)_n=\xi_{n-1},\text{ and}\\
(\tilde{\pi}(S_1\otimes 1)\xi)_n=(UW_\alpha\xi)_n=(W_\alpha U\xi)_n=W_\alpha\xi_{n+1}.
\end{gather*}
Lemma~\ref{lem:CP-conditions} gives $E^\op={}_{\alpha^{\op-1}}A^\op_{A^\op}$, and so,
just as above, we obtain
$$
\iota_{A,\CPa_E}\ox_{\CPa_E}\Delta=\extocls\ox_{A^\op}\mu.
$$
This is equal to $-\extobarcls\ox_{A^\op}\mu$ by Proposition~\ref{prop:minus-enn}.
\end{proof}

Given a Riemannian manifold $M$, by an \emph{almost-isometry} on $M$ we mean a diffeomorphism $\phi : M \to M$
such that for every $f \in C^\infty_c(M)$, we have the differentials of $f\circ\phi^k$ uniformly bounded, so
$\sup_k \|d(f \circ \phi^k)\| <
\infty$. Given an almost-isometry $\phi$ of a manifold $M$, we can construct a 
$K$-homology fundamental class for the crossed product $C_0(M)\rtimes_\alpha\Z$ of $C_0(M)$ by the automorphism
$\alpha$ dual to $\phi$.

\begin{corl} 
\label{cor:the-point}
Let $(M,g)$ be a complete Riemannian spin$^c$ manifold. Suppose that $\phi : M \to M$ is
a  spin$^c$-structure-preserving almost-isometry. Define $\alpha : C^\infty_0(M)
\to C^\infty_0(M)$ by $\alpha(f) = f \circ \phi$. Then 
there exists a class $\Delta\in KK^{d+1}(C_0(M)\rtimes_\alpha\Z\ox (C_0(M)\rtimes_\alpha\Z)^\op,\C)$
satisfying part 1. of Theorem \ref{thm:(OEop) sufficient}.

If $M$ is compact, then $\Delta$ is a
$K$-homology fundamental class for
$C(M)\rtimes_\alpha\Z$ represented by the spectral triple $\Delta_0$. In particular,
\begin{align*}
K_*(C(M)\rtimes_\alpha\Z) &\cong K^{*+\dim(M)+1}((C(M)\rtimes_\alpha\Z)^\op)
\text{ and}\\
K_*((C(M)\rtimes_\alpha\Z)^\op) &\cong K^{*+\dim(M)+1}(C(M)\rtimes_\alpha\Z).
\end{align*}
If $\phi$ is an isometry and $M$ is compact then $C(M)\rtimes_\alpha\Z$ is Poincar\'e self-dual.
\end{corl}
\begin{proof}

The first statement is a direct consequence of Theorem~\ref{thm:cp-fun}
and the uniform boundedness of the differentials $d(\phi^k)$. If $\phi$ is an
isometry and $M$ is compact then the discussion of the example of subsection~\ref{ex:automorphisms}
combined with Proposition~\ref{prop:spin-cee-hyp} gives a $K$-theory fundamental class,
and then Theorem~\ref{thm:(OE)op sufficient}(3) gives a Poincar\'e self-duality.
\end{proof}

\begin{corl} 
\label{cor:the-point-theta}
Let $(M_\theta,g)$ be a $\theta$-deformation of
a complete Riemannian spin$^c$ manifold. 
Suppose that  $\alpha : C^\infty_0(M_\theta)
\to C^\infty_0(M_\theta)$ is an automorphism unitarily implemented
on $L^2(S_\theta)$ and commuting with the Dirac operator. Then 
there exists a class $\Delta\in KK^{d+1}(C_0(M_\theta)\rtimes_\alpha\Z\ox (C_0(M_\theta)\rtimes_\alpha\Z)^\op,\C)$
satisfying part 1. of Theorem \ref{thm:(OEop) sufficient}.

If $M$ is compact, then $\Delta$ is a
$K$-homology fundamental class for
$C_0(M_\theta)\rtimes_\alpha\Z$. In particular,
\begin{align*}
K_*(C(M_\theta)\rtimes_\alpha\Z) &\cong K^{*+\dim(M_\theta)+1}((C(M_\theta)\rtimes_\alpha\Z)^\op)
\text{ and}\\
K_*((C(M_\theta)\rtimes_\alpha\Z)^\op) &\cong K^{*+\dim(M_\theta)+1}(C(M_\theta)\rtimes_\alpha\Z).
\end{align*}
If $\alpha$ is unitarily implemented and $M$ is compact then $C(M_\theta)\rtimes_\alpha\Z$ is Poincar\'e self-dual.
\end{corl}
\begin{proof}
The $KK$-equivalence $C(M_\theta)\sim_{KK}C(M)$, \cite{Yamashita}, gives us fundamental classes for
$C(M_\theta)$.
\end{proof}

\begin{rmk}
An analogous construction can be given when $(M,g)$ is oriented but not necessarily
spin$^c$. This construction, which uses Kasparov's fundamental class
\cite{KasparovEqvar}, starts from a Poincar\'e duality between $C(M)$ and $\Cliff(M)$ and
produces a Poincar\'e duality between $C(M)\rtimes_\alpha\Z$ and
$\Cliff(M)\rtimes_\alpha\Z$. See Appendix~\ref{sec:Magnus-made-me-do-it}.
\end{rmk}
\begin{rmk}
The important point in the above constructions is that we have an
explicit representative $\Delta_0$ of the fundamental class. Likewise
representatives of the dual $K$-theory class can be obtained from the
Bott class and the isometry.
\end{rmk}

\subsection{Cuntz--Krieger algebras and graph algebras} For unital Cuntz--Krieger algebras we
have fairly complete information. Let $E$ be the edge module for a finite graph with no sources nor sinks. When the graph algebra has associated shift space a Cantor set, Kaminker and Putnam provided an
extension representing a $K$-homology 
fundamental class $\Delta$ relating $\Oo_E$ and
$\Oo_{E^\op}$, \cite{KP}. Goffeng and Mesland provided a Kasparov module
representing this extension in \cite{GM}.

We show how  Goffeng and Mesland's construction fits our framework. Given further assumptions we produce
a $K$-homology fundamental class relating $\Oo_E$ and $\Oo_E^\op$, and so deduce a
$KK$-equivalence between $\Oo_{E^\op}$ and $\Oo_E^\op$.

\subsubsection{The $K$-homology fundamental class for $C^*(G^\op)$}
\label{subsub:GMI}

Fix a finite directed graph $G=(G^0,G^1,r,s)$ 
with no sources and no sinks, 
let $A = C(G^0)$, and write $E$ for the
Cuntz--Krieger module $C(G^1)$ of $G$. We regard $E$ as a bi-Hilbertian bimodule with
\[
(e \mid f)_A(v) = \sum_{r(g) = v} \overline{e(g)}f(g)
    \quad\text{ and }\quad
{_A(e \mid f)}(v) = \sum_{s(g) = v} e(g)\overline{f(g)}.
\]
We write $\phi : A \to \C$ for the functional $\phi(a) = \sum_{v \in G^0} a_v$. Observe that
\begin{equation}\label{eq:mu for CK}
    \mu=[(A\ox A^\op, L^2(A,\phi),0)]
\end{equation}
is a $K$-homology fundamental class for $A$ by Remark~\ref{rmk:Cr PD}. The functional
$\phi$ is invariant for $E$ in the sense that $\phi((e \mid f)_A) = \phi({_A(e \mid f)})$
for all $e,f \in E$, cf \cite[Section 4]{RRSext}.

Since $\phi$ is faithful, we can form the Hilbert space $L^2(\Fock_E,\phi)$ with inner
product $\langle \xi,\eta\rangle=\phi((\xi\mid \eta)_A)$. On this Hilbert space we can
define operators $L_e,\,R_e$ for $e\in E$ via the formulae
$$
L_e\xi=e\ox_A \xi,\quad R_e\xi=\xi\ox_Ae.
$$
We have $[L_e,R_f]=0$ for all $e,\,f\in E$. For $b\in L^2(A,\phi)$ we have
$[L_e,R_f^*]b=-{}_A(eb\mid f)P_0$, where $P_0$ is the projection onto  the degree zero
part $A \subseteq \Fock_E$ of the Fock space. Since $A$ has finite linear dimension, we
obtain commuting representations of $\CPa_E$ and $\CPa_{E^\op}$ modulo compacts, and
hence an extension
\begin{equation}
0
\to \K(L^2(\Fock_E,\phi))\to C^*(L,R)\to
\CPa_E\ox \CPa_{E^\op}\to 0.
\label{eq:KP-ext}
\end{equation}
This extension is almost the extension used by Kaminker and Putnam to define the
$K$-homology fundamental class: the difference is that Kaminker and Putnam's extension
replaces $L^2(A,\phi)$ in degree zero by $\C$. (For us, $L^2(A, \phi) = \ell^2(G^0) =
C(G^0)$.)

A Fredholm module representing Kaminker and Putnam's extension appears in \cite{GM}. Our
class is based directly on their idea. In the situation of Cuntz--Krieger algebras,
Goffeng and Mesland's Fredholm module is constructed from the Hilbert space
$L^2(\CPa_E)\ox L^2(\CPa_{E^\op})$---where the GNS $L^2$-spaces are defined with respect
to the unique KMS-states \cite{EFW} for the gauge actions---and the natural left action.
There is an isometric embedding of $L^2(\Fock_E)$ into $L^2(\CPa_E)\ox
L^2(\CPa_{E^\op})$. Writing $P$ for the projection onto the image of this embedding, the
Fredholm operator appearing in Goffeng and Mesland's Fredholm module is $2P-1$.

To adapt Goffeng and Mesland's approach to our setting, we require two pieces of
information from \cite{RRSext}. One is the construction of representatives of the classes
$\extcls$ and $\extocls$. The other is an expectation $\CPa_E\to A$ which, when composed
with the $E$-invariant functional $\phi:A\to\C$, yields a KMS functional on $\CPa_E$.

In order to obtain these ingredients, we make an assumption on the asymptotic
behaviour of the Watatani indices of the tensor powers of $E$. When $E$ is the module
associated to an irreducible matrix that is not a permutation matrix, our assumption is automatically satisfied,
\cite[Example~3.8]{RRSext}.

Recall that if $(a_n)$ and $(b_n)$ are sequences of positive real numbers, we say that
$(a_n) \in O(b_n)$ if there exist $N \in \N$ and $C \ge 0$ such that $a_n \le C b_n$ for
all $n \ge N$.

\begin{ass}
\label{ass:one} For each $k \ge 0$, let $\mathrm{e}^{\beta_k}$ denote the right Watatani
index of $E^{\ox k}$. We assume that for every $k\in \N$, there is a $\delta>0$ such that
for each $\nu\in E^{\otimes k}$ there exists $\tilde{\nu}\in E^{\otimes k}$ satisfying
$$
\big(\Vert \mathrm{e}^{-\beta_n}\nu \mathrm{e}^{\beta_{n-k}}-\tilde\nu\Vert\big)_{n=1}^\infty
\in O(n^{-\delta}).
$$
\end{ass}
\begin{defn}
\label{defn:xi-phi-infty}
Given Assumption~\ref{ass:one}, \cite[Proposition~3.5]{RRSext} describes an expectation
$\Phi_\infty:\CPa_E\to A$. We define $\Xi_E$ to be the completion of $\CPa_E$ in the norm
determined by the inner product $(a \mid b)_A:=\Phi_\infty(a^*b)$. When
Assumption~\ref{ass:one} holds for $E^\op$, we obtain the analogous module $\Xi_{E^\op}$.
\end{defn}

The delta functions $\delta_\alpha$ on paths $\alpha\in G^k$ of length $k$ constitute a
frame for $E^{\otimes k}$. For $\alpha \in G^k$, as a notational convenience, we will
write $S_\alpha$ for the element $S_{\delta_\alpha} \in \Oo_E$. We have $\CPa_E =
\ol{\linspan}\{S_\alpha S^*_\beta : s(\alpha) = s(\beta)\}$. We denote the image of
$S_\alpha S_\beta^*$ in $\Xi_E$ by $W_{\alpha,\beta}$. As a notational convenience, we
interchangeably write $W_{\alpha, 1}$ or $W_{\alpha, s(\alpha)}$ for the image of
$S_\alpha$ and $W_{1, \alpha}$ or $W_{s(\alpha), \alpha}$ for the image of $S^*_\alpha$.
We also denote the set of finite paths in the graph $G$ by $G^*=\bigcup_kG^k$. 

The opposite module $E^\op$ coincides with the Cuntz--Krieger module of the opposite
graph $G^\op$, which has vertex set $(G^\op)^0 = G^0$, edges $(G^\op)^1 = \{g^\op : g \in
G^1\}$ and range and source maps $r(g^\op) = s(g)$ and $s(g^\op) = r(g)$. If $\alpha =
\alpha_1 \cdots \alpha_k \in G^k$ then $\alpha^\op = \alpha^\op_k \cdots \alpha^\op_1 \in
(G^\op)^k$, so we obtain elements $\{W_{\alpha^\op, \beta^\op} : \alpha,\beta \in E^*,
r(\alpha) = r(\beta)\} \in \Xi_{E^\op}$.

By \cite[Theorem~3.14]{RRSext}, the subspace $\overline{{\rm span}}\{W_{\mu,1}:\,\mu\in
G^*\}$ is isometrically isomorphic to $\Fock_E$ and is complemented in $\Xi_E$. We write
$P_{\Fock_E}$ for the projection on this subspace. We have
\begin{equation}
\extcls=[(\CPa_E,\Xi_E,2P_{\Fock_E}-1)]\in KK^1(\CPa_E,A).
\label{eq:ext-Xi}
\end{equation}
This is the representative of $\extcls$ that we require.

By \cite[Proposition~4.7]{RRSext},  $\phi\circ\Phi_\infty:\CPa_E\to\C$ is a KMS$_1$ functional for the
dynamics $S_e\mapsto e^{i\beta t}S_e$, $e\in E,\,t\in\R$.

\begin{ntn}
The following notation will prove very helpful throughout our calculations below. Let $G$
be a directed graph. Given $\lambda = \lambda_1 \cdots \lambda_m \in G^m$, and $0 \le j
\le m$, we define
\[
\underline{\lambda}_j := \begin{cases}
        \lambda_1\cdots \lambda_j &\text{ if $j \ge 1$} \\
        r(\lambda) &\text{ if $j = 0$}
    \end{cases}
    \quad \mbox{and} \quad
\overline{\lambda}^{\op (m-j)} := \begin{cases}
        \lambda^\op_m \cdots \lambda^\op_{j+1} &\text{ if $j < m$} \\
        s(\lambda)^\op &\text{ if $j = m$.}
        \end{cases}
\]
Equivalently, if $\lambda = \alpha\beta$ is the unique factorisation with $|\alpha| = j$
and $|\beta| = |\lambda|-j$, then $\underline{\lambda}_j = \alpha$ and
$\overline{\lambda}^{\op (m-j)} = \beta^\op$.
\end{ntn}

\begin{lemma}[{see also \cite[Section~2.3]{GM}}]
\label{lem:fock-isom}
Let $G$ be a finite directed graph with no sources
and no sinks, 
let $A = C(G^0)$ and
let $E$ be the associated edge module, regarded as a finitely generated
bi-Hilbertian $A$-bimodule. Assume that $E$ satisfies Assumption~\ref{ass:one}. There is
an isometry 
$$
\V : \,L^2(\Fock_E,\phi) \to \Xi_E\ox\Xi_{E^\op}\ox_{A\ox A^\op}L^2(G^0,\phi)
$$
such that, for $\lambda = \lambda_1\cdots \lambda_m\in G^m$, we have
$$
\V\delta_\lambda=\sum_{j=0}^m \frac{1}{\sqrt{m+1}} W_{\underline{\lambda}_j,1}\ox W_{\overline{\lambda}^{\op
(m-j)},1}\ox_{A\ox A^\op} \delta_{s(\lambda_j)}.
$$
For $\alpha,\beta,\mu,\nu \in G^*$ with $s(\alpha) = s(\beta)$ and $r(\mu) = r(\nu)$ and
$x \in S$, we have
\begin{align*}
\V^*&(W_{\alpha,\beta}\ox W_{\mu^\op,\nu^\op}\ox x) \\
    &=\begin{cases}
        \delta_{r(\beta), s(\nu)} \frac{x_{r(\beta)}}{\sqrt{|\alpha|-|\beta| + |\mu| - |\nu| + 1}}
            \delta_{\underline{\alpha}} \cdot \Phi_\infty(S_\beta S^*_\beta)\Phi^\op_\infty(S_{\nu^\op} S^*_{\nu^\op}) \cdot \delta_{\overline{\mu}}
            &\text{if $\alpha=\underline{\alpha}\beta$ and $\mu = \nu\overline{\mu}$} \\
        0 &\text{otherwise,}
    \end{cases}
\end{align*}
where $\Phi_\infty$ is as in Definition \ref{defn:xi-phi-infty}.
\end{lemma}
\begin{proof}
Fix $\lambda = \lambda_1 \cdots \lambda_m$ and $\mu = \mu_1 \cdots \mu_n \in G^*$ and
calculate
\begin{align*}
\sum_{j,k}&\frac{1}{\sqrt{m+1}}\frac{1}{\sqrt{n+1}}\langle W_{\underline{\lambda}_j,1}\ox W_{\overline{\lambda}^{\op (m-j)},1}\ox \delta_{s(\lambda_j)},
        W_{\underline{\mu}_k,1}\ox W_{\overline{\mu}^{\op (n-k)},1}\ox \delta_{s(\mu_k)} \rangle\\
&=\delta_{m,n}\sum_{j,k} \frac{1}{m+1} \langle \delta_{s(\lambda_j)}, ((\underline{\lambda}_j\mid\underline{\mu}_k)_A\ox
                                                (\overline{\lambda}^{\op(m-j)}\mid \overline{\mu}^{\op (m-k)})_{A^\op})\cdot \delta_{s(\mu_k)}\rangle\\
&=\delta_{m,n} \sum_j \frac{1}{m+1}\phi((\delta_{s(\lambda_j)} \mid
                                        (\underline{\lambda}_j\mid\underline{\mu}_j)_A \delta_{s(\mu_j)} {}_A(\overline{\mu}^{(m-j)}\mid \overline{\lambda}^{(m-j)}))_A)\\
&= \begin{cases}
    \sum^m_{j=0} \frac{1}{m+1} \phi(\delta_{s(\lambda_j)} \mid \delta_{s(\lambda_j)}) &\text{ if $\lambda = \mu$}\\
    0 &\text{ otherwise}
    \end{cases}\\
&= \begin{cases}
    1 &\text{ if $\lambda = \mu$}\\
    0 &\text{ otherwise}
    \end{cases}\\
&= \langle \delta_\lambda, \delta_\mu\rangle_{L^2(\Fock_E)}.
\end{align*}
So there is an isometry $\V$ satisfying the desired formula. Now fix $\lambda,
\alpha,\beta,\mu,\nu \in G^*$ with $s(\alpha) = s(\beta)$ and $s(\mu) = s(\nu)$ and $x
\in S$. Put $l = |\lambda|$, and calculate
\begin{align*}
\big\langle \V& \delta_\lambda, W_{\alpha,\beta} \otimes W_{\mu^\op, \nu^\op} \otimes_{A \otimes A^\op} x\big\rangle\\
    &= \sum^l_{j=0} \frac{1}{\sqrt{l+1}} \langle W_{\underline{\lambda}_j,1} \otimes W_{\overline{\lambda}^{\op(l-j)}, 1} \otimes \delta_{s(\lambda_j)},
            W_{\alpha,\beta} \otimes W_{\mu^\op, \nu^\op} \otimes_{A \otimes A^\op} x\rangle\\
    &= \sum^l_{j=0} \frac{1}{\sqrt{l+1}} \phi\Big(\Big(\delta_{s(\lambda_j)} \mathbin{\Big|}
                    \big(\big(W_{\underline{\lambda}_j,1} \mathbin{\big|} W_{\alpha,\beta}\big)_A \otimes
                        \big(W_{\overline{\lambda}^{\op(l-j)},1} \mathbin{\big|} W_{\mu^\op, \nu^\op}\big)_{A^\op}\big) \cdot x\Big)\Big)\\
    &= \sum^l_{j=0} \frac{1}{\sqrt{l+1}} \phi\big(\big(\delta_{s(\lambda_j)} \mathbin{\big|}
                    \Phi_\infty\big(S_{\underline{\lambda}_j}^*S_\alpha S^*_\beta\big) \cdot x
                        \cdot \Phi^\op_\infty\big(S^*_{\overline{\lambda}^{\op(l-j)}} S_{\mu^\op} S^*_{\nu^\op}\big)\big)\big)\\
    &= \sum^l_{j=0} \frac{1}{\sqrt{l+1}} \Phi_\infty\big(S_{\underline{\lambda}_j}^*S_\alpha S^*_\beta\big)(s(\lambda_j))
            \phi\big(\big(\delta_{s(\lambda_j)} \mathbin{\big|} x\big)\big)
                        \Phi^\op_\infty\big(S^*_{\overline{\lambda}^{\op(l-j)}}S_{\mu^\op} S^*_{\nu^\op}\big)(s(\lambda_j)).
\end{align*}
This is nonzero only if $\alpha = \underline{\alpha}\beta$ and $\mu = \nu \overline{\mu}$
and $\lambda = \underline{\alpha}\overline{\mu}$, in which case $r(\beta) =
s(\underline{\alpha}) = r(\overline{\mu}) = s(\nu)$, and the final line of the preceding
calculation collapses to
\[
\frac{1}{\sqrt{l+1}} \Phi_\infty\big(S_\beta S^*_\beta\big)(r(\beta))
            x_{r(\beta)} \Phi^\op_\infty\big(S_{\nu^\op} S^*_{\nu^\op}\big)(r(\beta)).
\]
This is precisely the inner-product of $\delta_\lambda$ with the right-hand side of the
expression given for $\V^*(W_{\alpha,\beta} \otimes W_{\mu^\op, \nu^\op} \otimes x)$.
\end{proof}

\begin{prop}
\label{prop:Kas-mod} Let $G$ be a finite directed graph
with no sources and no sinks, let $A = C(G^0)$ and let $E$ be
the associated edge module which we assume satisfies Assumption \ref{ass:one}, 
regarded as a finitely generated bi-Hilbertian
$A$-bimodule. Let $W$ be as in Lemma~\ref{lem:fock-isom}, and let $P := \V\V^*$. Then
$$
\left(\CPa_E\ox\CPa_{E^\op}, \Xi_E \ox \Xi_{E^\op}, 2P-1\right)
$$
defines a Kasparov module, whose class $\Delta$ represents the
extension~\eqref{eq:KP-ext}.
\end{prop}
\begin{proof}
It suffices to show that $[P, a]$ is compact for every $a \in \CPa_E \ox \CPa_{E^\op}$:
since $P = P^*$ and $(2P - 1)^2 = 1$, the result will then follow from standard
Busby-invariant arguments.

We compute for generators $S_e\ox 1$ of $\CPa_E$, where $e \in G^1$. Together with the
Leibniz rule, an identical calculation for $1\ox S_{f^\op}$, and routine approximation,
this will suffice to show that $P$ has compact commutators with $\CPa_E\ox\CPa_{E^\op}$.

We have
\[
\V\V^*(S_e \ox 1) - (S_e \ox 1)\V\V^* = (\V\V^*(S_e \ox 1) \V\V^* - (S_e \ox 1)\V\V^*) + \V\V^*(S_e \ox 1)(1- \V\V^*),
\]
so it suffices to show that each of $(\V\V^*(S_e \ox 1) \V\V^* - (S_e \ox 1)\V\V^*)$ and
$\V\V^*(S_e \ox 1)(1- \V\V^*)$ is compact. For the former, we fix $\lambda \in G^*$, say
$|\lambda| = l$, so that $\V\delta_\lambda$ is a typical spanning element of $\V\V^*
\Xi$. If $s(e)\not=r(\lambda)$, then $(S_e \ox 1) \V\delta_\lambda = 0$, and so
$(\V\V^*(S_e \ox 1) \V\V^* - (S_e \ox 1)\V\V^*)\V\delta_\lambda = 0$. If $s(e) =
r(\lambda)$, then
\begin{align*}
\big(\V\V^*&(S_e \ox 1)\V\V^* - (S_e \ox 1)\V\V^*\big)\V\delta_\lambda
    = \big((\V\V^* - 1)(S_e \ox 1)\big) \sum^l_{j=0} \frac{1}{\sqrt{l+1}} W_{\underline{\lambda}_j, 1} \ox W_{\overline{\lambda}^{\op(l-j)},1} \ox \delta_{s(\underline{\lambda}_j)} \\
    &= (\V\V^* - 1) \sum^l_{j=0} \frac{1}{\sqrt{l+1}} W_{e\underline{\lambda}_j, 1} \ox W_{\overline{\lambda}^{\op(l-j)},1} \ox \delta_{s(\underline{\lambda}_j)} \\
    &= (\V\V^* - 1) \frac{1}{\sqrt{l+1}}
        \big(\sqrt{l+2}\V \delta_{e\lambda} - W_{r(e), 1} \ox W_{(e\lambda)^\op, 1} \ox \delta_{r(e)}\big)\\
    &= \frac{-(\V\V^* - 1)}{\sqrt{l+1}} (W_{r(e), 1} \ox W_{(e\lambda)^\op, 1} \ox \delta_{r(e)}) 
    = \frac{-1}{\sqrt{l+1}} \big(\V \frac{1}{l+2} \delta_{e\lambda} - W_{r(e), 1} \ox W_{(e\lambda)^\op, 1} \ox \delta_{r(e)}\big) \\
    &= \frac{-1}{\sqrt{l+1}} \Big(\Big(\frac{1}{l+2} \sum^{l+1}_{j=0} W_{\underline{e\lambda}_j, 1} \ox W_{\overline{e\lambda}^{\op(l+1-j)},1} \ox \delta_{s(\underline{e\lambda}_j)}\Big)
                                    - W_{r(e), 1} \ox W_{(e\lambda)^\op, 1} \ox \delta_{r(e)}\Big) \\
    &= \frac{1}{\sqrt{l+1}} W_{r(e), 1} \ox W_{(e\lambda)^\op, 1} \ox \delta_{r(e)}
            - \frac{1}{\sqrt{l+1}(l+2)} \sum^{l+1}_{j=0} W_{\underline{e\lambda}_j, 1} \ox W_{\overline{e\lambda}^{\op(l+1-j)},1} \ox \delta_{s(\underline{e\lambda}_j)}\\
    &= \frac{\sqrt{l+1}}{l+2} W_{r(e), 1} \ox W_{(e\lambda)^\op, 1} \ox \delta_{r(e)}
            - \frac{1}{\sqrt{l+1}(l+2)} \sum^{l+1}_{j=1} W_{\underline{e\lambda}_j, 1} \ox W_{\overline{e\lambda}^{\op(l+1-j)},1} \ox \delta_{s(\underline{e\lambda}_j)}.
\end{align*}
The terms in this sum are mutually orthogonal, and $\|W_{\underline{e\lambda}_j, 1} \ox
W_{\overline{e\lambda}^{\op(l+1-j)},1} \ox \delta_{s(\underline{e\lambda}_j)}\| = 1$ for
each $j$. So we have
\[
\|\big(\V\V^*(S_e \ox 1) - (S_e \ox 1)\V\V^*\big)\V\delta_\lambda\| \sim \sqrt{\frac{l+1}{(l+2)^2} + \frac{l^2}{(l+1)(l+2)^2}} \sim \sqrt{\frac{2}{l}}.
\]
Moreover, if $|\lambda'| = |\lambda|$, then the terms in the above sum for $\lambda$ are
all orthogonal to all of the terms in the corresponding sum for $\lambda'$. Thus
$\big\|\big(\V\V^*(S_e \ox 1) - (S_e \ox 1)\V\V^*\big)|_{\linspan\{\V\delta_\lambda :
|\lambda| = l\}}\big\| \sim \sqrt{\frac{2}{l}}$. The subspaces $Y_l :=
\linspan\{\V\delta_\lambda : |\lambda| = l\}$ are all finite-dimensional. Writing $P_l$
for the projection onto $Y_l$, we see that
\[
\|(\V\V^*(S_e \ox 1) \V\V^* - (S_e \ox 1)\V\V^*) (1 - P_l)\| \sim \sqrt{\frac{2}{l}} \to 0.
\]
So $(\V\V^*(S_e \ox 1) \V\V^* - (S_e \ox 1)\V\V^*) = \lim_l (\V\V^*(S_e \ox 1) \V\V^* -
(S_e \ox 1)\V\V^*)P_l$ is compact.

Now we must compute $\V\V^*(S_e \ox 1)(1- \V\V^*)$.

Fix paths $\alpha,\beta,\mu,\nu \in G^*$ with $s(\alpha) = s(\beta)$ and $r(\mu) = r(\nu)$ and
$x \in C(G^0)$. We compute
$\V\V^*(S_e \ox 1)(1- \V\V^*) (W_{\alpha,\beta} \ox
W_{\mu^\op,\nu^\op} \ox x)$. If $s(e) \not= r(\alpha)$ then both $\V\V^* (S_e \ox
1)(W_{\alpha,\beta}\ox W_{\mu^\op, \nu^\op} \ox x)$ and $(S_e \ox 1) \V\V^*
(W_{\alpha,\beta}\ox W_{\mu^\op, \nu^\op} \ox x)$ are zero, so we can assume that
$r(\alpha) = s(e)$.

Similarly, both terms are zero unless $(e\alpha) = \zeta\beta$ for some $\zeta$ and $\mu
= \nu\tau$, and $s(\zeta) = r(\tau)$. So we suppose that this is also the case. Two cases
remain: either $\alpha = \eta\beta$ for some $\eta$, or $e\alpha = \beta$.

First suppose that $\alpha = \eta\beta$. Let $l := |\alpha| - |\beta|
+ |\mu| - |\nu|$, and let $\lambda := \eta\tau$. Then
\begin{align}
\big(\V\V^*(S_e\ox 1)\big)&(W_{\alpha,\beta}\ox W_{\mu^\op, \nu^\op} \ox x) \nonumber\\
    &= \V\V^*(W_{e\alpha,\beta}\ox W_{\mu^\op, \nu^\op} \ox x) \nonumber\\
    &= \frac{x_{r(\beta)} (\Phi_\infty(S_\beta S^*_\beta)\Phi^\op_\infty(S_{\nu^\op} S^*_{\nu^\op}))(r(\beta))}{\sqrt{|e\eta| + |\tau| + 1}}
        \V \delta_{e\eta\tau}  \nonumber\\
    &= \frac{x_{r(\beta)} (\Phi_\infty(S_\beta S^*_\beta)\Phi^\op_\infty(S_{\nu^\op} S^*_{\nu^\op}))(r(\beta))}{l + 2}
        \sum^{l+1}_{j=0} W_{\underline{e\lambda}_j,1} \otimes W_{\overline{e\lambda}^{\op(l-j)},1} \otimes \delta_{s(e\underline{\lambda}_j)}.
        \label{eq:VV*Seox1}
\end{align}
On the other hand,
\begin{align*}
\big(\V\V^*&(S_e\ox 1)\V\V^*\big)(W_{\alpha,\beta}\ox W_{\mu^\op, \nu^\op} \ox x)\\
    &= \V\V^*(S_e\ox 1) \frac{x_{r(\beta)} (\Phi_\infty(S_\beta S^*_\beta)\Phi^\op_\infty(S_{\nu^\op} S^*_{\nu^\op}))(r(\beta))}{\sqrt{|\eta| + |\tau| + 1}}
        \V \delta_{\eta\tau} \\
    &= \V\V^*(S_e \ox 1) \frac{x_{r(\beta)} (\Phi_\infty(S_\beta S^*_\beta)\Phi^\op_\infty(S_{\nu^\op} S^*_{\nu^\op}))(r(\beta))}{\sqrt{l + 1}}
        \frac{1}{\sqrt{l + 1}} \sum^l_{j=0} W_{\underline{\lambda}_j, 1} \otimes W_{\overline{\lambda}^{l-j}, 1} \otimes \delta_{s(\underline{\lambda}_j)} \\
    &= \frac{x_{r(\beta)} (\Phi_\infty(S_\beta S^*_\beta)\Phi^\op_\infty(S_{\nu^\op} S^*_{\nu^\op}))(r(\beta))}{l + 1}
        \sum^l_{j=0} \V\V^*(S_e \ox 1)(W_{\underline{\lambda}_j, 1} \otimes W_{\overline{\lambda}^{\op(l-j)}, 1} \otimes \delta_{s(\underline{\lambda}_j)}).\\
\intertext{Our calculation~\eqref{eq:VV*Seox1} of $\big(\V\V^*(S_e\ox 1)\big)(W_{\alpha,\beta}\ox W_{\mu^\op, \nu^\op}
\ox x)$ applies to each term in this sum, and we obtain}
\big(\V\V^*&(S_e\ox 1)\V\V^*\big)(W_{\alpha,\beta}\ox W_{\mu^\op, \nu^\op} \ox x)\\
    &= \frac{x_{r(\beta)} (\Phi_\infty(S_\beta S^*_\beta)\Phi^\op_\infty(S_{\nu^\op} S^*_{\nu^\op}))(r(\beta))}{l + 1}
        \sum^l_{j=0} \frac{1}{l+2}\sum^{l+1}_{k=0}(W_{\underline{e\lambda}_k, 1} \otimes W_{\overline{e\lambda}^{\op(l-k)}, 1} \otimes \delta_{s(\underline{e\lambda}_k)})\\
    &= \frac{x_{r(\beta)} (\Phi_\infty(S_\beta S^*_\beta)\Phi^\op_\infty(S_{\nu^\op} S^*_{\nu^\op}))(r(\beta))}{l + 2}
        \sum^{l+1}_{k=0}(W_{\underline{e\lambda}_k, 1} \otimes W_{\overline{e\lambda}^{\op(l-k)}, 1} \otimes \delta_{s(\underline{e\lambda}_k)})\\
    &= \big(\V\V^*(S_e\ox 1)\big)(W_{\alpha,\beta}\ox W_{\mu^\op, \nu^\op} \ox x).
\end{align*}
Hence
\[
\V\V^*(S_e \ox 1)(1- \V\V^*) (W_{\alpha,\beta} \ox W_{\mu^\op,\nu^\op} \ox x) = 0.
\]

Now suppose that $e\alpha = \beta$ and $\mu=\nu\tau$. Then $\big(\V\V^*(S_e\ox
1)\V\V^*\big)(W_{\alpha,\beta}\ox W_{\mu^\op, \nu^\op} \ox x) = 0$, whereas
\begin{align*}
\big(\V\V^*(S_e\ox 1)\big)&(W_{\alpha,\beta}\ox W_{\mu^\op, \nu^\op} \ox x)\\
    &= \V\V^*(W_{e\alpha,\beta}\ox W_{\mu^\op, \nu^\op} \ox x)\\
    &= \frac{x_{r(e)} (\Phi_\infty(S_{e\alpha} S^*_{e\alpha})\Phi^\op_\infty(S_{\nu^\op} S^*_{\nu^\op}))(r(e))}{\sqrt{|\tau| + 1}}
        \V \delta_{\tau} \\
    &= \frac{x_{r(e)} (\Phi_\infty(S_{e\alpha} S^*_{e\alpha})\Phi^\op_\infty(S_{\nu^\op} S^*_{\nu^\op}))(r(e))}{|\tau| + 1}
        \sum^{|\tau|}_{j=0} W_{\underline{\tau}_j,1} \otimes W_{\overline{\tau}^{\op(|\tau|-j)},1} \otimes \delta_{s(\underline{\tau}_j)}\\
    &= \frac{x_{r(e)} (\Phi_\infty(S_{\beta} S^*_{\beta})\Phi^\op_\infty(S_{\nu^\op} S^*_{\nu^\op}))(r(e))}{\sqrt{|\tau|+1}}
        \V\delta_{\tau}.
\end{align*}

For $\lambda \in G^*$, let $P_\lambda$ be the projection onto the 1-dimensional subspace
$\C \V\delta_\lambda$ of $\Xi \ox \Xi^\op \ox \ell^2(S, \phi)$. We will prove that
$\V\V^*(S_e \ox 1)(1- \V\V^*) = \lim_{l \to \infty} \sum_{\lambda \in G^l} P_\lambda
\V\V^*(S_e \ox 1)(1- \V\V^*)$. Since each $G^l$ is finite, this will complete the proof
that the commutator is a compact operator.

To prove the result, first observe that we always have $W_{\alpha,\beta} \ox
W_{\mu^\op,\nu^\op} \ox x = x(r(\beta)) W_{\alpha,\beta} \ox W_{\mu^\op,\nu^\op} \ox 1$,
where $1 \in \C G^0$ is the vector $1_v = 1$ for all $v$. Thus we can always assume that $x=1$.

Next, if $\V\V^*(S_e \ox 1)(1-\V\V^*)(W_{\alpha_1,\beta_1}\ox W_{\mu^\op_1, \nu^\op_1} \ox 1) \in \C \V
\delta_{\lambda_1}$ and $\V\V^*(S_e \ox 1)(1-\V\V^*)(W_{\alpha_2,\beta_2}\ox W_{\mu^\op_2, \nu^\op_2} \ox 1) \in \C
\V \delta_{\lambda_2}$ with $\lambda_1 \not= \lambda_2$, and both are nonzero, then
$e\alpha_i = \zeta_i\beta_i$ and $\mu_i = \nu_i\tau_i$
and we have $\zeta_i\tau_i =
\lambda_i$. So either $\zeta_1 \not= \zeta_2$ or $\tau_1 \not= \tau_2$, and we deduce
that $W_{\alpha_1,\beta_1}\ox W_{\mu^\op_1, \nu^\op_1} \ox 1 \perp
W_{\alpha_2,\beta_2}\ox W_{\mu^\op_2, \nu^\op_2} \ox 1$.

It follows that for any finite $F \subseteq G^*$, we have
\[
\Big\|\Big(1 - \sum_{\lambda \in F} P_\lambda\Big) \V\V^*(S_e \ox 1)(1- \V\V^*)\Big\|
    = \sup_{\lambda \not\in F} \|P_\lambda \V\V^*(S_e \ox 1)(1- \V\V^*)\|.
\]
So it suffices to show that for any $\varepsilon > 0$ there is a finite set $F \subseteq
G^*$ such that
\begin{equation}\label{eq:toshow}
\|P_\lambda \V\V^*(S_e \ox 1)(1- \V\V^*)\| < \varepsilon\quad\text{ for all }\quad\lambda \not\in F.
\end{equation}

For this, first note from our calculations above that $\V\V^*(S_e \ox 1)(1-
\V\V^*)(W_{\alpha,\beta}\ox W_{\mu^\op, \nu^\op} \ox 1)$ is either zero, or perpendicular to $\C
\V\delta_\lambda$ unless $e\alpha = \beta$ and $\mu = \nu\lambda$. So, defining
\[
H_{e,\lambda} := \overline{\linspan}\{W_{\alpha,e\alpha} \ox W_{(\nu\lambda)^\op,\nu^\op} \ox 1
    : r(\alpha) = s(e),\, s(\nu) = r(\lambda)=r(e)\},
\]
we have
\[
\|P_\lambda \V\V^*(S_e \ox 1)(1- \V\V^*)\|
    = \big\|\V\V^*(S_e \ox 1)(1 - \V\V^*)|_{H_{e,\lambda}} \big\|.
\]

The Cuntz--Krieger relation gives $W_{\alpha,e\alpha} = \sum_{\alpha' \in s(\alpha)E^k}
W_{\alpha\alpha', e\alpha\alpha'}$ for any $\alpha \in s(e) G^*$ and any $k$. Similarly,
each $W_{(\nu\lambda)^\op, \nu^\op} = \sum_{\nu' \in G^kr(\nu)} W_{(\nu'\nu\lambda)^\op,
(\nu'\nu)^\op}$ for any $k$. So putting
\begin{equation}\label{eq:Hlk spanners}
H_{e, \lambda,k} = \linspan\{W_{\alpha,e\alpha} \ox W_{(\nu\lambda)^\op, \nu^\op} \ox 1 : \,\alpha \in s(e)G^k,\ \nu \in G^k r(\lambda)\},
\end{equation}
we have a filtration $H_{e,\lambda} = \bigcup_k H_{e, \lambda, k}$. This implies that
$$
\big\|\V\V^*(S_e
\ox 1)(1 - \V\V^*)|_{H_{e,\lambda}} \big\|
= \sup_k \big\|\V\V^*(S_e \ox 1)(1 -\V\V^*)|_{H_{e, \lambda, k}} \big\|.
$$
The spanning elements~\eqref{eq:Hlk spanners} are mutually orthogonal and span
$H_{e,\lambda, k}$. Let $l := |\lambda|$. Using the calculations of $\big(\V\V^*(S_e\ox
1)(1-\V\V^*)\big)(W_{\alpha,\beta}\ox W_{\mu^\op, \nu^\op} \ox x)$ above, we see that for
a spanning element $W_{\alpha,e\alpha} \ox W_{(\nu\lambda)^\op, \nu^\op} \ox 1$ of
$H_{e,\lambda, k}$, we have
\[
\big(\V\V^*(S_e\ox 1)(1-\V\V^*)\big)(W_{\alpha,e\alpha} \ox W_{(\nu\lambda)^\op, \nu^\op} \ox 1)\\
    = \frac{(\Phi_\infty(S_{e\alpha} S^*_{e\alpha})\Phi^\op_\infty(S_{\nu^\op} S^*_{\nu^\op}))(r(e))}{\sqrt{l+1}}
        \V\delta_\lambda.
\]

We have
\begin{align*}
\|W_{\alpha,e\alpha}{} \ox W_{(\nu\lambda)^\op, \nu^\op} \ox 1\|^2
    &= \phi\big(\big(1 \mid (W_{\alpha,e\alpha} \ox W_{(\nu\lambda)^\op} \mid W_{\alpha,e\alpha} \ox W_{(\nu\lambda)^\op})_{A \ox A^\op}\big)\big)\\
    &= \phi\big(\Phi_\infty(S_{e\alpha} S^*_\alpha S_\alpha S^*_{e\alpha}) \cdot 1 \Phi^\op_\infty(S_{\nu^\op} S^*_{(\nu\lambda)^\op} S_{(\nu\lambda)^\op} S^*_{\nu^\op})\big)\\
    &= \sum_{v \in G^0} \Phi_\infty(S_{e\alpha} S^*_{e\alpha})(v) \Phi^\op_\infty(S_{\nu^\op} S^*_{\nu^\op})(v)\\
    &= \big(\Phi_\infty(S_{e\alpha} S^*_{e\alpha}) \Phi^\op_\infty(S_{\nu^\op} S^*_{\nu^\op})\big)(r(e)).
\end{align*}

For $\alpha \in s(e)G^k$ and $\nu \in G^k r(\lambda)$, let
\[
\kappa_{\alpha,\nu} := \sqrt{\big(\Phi_\infty(S_{e\alpha} S^*_{e\alpha}) \Phi^\op_\infty(S_{\nu^\op} S^*_{\nu^\op})\big)(r(e))}
\]
and
\[
h_{\alpha,\nu} := \kappa^{-1}_{\alpha,\nu} W_{\alpha,e\alpha} \ox W_{(\nu\lambda)^\op, \nu^\op} \ox 1.
\]
Then $\{h_{\alpha,\nu} : \alpha \in s(e)G^k,\, \nu \in G^k r(\lambda)\}$ is an orthonormal
basis for $H_{e, \lambda, k}$, and
\[
\big(\V\V^*(S_e\ox 1)(1-\V\V^*)\big) h_{\alpha,\nu} = \frac{\kappa_{\alpha,\nu}}{\sqrt{l+1}} \V\delta_\lambda
\]
for each $\alpha,\nu$. The Cuntz--Krieger relation in 
each of $C^*(G)$ and $C^*(G^\op)$ shows
that
\[
\sum_{\alpha,\nu}\kappa_{\alpha,\nu}^2
    =\! \sum_{\alpha,\nu} \big(\Phi_\infty(S_{e\alpha} S^*_{e\alpha}) \Phi^\op_\infty(S_{\nu^\op} S^*_{\nu^\op})\big)(r(e))
    \!=\! \big(\Phi_\infty(S_{e} S^*_{e}) \Phi^\op_\infty(P_{r(e)})\big)(r(e))
    \!=\!\Phi_\infty(S_{e} S^*_{e})(r(e)),
\]
where $P_{r(e)}\in \Oo_E$ is the projection
corresponding to the function $\delta_{r(e)}\in C(G^0)$.
Hence
\[
\big\|\big(\V\V^*(S_e\ox 1)(1-\V\V^*)\big)|_{H_{e,\lambda, k}}\big\|
    = \Big(\sum_{\alpha,\nu} \Big(\frac{\kappa_{\alpha,\nu}}{\sqrt{l+1}}\Big)^2\Big)^{1/2}
    = \Big(\frac{\Phi_\infty(S_{e} S^*_{e})(r(\lambda))}{l+1}\Big)^{1/2}.
\]

Now fix $\varepsilon > 0$ and choose $l$ large enough so that $\frac{1}{\sqrt{l+1}}<\varepsilon$. Let $F := \bigcup_{k
\le l} G^k$. Then $F$ is finite, and the calculations above show that
\[
\sup_{\lambda \not\in F} \|P_\lambda \V\V^*(S_e \ox 1)(1- \V\V^*)\|
    \le \Big(\frac{\big(\Phi_\infty(S_{e} S^*_{e}) \Phi^\op_\infty(s_{r(\lambda)})\big)(r(\lambda))}{l+1}\Big)^{1/2}
    < \varepsilon,
\]
which is~\eqref{eq:toshow}.
\end{proof}

\begin{thm}\label{thm:CK-Khom}
Let $G$ be a finite graph with no sources and no sinks
and let $E$ be the associated edge module. Assume that the
Watatani indices of $E^{\ox m}$ and $E^{\op\ox m}$ satisfy Assumption~\ref{ass:one}. Then
the class $\Delta$ is a $K$-homology fundamental class.
\end{thm}
\begin{proof}
Let $\mu=(\C^k\ox\C^k,\C^k,0)$ be the $K$-homology fundamental class for $\C^k$.
Proposition~\ref{prp:graph classes} shows that this class satisfies the hypothesis of
Theorem~\ref{thm:(OEop) sufficient}(1). So by that theorem, it suffices to check that
\[
    \iota_{A,\CPa_E}\ox_{\CPa_E}\Delta = \extocls\otimes_{ A^{\op}}\mu \quad\text{ and }\quad
    \iota_{A^\op,\CPa_{E^\op}}\ox_{\CPa_{E^\op}}\Delta = \extcls\ox_{A}\mu.
\]

First we compute $\extcls\ox_A\mu$. As described at~\eqref{eq:mu for CK}, $\mu$ is
represented by $(\C^k \ox \C^k, L^2(\C^k,\phi), 0)$ where $\phi : \C^k \to \C$ is given
by $\phi(a) = \sum_j a_j$. Since $\C^k\ox\C^k$ acts diagonally on $\C^k$, writing
$\Phi_\infty:\CPa_E\to A$ for the expectation of \cite[Proposition~3.5]{RRSext}, we have
$$
(\Xi_{E,A}\ox A^\op)\ox_{A\ox A^\op}\C^k\cong L^2(\CPa_E,\phi\circ\Phi_\infty).
$$

By \cite[Theorem~3.14]{RRSext}, there is an isometric inclusion
\[
Y : L^2(\Fock_E,\phi)\to L^2(\CPa_E, \phi)
\]
satisfying $Y(\mu) := [S_\mu]$. We write $Q := YY^* : L^2(\CPa_E, \phi) \to
Y(L^2(\Fock_E,\phi))$. Then $\extcls\ox_A\mu$ is represented by the Kasparov module
$(\CPa_E\ox A^\op,L^2(\CPa_E,\phi\circ\Phi_\infty),2Q-1)$.

For $t \in [0,\infty]$, define
$$
\mathbb{P}_t:=\frac{1}{1+t^2}\begin{pmatrix} \V\V^* & t\V(Y^*\ox 1)\\ t(Y\ox1)\V^* & t^2 Q\end{pmatrix}.
$$
A direct computation shows that each $\mathbb{P}_t$ is a projection, and that
$$
F_t := \left(\CPa_E\ox A^\op, \begin{pmatrix} L^2(\CPa_E)\ox L^2(\CPa_{E^\op})\\ L^2(\CPa_E)\end{pmatrix},
2\mathbb{P}_t-1\right)
$$
is a Fredholm module. At $t = 0$, this Fredholm module represents
$\iota_{A^\op,\CPa_{E^\op}}\ox_{\CPa_{E^\op}}\Delta$. At $t = \infty$, we have
$$
\left(\CPa_E\ox A^\op, \begin{pmatrix} L^2(\CPa_E)\ox L^2(\CPa_{E^\op})\\ L^2(\CPa_E)\end{pmatrix},
\begin{pmatrix} -1 & 0\\ 0 & 2Q-1\end{pmatrix}\right).
$$
Since $(\CPa_E \ox A^\op, L^2(\CPa_E) \ox L^2(\CPa_{E^\op}), -1)$ is a degenerate
Kasparov module, we deduce that
\[
[F_\infty] = \left[\CPa_E\ox A^\op,  L^2(\CPa_E),
 2Q-1\right]
\]
in $KK(\CPa_E \otimes A^\op, \C)$. Since $A \ox A^\op = \C^k\ox\C^k$ acts diagonally on
$L^2(A, \phi) =\C^k$, we deduce from \cite[Theorem~A.5]{CS} that
\begin{align*}
[\left(\CPa_E\ox A^\op,\Xi_E\ox A^\op, (2Q-1)\ox 1_{A^\op}\right)]&{}\ox_{A\ox A^\op}[(A\ox A^\op,A_\C,0)] \\
    & =[\left(\CPa_E\ox A^\op,  L^2(\CPa_E), 2Q-1\right)] = [F_\infty],
\end{align*}
so $F_\infty$ represents $\extcls \ox_A \mu$. Since the $F_t$ constitute a homotopy, this
shows that $\iota_{A^\op,\CPa_{E^\op}}\ox_{\CPa_{E^\op}}\Delta = \extcls\ox_{A}\mu$.

The same argument applied to the opposite graph shows that
$\iota_{A,\CPa_E}\ox_{\CPa_E}\Delta = \extocls\otimes_{ A^{\op}}\mu$, and this completes
the proof.
\end{proof}

\begin{example} \label{eg:CK-so-good}
Verifying Assumption~\ref{ass:one} in Theorem~\ref{thm:CK-Khom} in general seems
complicated, but it is  checkable in concrete instances. For example, $SU_q(2)$, realised
as a graph $C^*$-algebra as in \cite[Example~3.10]{RRSext}, is easily seen to satisfy
Assumption~\ref{ass:one}. Likewise, by \cite[Lemma~3.7 and Example~3.8]{RRSext},
Cuntz--Krieger modules associated to primitive non-negative matrices all satisfy
Assumption~\ref{ass:one}.
\end{example}

\begin{corl}
\label{cor:CK-PD} Let $G$ be a finite directed graph
with no sinks and no sources. Suppose that the edge modules of $G$ and
$G^\op$ both satisfy Assumption~\ref{ass:one}. Then $C^*(G)$ is Poincar\'e dual to the
graph algebra of the opposite graph $C^*(G^\op)$. Thus there are isomorphisms
$$
K_*(C^*(G))\to K^{*+1}(C^*(G^\op)),\quad
K_*(C^*(G^\op))\to K^{*+1}(C^*(G)).
$$
\end{corl}

We require no assumptions on the associated shift space here. From the
characterisation of the $K$-theory of graph algebras \cite[Theorem~7.18]{Raeburn} we
obtain

\begin{corl}
\label{cor:graph-K-hom} Suppose the 
finite directed graphs
$G=(G^0,G^1,r,s)$ and $G^\op$ have no sinks and no sources, and that their
edge modules  satisfy 
Assumption~\ref{ass:one}. Then the even
$K$-homology groups of $C^*(G)$ and $C^*(G^\op)$ are torsion-free and have the same rank
as the corresponding odd $K$-theory groups.
\end{corl}

\subsubsection{The $K$-homology fundamental class 
for $C^*(G)^\op$}
\label{subsub:other-op}

It turns out that proving that $C^*(G)^\op$ is a Poincar\'{e} dual for 
$C^*(G)$ requires more assumptions than for $C^*(G^\op)$.

We let $A=C(G^0)$ and $E=C(G^1)$ be the algebra and module for a finite directed graph
$G$ with no sources. We let $\phi:A\to\C$ be the state given by $\phi(a)=\sum_{v\in G^0}a(v)$.

Suppose that the bi-Hilbertian $A$-bimodule $E$ satisfies Assumption~\ref{ass:one}, and additionally that
for all $\nu\in E^{\ox k}$ there is $c_{|\nu|}\in A$ (necessarily central) which is
invertible and such that
\begin{equation}
{\bf q}(\nu):=\lim_{n\to\infty}\mathrm{e}^{-\beta_n}\nu \mathrm{e}^{\beta_{n-k}}=c_{|\nu|}\nu.
\label{eq:super-strong}
\end{equation}
This  assumption is a special case of a key assumption in \cite{GMRkappa}; see
\cite{GMRkappa} for more detail. This condition is not satisfied for all graph algebras.
For example, the bimodule of the graph with vertices $v,w$ and edges $e,f,g$ satisfying
$s(e) = r(e) = s(f) = v$ and $r(f) = s(g) = r(g) = w$, whose $C^*$-algebra is $SU_q(2)$,
does not satisfy this hypothesis. On the other hand, the graph-modules associated to
Cuntz--Krieger algebras of primitive $0$--$1$ matrices all do,
\cite[Example~3.8]{RRSext}.

We can now produce a fundamental $K$-homology class
$$
\overline{\Delta}\in KK^1(\CPa_E\ox\CPa_E^\op,\C)
=KK^1(\CPa_E\ox\CPa_{\ol{E}^\op},\C).
$$
The process is largely the same as for $\CPa_{E^\op}$, so we just describe it briefly. We
first claim that there is an isometry
$$
\ol{\V}:\,L^2(\Fock_E,\phi)\hookrightarrow \Xi_E\ox\Xi_{\ol{E}^\op}\ox_{A\ox A^\op}L^2(A,\phi)
$$
such that for every elementary tensor $\lambda \in E^{\ox |\lambda|}$, we have
$$
\ol{\V}(\lambda)=\sum_{j=0}^{|\lambda|}\frac{1}{\sqrt{|\lambda|+1}}
W_{\lambda_j,1}\ox W_{1,\ol{c^{-1/2}_{|\lambda|}\lambda}^{\op(|\lambda|-j)}}\ox \delta_{s(\lambda_j)}.
$$
Here $\ol{\lambda}^\op$ refers to the element of the conjugate module $\ol{E}^\op$
corresponding to $\lambda$. We write $\ol{\lambda}^{\op(|\lambda|-j)}$ for the element of
the conjugate module corresponding to the tail of $\lambda$ of length $|\lambda|-j$. By
Lemma~\ref{lem:half-a-bar} we see that here we do {\em not} reverse the order of elements
in a tensor product when we factor a simple tensor: that is, $\delta_{\ol{\lambda}^{\op}}
= \delta_{\ol{\lambda}^\op_{j}} \ox \delta_{\ol{\lambda}^{\op(|\lambda|-j)}}$ under the
identification of $(\ol{E}^\op)^{\ox |\lambda|}$ with $(\ol{E}^\op)^{\ox j} \ox
(\ol{E}^\op)^{\ox |\lambda| - j}$. The need for the additional
assumption~\eqref{eq:super-strong} becomes clear in the following computation, which uses
\cite[Lemmas 3.2~and~3.3]{RRSext}:
\begin{align*}
(W_{1,\ol{c^{-1/2}_{|\lambda|}\lambda}^\op}\mid W_{1,\ol{c^{-1/2}_{|\rho|}\rho}^\op})_{A^\op}
&={}_{A^\op}(\ol{c^{-1/2}_{|\lambda|}\lambda}^\op\mid {\bf q}(\ol{c^{-1/2}_{|\rho|}\rho})^\op)
={}_A({\bf q}(c^{-1/2}_{|\rho|}\rho)\mid c^{-1/2}_{|\lambda|}\lambda)^\op\\
&=\delta_{|\rho|,|\lambda|}{}_A({\bf q}(c^{-1/2}_{|\lambda|}c^{-1/2}_{|\rho|}\rho)\mid \lambda)^\op={}_A(\rho\mid \lambda)^\op=(\lambda^\op\mid\rho^\op)_{A^\op}.
\end{align*}
With this calculation in place, the remainder of the computations used to prove that
$\ol{\V}$ is an isometry are identical to those for the $E^\op$ case. The formula for the
adjoint is slightly modified, becoming
\[
\ol{\V}^* (W_{\alpha,\beta}\ox W_{\ol{\mu}^\op,\ol{\nu}^\op}\ox x)
    = x_{r(\beta)}\Phi_\infty(S_{\ol{\alpha}}S_\beta^*)(r(\beta))\Phi_\infty(S_{\ol{\nu}}S_\mu^*)(r(\beta))\,\delta_{\underline{\alpha},\underline{\sigma}}\,c_{|\underline{\sigma}|}^{1/2}.
\]
The additional factor of $c_{|\underline{\sigma}|}^{1/2}$ does not alter any subsequent
computations because it is central and invertible. For computing commutators of
$\Oo_E\ox\Oo_{\ol{E}^\op}$ with $\ol{\V}\ol{\V}^*$, the computations are mostly as for
the case of $E^\op$, with only minor differences. We  summarise the results as follows:

\begin{thm}
\label{thm:ee-op}
Let $G$ be a finite directed graph with no sources,
$A=C(G^0)$ and $E=C(G^1)$. Assume that $E$ satisfies Assumption~\ref{ass:one} and Equation
\eqref{eq:super-strong}. Then there is a
$K$-homology fundamental class
$$
\ol{\Delta}\in KK^1(\CPa_E\ox\CPa_E^\op,\C).
$$
Hence there are isomorphisms
$$
\cdot\ox_{\CPa_E}\ol{\Delta} : K_*(\CPa_E)\stackrel{\cong}{\to} K^{*+1}(\CPa_E^\op)
    \quad\mbox{and}\quad
\cdot\ox_{\CPa_E^\op}\ol{\Delta}:K_*(\CPa_E^\op)\stackrel{\cong}{\to} K^{*+1}(\CPa_E).
$$
\end{thm}

\subsubsection{A comparison of the dual algebras}
\label{sub:two-is-too-many}

We now discuss what happens when we can construct $\delta\in
KK^{d+1}(\C,\mathcal{O}_E\ox\mathcal{O}_{E^\op})$ and $\ol{\Delta}\in
KK^{d+1}(\mathcal{O}_E\ox\mathcal{O}_{E}^\op,\C)$ (or $\ol{\delta}$ and $\Delta$). Our
methods guarantee that the Kasparov product with these classes gives us isomorphisms
$$
K^*(\CPa_{E^\op})\stackrel{\delta\ox\cdot}{\longrightarrow}K_{*+d+1}(\CPa_E)
\stackrel{\cdot\ox\ol{\Delta}}{\longrightarrow}K^*(\CPa_E^\op)
$$
and
$$
K_*(\CPa_E^\op)\stackrel{\cdot\ox\ol{\Delta}}{\longrightarrow}K^{*+d+1}(\CPa_E)
\stackrel{\delta\ox\cdot}{\longrightarrow}K_*(\CPa_{E^\op}).
$$
Hence when $A$, and so $A^\op$, $\CPa_{E^\op}$ and $\CPa_E^\op$ are in the bootstrap
class, we find that $\CPa_{E^\op}$ and $\CPa_E^\op$ are $KK$-equivalent. Thus the two
potential dual algebras are indistinguishable at the level of $KK$-theory. Indeed if, for
example, $\CPa_E^\op$ and $\CPa_{E^\op}$ are UCT Kirchberg algebras, then they coincide.

\begin{thm}
\label{thm:KK-equiv-ops} Let $A$ be unital and Poincar\'e self-dual with fundamental
classes $\beta$ and $\mu$. If $E$ is a bi-Hilbertian bimodule with finite left and right
Watatani indices such that 
$\beta\ox_A[E]=\beta\ox_{A^\op}[E^\op]$ and
$[E]\ox_A\mu=[\ol{E}^\op]\ox_{A^\op}\mu$ then $\CPa_{E^\op}$ and $\CPa_E^\op$ are
$KK$-equivalent. Hence $\CPa_E$ is Poincar\'e dual to both $\CPa_{E^\op}$ and
$\CPa_E^\op$.
\end{thm}

\begin{corl}
\label{cor:graph-self-dual} Suppose that the finite
graph $G$ has no sinks and no sources,
and that the edge modules $E$ and $E^\op$ 
 satisfy Assumption~\ref{ass:one} and Equation~\eqref{eq:super-strong}. Then
$C^*(G^\op)$ and $C^*(G)^\op$ are $KK$-equivalent and both are Poincar\'e dual to
$C^*(G)$.
\end{corl}

\appendix

\section{General Poincar\'e duality}
\label{sec:Magnus-made-me-do-it}

We briefly comment on a more general situation than that of lifting a Poincar\'e
self-duality for $A$ to one for $\Oo_A$. We thank Magnus Goffeng for pointing out the
applications of our approach to this setting. Let $A$, $B$ be Poincar\'e dual
unital $C^*$-algebras, with the duality realised by classes $\mu\in KK^d(A\ox B,\C)$ and
$\beta\in KK^d(A\ox B,\C)$. Suppose that $E$ is an $A$--$A$-correspondence and $F$ is a
$B$--$B$-correspondence.

By considering diagrams like that in Lemma~\ref{lem:diagram-commutes} with $B$ in place
of $A^\op$, and applying arguments like those in Section~\ref{sec:funky-diagram}, we
arrive at the following statement.

\begin{thm}
\label{thm:Aopp-Bee} With $A$, $B$, $E$ and $F$ as above, suppose that
$[E]\ox_A\mu=[F]\ox_B\mu \in KK(A \otimes B, \C)$ and $\beta\ox_A[E]=\beta\ox_B[F] \in
KK(C, A \otimes B)$. Then classes $\Delta\in KK^{d+1}(\CPa_E\ox\CPa_F,\C)$ and
$\delta\in KK^{d+1}(\C,\CPa_E\ox\CPa_F)$ define isomorphisms as in
Lemma~\ref{lem:diagram-commutes} if
$$
\iota_{A,\CPa_E}\ox_{\CPa_E}\Delta={\rm ext}_B\ox_B\mu
\quad\mbox{and}\quad
\iota_{B,\CPa_F}\ox_{\CPa_F}\Delta={\rm ext}_A\ox_A\mu
$$
and
$$
-\delta\ox_{\CPa_E}{\rm ext}_A=\beta\ox_B\iota_{B,\CPa_F}
\quad\mbox{and}\quad
\delta\ox_{\CPa_F}{\rm ext}_B=\beta\ox_A\iota_{A,\CPa_E}.
$$
\end{thm}

The same arguments as in Theorem~\ref{thm:yeah} allow us to obtain a $K$-theory
fundamental class for $\CPa_E\ox\CPa_F$.
\begin{thm}
\label{thm:K-class-general} Let $A$ and $B$ be unital $C^*$-algebras. Suppose that $A$
and $B$ are Poincar\'e dual with invariant classes $\beta,\,\mu$ as in
Theorem~\ref{thm:Aopp-Bee}, and assume that $E$ and $F$ are both finitely generated as
right modules. Construct $\mathbb{W}_A$ from $E_A$ and $\mathbb{W}_B$ from $F_B$ as in
Definition~\ref{defn:K-fun} and Lemma~\ref{lem:eval=0}. Then
$$
\delta := \beta\ox_{A\ox B}\mathbb{W}_A - \beta\ox_{A\ox B}\mathbb{W}_B\in KK(\C,A\ox B)
$$
is a $K$-theory fundamental class for $\Oo_E\ox \Oo_F$.
\end{thm}

Obtaining $K$-homology fundamental classes is harder, but we mention one
important case.

\begin{example}\label{ex:riem-mfld}
Let $(M,g)$ be a compact oriented Riemannian manifold of dimension $d$ and $\alpha$ the
automorphism of $C^\infty(M)$ dual to an orientation-preserving isometry.

Insisting on an actual isometry, as opposed to an almost-isometry, ensures that we obtain
an automorphism $\tilde{\alpha}$ of the bundle of Clifford algebras $\Cliff(M,g)$. Thus
we obtain correspondences $E={}_\alpha C(M)_{C(M)}$ and
$F={}_{\tilde{\alpha}}\Cliff(M,g)_{\Cliff(M,g)}$.

The orientable version of Kasparov's Bott class for $M$ is as we described in the proof
of Proposition~\ref{prop:spin-cee-hyp}, except the spinor bundle is replaced by the
Clifford bundle. Thus $\beta$ has a representative $(\C,X_{C(M)\ox\Cliff(M,g)},T)$ with
$\beta\in KK^d(\C,C(M)\ox\Cliff(M,g))$. Just as in the proof of
Proposition~\ref{prop:spin-cee-hyp}, we can implement $\alpha$ and $\tilde{\alpha}$ via a
$\C$-linear map $V:X\to X$ satisfying conditions analogous to those of part~1 of
Lemma~\ref{lem:CP-conditions}. In particular, we can show that
$\beta\ox_{C(M)}[E]=\beta\ox_{\Cliff(M,g)}[F]$, and so we can obtain a $K$-theory
fundamental class $\delta\in KK^{d+1}(\C,C(M)\ox\Cliff(M,g))$. As in the spin$^c$ case,
we restrict to compact manifolds to obtain the $K$-theory class.

We can also produce the $K$-homology fundamental class, and this does not require compactness, though our formulation of Poincar\'e duality does require compactness. Kasparov's
fundamental class \cite{KasparovEqvar,LRV} for the oriented manifold $(M,g)$ is
$$
\lambda=
\left[\left(C^\infty(M)\ox\Cliff(M,g), {}_{\pi}L^2(\Lambda_+T^*M\oplus\Lambda_-T^*M),
\begin{pmatrix} 0 & (d+d^*)_-\\
(d+d^*)_+ & 0 \end{pmatrix}\right)\right],
$$
where $\pi$ is the  representation defined for $f\in C(M)$ and $v\in \Gamma(T^*M)$ by
$$
\pi(f\ox v)\omega(x)=f(x)(v(x)\wedge \omega(x)+v(x)\llcorner\omega(x)),\quad\omega\in L^2(\Lambda^*T^*M).
$$
Just as for the Bott class, we can implement $\alpha$ and $\tilde{\alpha}$ on
$L^2(\Lambda^*T^*M)$ via a $\C$-linear map $W$. One checks that $W$ and $\lambda$ satisfy
analogues of the conditions of part~2 of Lemma~\ref{lem:CP-conditions}. In particular, we
can show that $[E]\ox_{C(M)}\lambda=[F]\ox_{\Cliff(M,g)}\lambda$. So mildly modifying the
constructions of Section~\ref{subsec:cpi}, we obtain a fundamental class implementing
duality between the crossed product of the functions and the crossed product of the
Clifford algebra,
$$
\Delta=
\left[\left((C^\infty(M)\rtimes_\alpha\Z)\ox(\Cliff(M,g)\rtimes_\alpha\Z), {}_{\tilde{\pi}}\ell^2(\Z)\ox L^2(\Lambda T^*M),
\begin{pmatrix} N & (d+d^*)_-\\
(d+d^*)_+ & -N \end{pmatrix}\right)\right],
$$
where the representation $\tilde{\pi}$ is defined analogously
to Equation \eqref{eq:two-ways}.
See \cite{KasparovEqvar} and \cite{LRV} for
more information about Kasparov's fundamental class.
\end{example}

\section{Relationships between the extension classes}
\label{sec:one-more-op}

If $E$ is not an invertible bimodule, then constructing a representative of the extension
class is more complicated than for an invertible module \cite[Theorem~3.14]{RRSext}. As a
result, the relationship between the extension class for $\CPa_E$ and that for
$\CPa_{E}^\op$ is also more complicated than in Section~\ref{subsec:ee-smeb}, as we now
explain. Throughout this section we assume that $A$ is unital, and that $E$ is finitely
generated and bi-Hilbertian and satisfies Assumption~\ref{ass:one} of subsection
\ref{subsub:GMI}.

Just as for $K$-theory where we compared $K_*(A)$ and $K_*(A^\op)$ in
Lemma~\ref{lem:K-op}, we can also produce an explicit isomorphism for $KK$-groups of
algebras and their opposites. This will allow us to compare the classes $\extcls$ and
$\extobarcls$. (When $E$ is not an invertible bimodule, there can be no corresponding
comparison with $\extocls$.)

\begin{prop}
\label{prop:kay-op-kay}
For any $C^*$-algebras $A,\,B$ there is an isomorphism
$$
{\rm OP}:KK(A,B)\stackrel{\cong}{\longrightarrow} KK(A^\op,B^\op)
$$
given on cycles by the map
$$
(A,X_B,T)\mapsto (A^\op,\ol{X}_{B^\op},\ol{T}).
$$
Here $\ol{X}_{B^\op}$ is the (left $B$-)conjugate module
considered
as a right $B^\op$ module,
and for $a\in A$, $x\in X$ we define
$a^\op\ol{x}=\ol{a^*x}$ and
$\ol{T}\ol{x}=\ol{Tx}$.
\end{prop}
\begin{proof}
Given a Kasparov module $(A,{}_\phi X_B,T)$ (bounded or unbounded), the data
$(A^\op,\ol{X}_{B^\op},\ol{T})$ defines a Kasparov $A^\op$--$B^\op$-module. To see that
$(A,{}_\phi X_B,T) \mapsto (A^\op,\ol{X}_{B^\op},\ol{T})$ descends to a well-defined map
from $KK(A, B)$ to $KK(A^\op, B^\op)$, observe that if $(A, Y_{B \ox C([0,1])},
\mathbf{T})$ is a homotopy of Kasparov $A$--$B$-modules, then $(A^\op, \ol{Y}_{(B \ox
C([0,1]))^\op}, \ol{\mathbf{T}})$ is a homotopy of the corresponding Kasparov
$A^\op$--$B^\op$-modules. Since this entire discussion is symmetric in $A$--$B$ and
$A^\op$--$B^\op$, we are done.
\end{proof}

\begin{prop}
Let $A$ be unital, and $E$ a finitely generated bi-Hilbertian $A$ bimodule satisfying
Assumption~\ref{ass:one}. Let $(\Oo_E,\Xi_{E,A},2P-1)$ be the representative of $\extcls$
described in Equation~\ref{eq:ext-Xi}, let $(\Oo_E^\op,\ol{\Xi_{E,A}}_{A^\op},2\ol{P}-1)$
be the class provided by Proposition~\ref{prop:kay-op-kay}, and let 
$(\Oo_E^\op,\Xi_{\ol{E}^\op,A^\op},2P^\op-1)$ be the representative of $\extobarcls$. Then the map
$$
S_{\ol{\mu}^\op}S_{\ol{\nu}^\op}^*\mapsto \ol{S_\mu S_\nu^*}^\op
$$
extends to a unitary isomorphism of Kasparov modules
$$
(\Oo_E^\op,\Xi_{\ol{E}^\op,A^\op},2P^\op-1)\cong(\Oo_E^\op,\ol{\Xi_{E,A}}_{A^\op},2\ol{P}-1).
$$
Hence under the isomorphism $KK(\Oo_E,A)\to KK(\Oo_E^\op,A^\op)$ of
Proposition~\ref{prop:kay-op-kay}, the class $\extcls$ is mapped to the class
$\extobarcls$.
\end{prop}
\begin{proof}
Fix elementary tensors $\mu,\nu,\rho, \sigma \in \Fock_E$. Assumption~\ref{ass:one}
provides a positive adjointable (for both inner products) map $q:\Fock_E\to \Fock_E$
given by $q(\mu)=\lim_{n\to\infty}\mathrm{e}^{-\beta_n}\nu \mathrm{e}^{\beta_{n-|\mu|}}$.

We write $\mu=\mu_{in} \ox \mu_f$ where $\mu_{in}$ is an initial tensor factor of $\mu$
whose length will be clear from context, and $\mu_f$ the corresponding final tensor
factor. Write $W_{\ol{\mu}^\op,\ol{\nu}^\op}$ for the image of
$S_{\ol{\mu}^\op}S^*_{\ol{\nu}^\op}$ in the completion $\Xi_{\ol{E}^\op,A^\op}$ and
$W_{\mu,\nu}$ for the image of $S_{\mu,\nu}$ in $\Xi_{E,A}$. We have
\begin{align*}
(W_{\ol{\mu}^\op,\ol{\nu}^\op} \mid W_{\ol{\rho}^\op,\ol{\sigma}^\op})_{A^\op}
&=\Phi_\infty(S_{\ol{\nu}^\op}S_{\ol{\mu}^\op}^*S_{\ol{\rho}^\op}S_{\ol{\sigma}^\op}^*)\\
&=\left\{\begin{array}{ll}{}_{A^\op}(\ol{\nu}^\op(\ol{\mu}^\op \mid \ol{\rho_{in}}^\op)_{A^\op}\ol{\rho_f}^\op \mid \ol{q}^\op(\ol{\sigma}^\op)) & |\rho|\geq |\mu|\\
{}_{A^\op}(\ol{\nu}^\op \mid \ol{q}^\op(\ol{\sigma}^\op( \ol{\rho}^\op\mid \ol{\mu_{in}}^\op )_{A^\op})\ol{\mu_f}^\op) & |\rho|\leq |\mu|\end{array}\right.\\
&=\left\{\begin{array}{ll}{}_{A^\op}(\ol{\nu}^\op{}_A(\ol{\rho_{in}} \mid \ol{\mu})^\op \ol{\rho_f}^\op\mid \ol{q}^\op(\ol{\sigma}^\op)) & |\rho|\geq |\mu|\\
{}_{A^\op}(\ol{\nu}^\op \mid \ol{q}^\op(\ol{\sigma}^\op{}_A(\ol{\mu_{in}}\mid\ol{\rho})^\op\ol{\mu_f}^\op)) & |\rho|\leq |\mu|\end{array}\right.\\
&=\left\{\begin{array}{ll}{}_{A^\op}(\ol{\nu(\mu \mid \rho_{in})_A\rho_f}^\op \mid \ol{q(\sigma)}^\op)) & |\rho|\geq |\mu|\\
{}_{A^\op}(\ol{\nu}^\op \mid \ol{q(\sigma(\rho\mid\mu_{in})_A\mu_f)}^\op)) & |\rho|\leq |\mu|\end{array}\right.\\
&=\left\{\begin{array}{ll}{}_A(q(\sigma) \mid \nu(\mu \mid \rho_{in})_A\rho_f)^\op& |\rho|\geq |\mu|\\
{}_A(q(\sigma(\rho\mid\mu_{in})_A\mu_f)\mid\nu)^\op&|\rho|\leq|\mu|.\end{array}\right.
\end{align*}
Likewise in the conjugate module $\ol{\Xi_{E,A}}_{A^\op}$
we have
\begin{align*}
(\ol{W_{\mu,\nu}}^\op \mid \ol{W_{\rho,\sigma}}^\op)_{A^\op}
&={}_A(\ol{W_{\rho,\sigma}}\mid \ol{W_{\mu,\nu}})^\op
=(W_{\rho,\sigma}\mid W_{\mu,\nu})_A^\op
=\Phi_\infty(S_\sigma S_\rho^* S_\mu S_\nu^*)\\
&=\begin{cases}
    {}_A(\sigma\mid q(\nu(\mu\mid \rho_{in})_A\rho_f))^\op & \text{ if $|\rho|\geq |\mu|$}\\
    {}_A(\sigma(\rho\mid\mu_{in})_A\mu_f\mid q(\nu))^\op & \text{ if $|\rho|\leq |\mu|$.}
\end{cases}
\end{align*}
The self-adjointness of $q$ and sesquilinearity of the inner product show that there is
an isometry
$$
U:\Xi_{\ol{E}^\op,A^\op}\to \ol{\Xi_{E,A}}_{A^\op}
    \quad\text{ such that }\quad
UW_{\ol{\mu}^\op, \overline{\nu}^\op} = \ol{W_{\mu,\nu}}^\op.
$$
This $U$ carries the embedded image $\ol{\linspan}\{W_{\ol{\mu}^\op,1}:\,\mu\in\Fock_E\}$
of $\Fock_{\ol{E}^\op}$ to the embedded image
$\ol{\linspan}\{\ol{W_{\mu,1}}:\,\mu\in\Fock_E\}$ of $\ol{\Fock_{E}}^\op$. Hence
$$
UP^\op U^*=\ol{P}.
$$
Likewise,
$$
US_{\ol{e}^\op}W_{\ol{\rho}^\op,\ol{\sigma}^\op}
=UW_{\ol{e\rho}^\op,\ol{\sigma}^\op}
=\ol{W_{e\rho,\sigma}}^\op
=S_e^{*\op} UW_{\ol{\rho}^\op,\ol{\sigma}^\op}.
$$
and we deduce that the actions are also intertwined.
\end{proof}

\end{document}